\documentclass[default]{sn-jnl}

\usepackage{graphicx}%
\usepackage{multirow}%
\usepackage{amsmath,amssymb,amsfonts}%
\usepackage{amsthm}%
\usepackage{mathrsfs}%
\usepackage[title]{appendix}%
\usepackage{xcolor}%
\usepackage{textcomp}%
\usepackage{manyfoot}%
\usepackage{booktabs}%
\usepackage{algorithm}%
\usepackage{algorithmicx}%
\usepackage{algpseudocode}%
\usepackage{listings}%

\usepackage{tikz} 
\usetikzlibrary{positioning, shapes.geometric} 
\newtheorem{Def}{Definition}[section]
\newtheorem{lem}[Def]{Lemma}
\newtheorem{tho}[Def]{Theorem}
\newtheorem{exam}[Def]{Example}
\newtheorem{coro}[Def]{Corollary}
\newtheorem{hyp}{Hypothesis}
\newtheorem{prop}[Def]{Proposition}
\newtheorem{rem}[Def]{Remark}
\newtheorem{assumptionex}{Assumption}
\newenvironment{asp}
  {\pushQED{\qed}\assumptionex}
  {\popQED\endassumptionex}

\newcommand{\ud}{\mathrm d}
\newcommand{\E}{\mathbb E}
\newcommand{\R}{\mathbb{R}}

\allowdisplaybreaks \allowdisplaybreaks[4]

\theoremstyle{thmstyleone}%
\theoremstyle{thmstyletwo}%

\theoremstyle{thmstylethree}%

\begin{document}

\title[Article Title]{
Asymptotic-preserving approximations  for  stochastic incompressible viscous fluids and SPDEs on graph
}


\author[1]{\fnm{Jianbo} \sur{Cui}}\email{jianbo.cui@polyu.edu.hk}

\author*[1]{\fnm{Derui} \sur{Sheng}}\email{derui.sheng@polyu.edu.hk}


\affil[1]{\orgdiv{Department of Applied Mathematics}, \orgname{The Hong Kong Polytechnic University}, \orgaddress{\street{Hung Hom}, 
\state{Kowloon}, \country{Hong Kong}}}

%

\abstract{
The long-term dynamics of particles  
involved in an incompressible flow with a small viscosity ($\epsilon>0$) and slow chemical reactions, is depicted by a class of stochastic reaction-diffusion-advection (RDA) equations with a fast advection term of magnitude $1/\epsilon$. 
It has been shown in 
\cite{CF19} the fast advection asymptotics of stochastic RDA equation in $\R^2$ can be characterized through a stochastic partial differential equation (SPDE) on the graph associated with certain Hamiltonian. 
To simulate such fast advection asymptotics,
 we introduce and study an asymptotic-preserving (AP) exponential Euler approximation for the multiscale stochastic RDA equation. There are three key ingredients in proving asymptotic-preserving property of the proposed approximation. First,  a strong error estimate, which depends on $1/\epsilon$ linearly, is obtained via a variational argument. Second, 
 we prove the consistency of exponential Euler approximations on the fast advection asymptotics between the original problem and the SPDE on graph. Last, a graph weighted space is introduced to quantify the approximation error for SPDE on graph, which avoids the possible singularity near the vertices. Numerical experiments are carried out to support the theoretical results.
}

\pacs[MSC Classification]{60H15, 60H35, 35R02}



\keywords{stochastic RDA
equations, SPDE on graph, exponential Euler approximation, fast advection asymptotics, asymptotic-preserving property}

\maketitle
\section{Introduction}
Given a certain Hamiltonian function (as called the stream function)  $H:\R^2\to \R$, which describes the flow pattern of a fluid or gas in a two-dimensional space, the density of particles moving along with an incompressible flow in $\R^2$ adheres to the Liouville equation, $$\partial_t u(t,x)=\langle\nabla^{\perp} H(x),\nabla u(t,x)\rangle,$$ where $\nabla^\perp H(x)=(-\partial_2 H(x),\partial_1 H(x))^\top$ for any point $x\in\R^2$.
In cases where this incompressible flow exhibits a small viscosity of magnitude $\epsilon>0$ and the particles are involved in a slow chemical reaction, which includes both a deterministic and a stochastic component, the particle density $\tilde{u}_{\epsilon}(t,x)$ can be described by the following equation (as referenced in \cite{CF19}),
\begin{equation*}
\partial_t \tilde{u}_{\epsilon}(t,x)=\frac{\epsilon}{2}\Delta \tilde{u}_{\epsilon}(t,x)+\langle\nabla^{\perp} H(x),\nabla\tilde{u}_{\epsilon}(t,x)\rangle
 +\epsilon b(\tilde{u}_{\epsilon}(t,x))+\sqrt{\epsilon}g(\tilde{u}_{\epsilon}(t,x))\partial_t\tilde{\mathcal{W}}(t,x).
\end{equation*}
Here $\tilde{\mathcal{W}}(t,x)$ represents a spatially homogeneous Wiener process (see section \ref{S:SRDA} for more details). This equation describes how the density of particles changes over time, taking into account the effects of viscosity, the flow pattern, and the chemical reaction. Suppose that $b,g:\R\to\R$ are globally Lipschitz continuous, and the Hamiltonian $H$ grows essentially quadratically at infinity.
As $\epsilon$ tends to zero, on any finite time interval $[0,\mathsf T]$, 
the function $\tilde{u}_{\epsilon}(t,x)$ is anticipated to converge in probability to $u(t,x)$. 
In order to study the asymptotical behavior of $\tilde{u}_{\epsilon}(t,x)$ over the large time interval $[0,\mathsf T/\epsilon]$ ($0<\epsilon\ll1$), with a change of time, we notice that $u_\epsilon(t,x)=\tilde u_\epsilon(t/\epsilon,x)$ satisfies
the following stochastic reaction-diffusion-advection (RDA) equation
\begin{equation}\label{eq:SPDE}
\left\{
\begin{split}
\partial_t u_{\epsilon}(t,x)&=\mathcal L_\epsilon u_{\epsilon}(t,x)
 +b(u_{\epsilon}(t,x))+g(u_{\epsilon}(t,x))\partial_t\mathcal{W}(t,x), \quad t\in[0,\mathsf T],\\
 u_{\epsilon}(0,x)&=\psi(x),\quad x\in\R^2,
 \end{split}
 \right.
\end{equation}
where $\mathcal L_\epsilon=\frac{1}{2} \Delta +\frac{1}{\epsilon}\nabla^{\perp} H(x)\cdot \nabla$ and $\mathcal W(t,x)=\sqrt{\epsilon}\tilde{\mathcal{W}}(t/\epsilon,x)$.

By assuming that the derivative of the period of the motion on level sets of $H$ does not vanish (see \eqref{eq:Tneq0}),
the authors in \cite{CF19} characterize the limit of $u_\epsilon$ as  $\epsilon\to 0$, by a stochastic partial differential equation (SPDE) defined on a graph $\Gamma$ associated with the Hamiltonian $H$.
More precisely, $\Gamma:=\Pi(\R^2)$ is the image of
the projection $\Pi$ on $\R^2$ identifying all points on the same connected component of each level set of $H$ (as detailed in section \ref{S2.1}). For a class of graph weighted functions $\gamma:\Gamma\to \R$ (see Assumption \ref{Asp:gamma}),
\cite[Theorem 5.3]{CF19} shows that
 for every $p\ge1$ and $0< \tau_0<\mathsf T$, 
 \begin{equation}\label{eq:u-u}
\lim_{\epsilon\to0}\E\bigg[\sup_{t\in[\tau_0,\mathsf T]}\|u_\epsilon(t)-\bar{u}(t)\circ \Pi\|_{\mathbb{H}_\gamma}^p\bigg]=0
\end{equation}
with $\mathbb{H}_\gamma:=L^2(\R^2,\gamma(\Pi(x))\ud x)$. 
The limiting process $\bar{u}$  is governed by the SPDE on $\Gamma$, 
\begin{equation}\label{eq:utzk}
\left\{
\begin{split}
\partial_t \bar{u}(t, z, k)&=\bar{\mathcal{L}} \bar{u}(t, z, k)+b(\bar{u}(t, z, k))+g(\bar{u}(t, z, k)) \partial_t \bar{\mathcal{W}}(t, z, k), \quad t\in[0,\mathsf T], \\
\bar{u}(0, z, k)&=\psi^{\wedge}(z, k), \quad(z, k) \in \Gamma;
\end{split}
\right.
\end{equation}
see section \ref{S2} for more details about the operators $^\wedge$ and $\mathcal L$, as well as the random field $\bar{\mathcal W}$.
The solution $u_\epsilon$ (resp. $\bar{u}$) of \eqref{eq:SPDE} (resp. \eqref{eq:utzk}) belongs to $\cap_{p\ge1}L^p(\Omega; \mathcal C([0,T];\mathbb{H}_\gamma))$ (resp. $\cap_{p\ge1}L^p(\Omega; \mathcal C([0,T];\bar{\mathbb{H}}_\gamma))$), where $\bar{\mathbb{H}}_\gamma$ is a weighted $L^2$-space on the graph $\Gamma$ such that
$\|f\circ\Pi\|_{\mathbb{H}_\gamma}=\|f\|_{\bar{\mathbb{H}}_\gamma}$ for any $f\in \bar{\mathbb{H}}_\gamma$.
 Recently, \cite{CX21} extends such fast asymptotics of \eqref{eq:SPDE} to encompass scenarios with more flexible assumptions regarding the Hamiltonian and the regularity of noise. 
 It is important to note that the SPDEs on graphs allow for the modeling of complex systems where the underlying space may be not a Euclidean space but rather a network with a possibly irregular structure (see, e.g., \cite{BMZ08}). In particular, it can provide a concise framework for exploring the asymptotic behaviors of SPDEs with a small parameter.
The study of this area is still in its infancy, with few results currently available in the literature. In addition to the aforementioned papers \cite{CF19,CX21}, the seminal work \cite{CF17} demonstrates that the parabolic SPDE in a narrow channel converges in distribution to an SPDE on a certain graph, as the width $\epsilon$ of the channel tends to $0$. 
In \cite{FW21}, a new class of SPDEs on graphs of Wright--Fisher type is introduced, derived as scaling limits of suitably defined biased voter models.

 In the present work, we mainly focus on asymptotic-preserving (AP) numerical approximations  for the original multiscale model
 \eqref{eq:SPDE}
 and the limit equation \eqref{eq:utzk} since their analytical solutions do not exist in general.  Implementing a numerical scheme without the AP property could potentially lead to incorrect conclusions about the limiting equation. The concept of AP schemes was introduced in for instance \cite{JS99,JS12} for deterministic partial differential equations and extended in \cite{BR22} to a class of slow-fast stochastic ordinary differential equations. We also refer to, e.g., \cite{ELV05} for an in-depth error analysis of heterogeneous multiscale methods for multiscale stochastic ordinary differential equations. 
 Roughly speaking, a numerical scheme is called 
AP for \eqref{eq:SPDE} if, as the $N=\mathsf T/\tau$ (with
$\mathsf T>0$ fixed and $\tau$ being the stepsize) tends to infinity,
 the fast advection asymptotical behavior of the corresponding numerical solution $U_\epsilon^N$ of $u_\epsilon(\mathsf{T})$ is consistent with $\bar{u}(\mathsf T)$, say 
\begin{equation}
\label{eq:UNep-u}\lim_{N\to\infty}\lim_{\epsilon\to 0}U_\epsilon^N=\bar{u}(\mathsf T)\circ\Pi\quad\text{in some sense}.
\end{equation}
 This consistency is crucial for accurately capturing the transition from the multiscale model \eqref{eq:SPDE} to its asymptotical limit \eqref{eq:utzk}, ensuring that the numerical simulations remain valid across different scales $0<\epsilon\ll 1$. 
 \begin{figure}[!htb]
\centering
\scalebox{1}{
\begin{tikzpicture}[node distance=50pt]
\node[text width=4em, align=center] (UpperLeft) {$U_{\epsilon}^N$};
\node[text width=4em, align=center, right=80pt of UpperLeft] (UpperRight) {$\bar{U}^N \circ \Pi$};
\node[text width=4em, align=center, below=of UpperLeft] (LowerLeft) {$U_{\epsilon}(\mathsf T)$};
\node[text width=4em, align=center, below= of UpperRight] (LowerRight) {$\bar{u}(\mathsf T) \circ \Pi$};
\node[text width=4em, align=center, right=80pt of UpperRight] (Upper2Right) {$\bar{U}^N$};
\node[text width=4em, align=center, right= 80pt of LowerRight] (Lower2Right) {$\bar{u}(\mathsf T)$};

\draw[->] (UpperLeft) -- node[above] {$\epsilon \rightarrow 0$} (UpperRight);
\draw[->] (UpperLeft) -- node[below] {\tiny{Theorem \ref{theo:asy}
}} (UpperRight);

\draw[->] (UpperLeft) -- node[left, text width=0.6em, align=center] {\rotatebox{-90}{$N \rightarrow \infty$}} (LowerLeft);
\draw[->] (UpperLeft) -- node[right, text width=0.8em, align=center] {\rotatebox{-90}{\tiny{Theorem 
\ref{theo:MS-2D}}}} (LowerLeft);

\draw[->] (LowerLeft) -- node[above] {$\epsilon \rightarrow 0$} (LowerRight);
\draw[->] (LowerLeft) -- node[below] {\tiny{\cite[Theorem 5.3]{CF19}}} (LowerRight);

\draw[->] (UpperRight) -- node[below] {\tiny{Projection}}(Upper2Right);

\draw[->] (LowerRight) -- node[below] {\tiny{Projection}}(Lower2Right);

\draw[->] (Upper2Right) -- node[right, text width=0.6em, align=center] {\rotatebox{-90}{$N \rightarrow \infty$}} (Lower2Right);
\draw[->] (Upper2Right) -- node[left, text width=0.8em, align=center] {\rotatebox{-90}{\tiny{Theorem 
\ref{theo:MS-1D}}}} (Lower2Right);
\end{tikzpicture}
}
\caption{A commutative diagram.} 
\label{fig.doubleDiagram}
\end{figure}

However, the design and analysis of AP schemes will encounter several challenges. First, to ensure that the numerical approximation captures the correct statistical properties of \eqref{eq:SPDE}, we need to track how the approximation errors propagate across different scales. Second, for a fixed $N>1$, the error bound between the original solution and its numerical solution should depend on $\epsilon^{-1}$ at most polynomially (other than exponentially) in order to simulate \eqref{eq:SPDE} with acceptable computational costs. This phenomenon   
potentially makes that the error bound of $\|U_\epsilon^N-\bar{u}(\mathsf T)\circ\Pi\|_{\mathbb{H}_\gamma}$ explodes 
as $\epsilon\to 0$ (see, e.g., \cite{BR22}). 
Third, to identify the AP property \eqref{eq:UNep-u}, it is crucial to find a new metric space other than $\bar{\mathbb{H}}_\gamma$, so that the error between $(U_\epsilon^N)^\wedge$ and $\bar{u}(\mathsf T)$  is controllable in such new space. In addition, a technical challenge is to 
handle the singularity near each vertex since the dominated operator $\bar{\mathcal L}$ is not uniformly elliptic on the graph.

To address the aforementioned challenges, one key ingredient is the property from \cite{CF19} that the semigroup $S_\epsilon(t)$ generated by $\mathcal L_\epsilon$ converges to the semigroup $\bar{S}(t)$ generated by $\bar{\mathcal L}$ (as referenced in \eqref{eq:Seps-S}). 
To make use of this  property,
 we propose the exponential Euler approximation for the multiscale SPDE \eqref{eq:SPDE} in the time direction, 
which is designed by
 freezing the integrands at the left endpoints of the intervals. 
 In detail,
 let $\{t_n:=n\tau\}_{n=0}^N$ be a uniform partition of the time interval $[0,\mathsf{T}]$ with the time stepsize
$\tau=\mathsf{T}/N$.
The exponential Euler approximation of  \eqref{eq:SPDE} reads $U_{\epsilon}^0=\psi$ and
\begin{equation}\label{eq:EEM}
U_{\epsilon}^{n}=S_\epsilon(\tau) U_{\epsilon}^{n-1}+S_\epsilon(\tau) B(U_\epsilon^{n-1}) \tau+S_\epsilon(\tau) G(U_\epsilon^{n-1}) \delta\mathcal{W}_{n-1}
\end{equation}
for $n=1,\ldots,N$, where
$\delta\mathcal{W}_{n-1}:=\mathcal{W}(t_n)-\mathcal{W}(t_{n-1})$ is the temporal increment of $\mathcal{W}$, and $B$ and $G$ are Nemytskij operators associated with $b$ and $g$, respectively.
It should be pointed out  that the error estimate of the approximation \eqref{eq:EEM} is different from the bounded domain setting (see, e.g., \cite{LT13}). 
Due to the loss of the smoothing property of the semigroup $S_{\epsilon}(t),t\ge 0$, we use the variational arguments, together with the skew-symmetry of the symplectic matrix and the divergence-free property of $\nabla^{\perp} H(x)$, to establish gradient estimates of SPDEs at first. Then the approximation error of \eqref{eq:EEM} is measured as follows
 \begin{align*}
\max_{0\le n\le N}\mathbb{E}\left[\|u_{\epsilon}(t_n)-U_{\epsilon}^{n}\|^2_{\mathbb{H}_\gamma}\right]
\lesssim \tau(1+\epsilon^{-2})
\end{align*}
as stated in Theorem \ref{theo:MS-2D}. Note that the mean square  convergence rate of \eqref{eq:EEM} is of order $1/2$ and 
depends on $1/\epsilon$ linearly.

Concerning the fast advection asymptotics, a further question is 
\emph{whether 
the exponential Euler approximation preserves the asymptotical behavior \eqref{eq:u-u} or not.}
To answer this question, we provide in Theorem \ref{theo:asy} 
 a discrete analogue of \eqref{eq:u-u}, namely
 \begin{equation}
 \label{eq:IntroUn-Un}
\lim_{\epsilon\to 0}\E\left[\sup_{1\le n\le N}\|U_{\epsilon}^{n}-(\bar U^{n})^{\vee}\|_{\mathbb{H}_\gamma}^p\right]=0\quad \forall~p\ge2
\end{equation} by using the asymptotical behavior \eqref{eq:Seps-S} of $\{S_\epsilon(t)\}_{t\in[0,\mathsf T]}$ and a factorization argument (see, e.g., \cite{DP14}). Here, $\{\bar{U}^n\}_{n=0}^N$ is the exponential Euler approximation of the limiting process $\{\bar{u}(t_n)\}_{n=0}^N$ (see \eqref{eq:Un}).
  Since the numerical solution $\{U_\epsilon^n\}_{n=0}^N$ is a discrete-time process that evolves on grid points progressively,  an additional condition that the initial value $\psi\in \mathbb{H}_\gamma$ is constant on every connected component of $H$ is imposed compared to the continuous case in \cite{CF19} 
 (see Lemma \ref{prop:AC} for more details). 
 We also point out that common numerical discretizations, like the standard Euler--Maruyama approximation,  may fail to preserve the fast advection asymptotics \eqref{eq:u-u}. 
  The property \eqref{eq:IntroUn-Un} reveals that the exponential Euler approximation \eqref{eq:EEM} can preserve the asymptotical behavior \eqref{eq:u-u} of the stochastic RDA equation \eqref{eq:SPDE} in the fast advection limit. 
Consequently,
verifying the AP property of the exponential Euler approximation for \eqref{eq:SPDE} boils down to measuring the error between $\bar{u}(\mathsf T)$ and $\bar{U}^N$. 
To this end,
we focus on a simplified case that $H$ is a radius function with a unique 
critical point $0$, as outlined in Assumption \ref{Asp:zeta}. 
In order to offset the singularity of the semigroup $\bar{S}(t)$ near the vertex $0$, we carry out the error analysis of the exponential Euler approximation for the limiting equation \eqref{eq:utzk} in a suitable metric space $L^2(\Gamma;A\gamma\ud z)$, by virtue of the fact that $A(z):=\frac12\oint_{\{x\in\R^2:H(x)=z\}} |\nabla H(x)| \ud l_{z}$ behaves like a linear function (see Lemma \ref{lem:TFA}). 
This reveals that the exponential Euler approximation \eqref{eq:Un} of the limiting equation \eqref{eq:utzk} also exhibits a mean square convergence of order $1/2$ (as stated in Theorem \ref{theo:MS-1D}).
As a result, we are able to depict the AP property of the exponential Euler approximation for \eqref{eq:SPDE}, as in Fig.\ \ref{fig.doubleDiagram}.

 The underlying graph $\Gamma$ considered in this work is a continuous object consisting of all points on their edges. By discretizing SPDEs on graphs spatially, one may obtain more discrete models, such as the interacting stochastic differential equations discussed in \cite{FW21}. 
 Such models can be regarded as SPDEs where the spatial domain is a discrete set comprising exclusively the vertices of graphs. 
 SPDEs on discrete graphs have attracted  more attention recently, 
 including the Wasserstein Hamiltonian flows on graphs with common noise \cite{CLZ23} and the stochastic nonlinear Schr\"odinger equations on graphs \cite{MR4612606}. They can be naturally perceived as stochastic extensions of deterministic partial differential equations on graphs, as developed in previous works \cite{CHLZ12, MR3926122, MA11, MJ11}.  In the future, we plan to combine the exponential Euler approximation with proper spatial discretizations and study fully discrete AP schemes of multiscale SPDEs such as \eqref{eq:SPDE}.

The rest of this paper is organized as follows. Section \ref{S2} collects some preliminaries on the Hamiltonian, graph, and projection operator. We also revisit the well-posedness of \eqref{eq:SPDE} and \eqref{eq:utzk} in section \ref{S2}.
The mean square convergence analysis of the exponential Euler approximation for the stochastic RDA equation \eqref{eq:SPDE} is presented in section \ref{S3}. Then we identify the fast advection asymptotics of the numerical solution in section \ref{S4}. 
Section \ref{S5} is devoted to studying the AP property of the proposed
approximation. We perform some numerical experiments in section \ref{S6} to illustrate the theoretical findings.

\section{Preliminaries}\label{S2}
 In the sequel, $C$ and $c$ denote generic positive constants which are independent of the time stepsize $\tau>0$, time variable $t\in[0,\mathsf{T}]$, small parameter  $\epsilon\in(0,1]$ and may differ from occurrence to occurrence. We use $x\lesssim y$ to indicate that $x\le Cy$ for $x,y\in\R$. Sometimes, we write $C(a,b)$ to emphasize its dependence on certain quantities $a,b$. 
In sections \ref{S2.1}-\ref{S2.2}, we introduce some notions about the graph $\Gamma$ and projection $\Pi,$ and main assumptions about the Hamiltonian $H$ and the graph weighted function $\gamma$. We revisit the well-posedness of the stochastic RDA equation \eqref{eq:SPDE} and the limiting equation \eqref{eq:utzk} in sections \ref{S:SRDA} and \ref{S:SPDE}, respectively.
  The main setting is borrowed from \cite{CF19}.
  
\subsection{Hamiltonian, graph, and projection}\label{S2.1}
For every fixed $\epsilon>0$, consider the diffusion process
\begin{align}\label{eq:Xt}
\ud X_\epsilon(t)=\frac{1}{\epsilon} \nabla^{\perp} H\left(X_\epsilon(t)\right) \ud t+\ud \textup{B}(t), 
\end{align}
where $\{\textup B(t)\}_{t\ge0}$ is a $2$-dimensional Brownian motion defined on some stochastic basis $(\Omega, \mathcal{F},\{\mathcal{F}_t\}_{t \geq 0}, \mathbb{P})$. Throughout this paper, we always assume that the Hamiltonian $H$ satisfies the following Hypothesis \ref{asp:H}.

 \begin{hyp}\label{asp:H} Assume that the Hamiltonian $H: \mathbb{R}^2 \rightarrow \mathbb{R}$ satisfies
\begin{enumerate}
\item[(i)]  $H$ is fourth continuously differentiable with bounded second derivative and
$\min _{x \in \mathbb{R}^2} H(x)=0;
$
\item[(ii)]  $H$ has only a finite number of critical points $x_1, \ldots, x_{m_1}$. The Hessian matrix $\nabla^2 H(x_i)$ is non-degenerate for every $i=1, \ldots, m_1$, and $H(x_i) \neq$ $H(x_j)$ if $i \neq j$;
\item[(iii)]  There exist $a_1, a_2, a_3>0$ such that 
 for all $x \in \mathbb{R}^2$ with $|x|$ large enough,
$$H(x) \ge a_1|x|^2,\quad|\nabla H(x)| \ge a_2|x|,\quad\Delta H(x) \ge a_3.$$
\end{enumerate}
\end{hyp}

To characterize the limiting process of $\{X_\epsilon(t)\}_{t\ge0}$, we  introduce a graph $\Gamma$ associated with the Hamiltonian $H$ as follows.
By identifying all points in $\mathbb{R}^2$ in the same connected component of a given level set $C(z):=\{x\in\R^2:H(x)=z\}$ of the Hamiltonian $H$, we obtain a graph $\Gamma$ consisting of several intervals $I_1, \ldots I_{m}$ and vertices $O_1, \ldots, O_{m_1}$. 
\begin{figure}[!htb]
\centering
\includegraphics[width=0.7\linewidth]{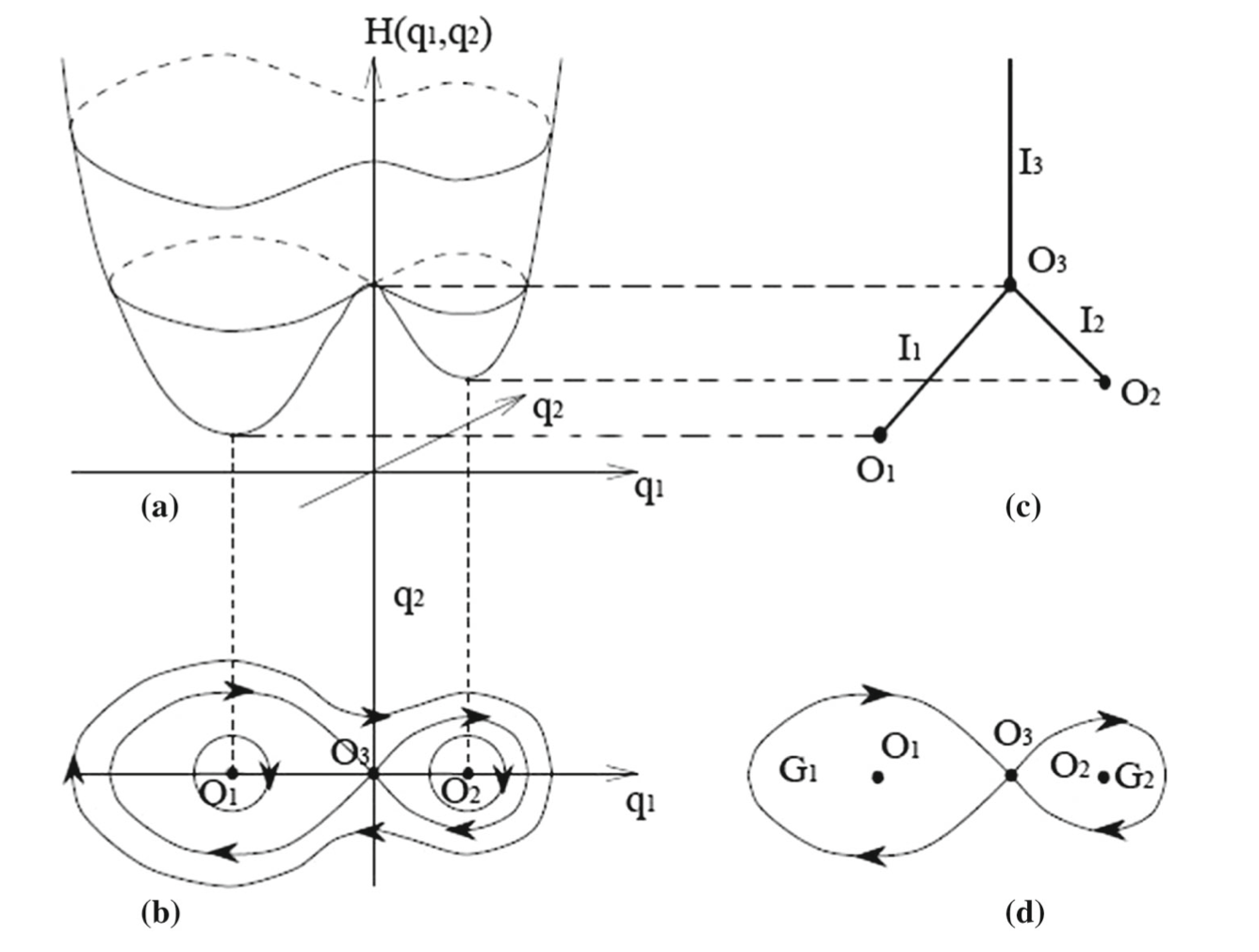}
\caption{(a) Hamiltonian $H$ with three critical points corresponding to $O_1,O_2,O_3$; (b) Level sets of $H$; (c) Graph $\Gamma$ with three edges $I_1,I_2,I_3$ and four vertices $O_1,O_2,O_3,O_\infty$; (d) Level set $\{x\in\R^2,H(x)=H(O_3)\}$ \cite{CF19}.}
\end{figure}
The vertices will be of two different types, external and internal vertices, which, respectively, correspond to local extrema and saddle points of $H$. We will also include $O_{\infty}$ among external vertices, the endpoint of the only unbounded interval in the graph, corresponding to the point at infinity.

We  define the projection $\Pi: \mathbb{R}^2 \rightarrow \Gamma$
via $\Pi(x)=(H(x), k(x))$, where $k(x)$ denotes the number of the interval on the graph $\Gamma$ containing the point $\Pi(x)$. 
When $x$ corresponds to an interior vertex $O_i$, there are three edges having $O_i$ as their endpoint, which implies that $k(x)$ is not uniquely defined. 
On the graph $\Gamma$, a distance can be introduced in the following way. If $y_1=(z_1, k)$ and $y_2=(z_2, k)$ belong to the same edge $I_k$, then $d(y_1, y_2)=|z_1-z_2|$. In the case that $y_1$ and $y_2$ belong to different edges, 
$$
d(y_1, y_2):=\min \left\{d(y_1, O_{i_1})+d(O_{i_1}, O_{i_2})+\cdots+d(O_{i_j}, y_2)\right\},
$$
where the minimum is taken over all possible paths from $y_1$ to $y_2$, through every possible sequence of vertices $O_{i_1}, \ldots, O_{i_j}$, connecting $y_1$ to $y_2$.

For every $(z,k)\in\Gamma$, the probability measure 
$$\ud\mu_{z,k}:=\frac{1}{T_k(z)}\frac{1}{|\nabla H(x)|} \ud l_{z, k}$$ is invariant for the deterministic Hamiltonian system 
$
\dot{X}(t)=\nabla^{\perp} H(X(t))$.
Here $\ud l_{z, k}$ is the length element on the connected component $C_k(z)$ of the level set $C(z)$, corresponding to the edge $I_k$, and
$T_k(z)=\oint_{C_k(z)} \frac{1}{|\nabla H(x)|} \ud l_{z, k}$ is the 
 period of the motion on the level set $C_k(z)$. 
By applying an averaging procedure with respect to the invariant measure $\mu_{z,k}$, it has been shown in \cite[Chapter 8]{FW12} that for any $X_\epsilon(0)=x\in\R^2$,
 the process $\Pi(X_\epsilon(t))$
 converges, in the sense of weak convergence of distributions in $\mathcal C([0,\mathsf{T}];\Gamma)$ to a Markov process $\bar{Y}$ on $\Gamma$, corresponding to the infinitesimal generator 
\begin{equation}\label{eq:barL}\bar{\mathcal L} f(z, k)=\frac{1}{ T_k(z)} \frac{\ud}{\ud z}\left(A_k \frac{\ud f}{\ud z}\right)(z), \quad\text{ if $(z, k)$ is an interior point of $I_k$,}
\end{equation}where $f$ is a function defined on $\Gamma$ and 
$A_k(z)=\frac12\oint_{C_k(z)} |\nabla H(x)| \ud l_{z, k}$ (see also \cite[section 2]{CF19} for more details). 

Borrowed from \cite{FW12},
we give an intuition for the operator $\bar{\mathcal L}$, by assuming that there is only one edge $[0,\infty)$ on $\Gamma$. In the case of $\Gamma=[0,\infty)$, we omit the subscript $k$ of $C_k(z), T_k(z)$, $A_k(z)$, $l_{z,k}$, $\mu_{z,k}$, etc.
The It\^o formula applied to $H(X_\epsilon(t))$ gives that
$$H(X_\epsilon(t))=H(X_\epsilon(0))+\int_0^t\frac{1}{2}\Delta H(X_\epsilon(s))\ud s+\int_0^t\nabla H(X_\epsilon(s))\ud \textup{B}(s),\quad t\ge 0.$$
Roughly speaking, $\{X_\epsilon(t)\}_{t>0}$ rotates many times along the trajectories of the deterministic Hamiltonian system before $H(X_\epsilon(t))$ changes considerably. By an averaging procedure with respect to the invariant measure $\mu_{z}$, one has that for sufficiently small $\epsilon>0$, the deterministic term $\int_0^t\frac{1}{ 2}\Delta H(X_\epsilon(s))\ud s$ and the quadratic variation $\int_0^t|\nabla H(X_\epsilon(s))|^2\ud s$ of the stochastic term 
can be approximated by $\int_0^t \beta(H(X_\epsilon(s)))\ud s$ and $\int_0^t 2\alpha(H(X_\epsilon(s)))\ud s$, respectively, where
$$
\beta(z):=\frac{1}{2}\oint_{C(z)}\Delta H(x)\ud\mu_{z},\quad \alpha(z):=\frac12\oint_{C(z)}|\nabla H(x)|^2\ud\mu_{z},\quad z\in[0,\infty).$$
Due to \cite[Lemma 1.1]{FW12}, one have that 
\begin{equation}\label{eq:ATA}
\alpha(z)=\frac{A(z)}{T(z)},\quad
\beta(z)=\frac{1}{2T(z)}\oint_{C(z)}\frac{\Delta H}{|\nabla H|}\ud l_z=\frac{A^\prime(z)}{T(z)}.
\end{equation}
Hence at least formally, the generator of the limiting process of $\{H(X_\epsilon(t))\}_{t\ge 0}$, as $\epsilon\to 0$, is
 $\alpha(z)\partial_{zz}+\beta(z)\partial_z$, which is consistent with the operator $\bar{\mathcal L}$ defined in \eqref{eq:barL}.

\subsection{Main assumptions}\label{S2.2}
For fixed $\epsilon>0$, the operator $\mathcal{L}_\epsilon=\frac{1}{2} \Delta +\frac{1}{\epsilon}\nabla^{\perp} H(x)\cdot \nabla$
is the generator of the diffusion process $\{X_\epsilon(t),t\ge0\}$. In this part, we present that the asymptotical behavior of the semigroup $\{S_\epsilon(t),t>0\}$ generated by $\mathcal{L}_\epsilon$ can be described by the semigroup $\{\bar{S}(t),t>0\}$ generated by $\bar{\mathcal{L}}$ (see \cite{CF19} for more details). This property will be used to study the fast advection asymptotics of the numerical solution in section \ref{S4}.

Considering that $S_\epsilon(t)$ is an operator acting on functions defined on the whole space $\R^2$, we study it in a space of weighted square-integrable function  $
\mathbb{H}_\gamma:=L^2(\R^2;\gamma(\Pi(x))\ud x)
$, where $\gamma: \Gamma \rightarrow(0,+\infty)$ is a bounded continuous function with
\begin{align}\label{eq:IkTk}
\sum_{k=1}^m \int_{I_k} \gamma(z, k) T_k(z) \ud z<\infty.
\end{align} 
For simplicity, we define 
$
f^{\vee}:=f\circ\Pi
$ for any function $f$ defined on $\Gamma$, and rewrite $
\mathbb{H}_\gamma=L^2(\R^2;\gamma^{\vee}(x)\ud x)
$.
It should be noted that $f^{\vee}$ is a well-defined function on $\R^2$ provided that $f(z,k_1)=f(z,k_2)=f(z,k_3)$ whenever $(z,k_1)$, $(z,k_2)$ and $(z,k_3)$ correspond to the same interior vertex of $\Gamma$.
By \eqref{eq:IkTk} and the definition of $T_k(z)$, it holds that 
\begin{equation}\label{eq:L1}
\int_{\mathbb{R}^2} \gamma^{\vee}(x) \ud x =\sum_{k=1}^m \int_{I_k} \oint_{C_k(z)} \frac{\gamma^{\vee}(x)}{|\nabla H(x)|} \ud l_{z, k} \ud z =\sum_{k=1}^m \int_{I_k} \gamma(z, k) T_k(z) \ud z<\infty,
\end{equation}
which shows that $\gamma^{\vee}: \mathbb{R}^2 \rightarrow(0,+\infty)$ is bounded, continuous and integrable. We notice that for any bounded function $\varphi$ on $\R^2$,
\begin{align*}
\int_{\mathbb{R}^2} \varphi(x)\gamma^{\vee}(x) \ud x & =\sum_{k=1}^m \int_{I_k}\gamma(z,k) \oint_{C_k(z)}  \frac{\varphi(x)}{|\nabla H(x)|} \ud l_{z, k} \ud z.
\end{align*}
Hence by defining the formal projection $^\wedge:\varphi\mapsto\varphi^{\wedge}$  via
$$
\varphi^{\wedge}(z, k):=\frac{1}{T_k(z)} \oint_{C_k(z)} \frac{\varphi(x)}{|\nabla H(x)|} \ud l_{z, k}=\oint_{C_k(z)}\varphi(x)\ud\mu_{z,k},\quad (z,k)\in\Gamma,
$$
we have that for any bounded function $\varphi$ on $\R^2$,
$$
\int_{\mathbb{R}^2} \varphi(x)\gamma^{\vee}(x) \ud x  =\sum_{k=1}^m \int_{I_k}\varphi^{\wedge}(z, k)\gamma(z,k) T_k(z) \ud z=\int_{\Gamma}\varphi^{\wedge}(z,k)\ud\nu_{\gamma}(z,k).
$$ 
Here the measure $\nu_\gamma$ defined by $\nu_\gamma(E):=\sum_{k=1}^m\int_{I_k}\mathbb{I}_{E}(z,k)\gamma(z,k) T_k(z) \ud z$ for any Borel measurable set $E\subset\Gamma$ can be seen as a projection of the measure $\gamma^{\vee}(x)\ud x$, where $\mathbb{I}_{E}$ denotes the indicator function on the set $E$.
Moreover, the space $\mathbb{H}_\gamma$ is properly projected to the space
\begin{equation}\label{eq:H1}
\bar{\mathbb{H}}_\gamma:=\left\{f: \Gamma \rightarrow \mathbb{R}: \sum_{k=1}^m \int_{I_k}|f(z, k)|^2 \gamma(z, k) T_k(z) \ud z<\infty\right\}=L^2(\Gamma,\nu_\gamma).
\end{equation}

By \cite[Proposition 3.1]{CF19}, the operators $^\wedge$ and $^\vee$ are contractions, i.e.,
\begin{equation}\label{fvarphi}
\|\varphi^{\wedge}\|_{\bar{\mathbb{H}}_\gamma}\le \|\varphi\|_{\mathbb{H}_\gamma},\quad \|f^{\vee}\|_{\mathbb{H}_\gamma}= \|f\|_{\bar{\mathbb{H}}_\gamma}.
\end{equation}
Given a Hilbert space $V$, denote by $\mathscr{L}(V)$ the space of bounded linear operators from $V$ to $V$.
For every $\mathcal Q \in \mathscr{L}\left(\mathbb{H}_\gamma\right)$ and $f \in \bar{\mathbb{H}}_\gamma$, we define
$
\mathcal Q^{\wedge} f:=\left(\mathcal Q f^{\vee}\right)^{\wedge}.
$
In an analogous way, for any $\mathcal A \in \mathscr{L}(\bar{\mathbb{H}}_\gamma)$ and $\varphi \in \mathbb{H}_\gamma$, we define
$
\mathcal A^{\vee} \varphi:=(\mathcal A \varphi^{\wedge})^{\vee}.
$
Due to \eqref{fvarphi}, it holds that $\mathcal Q^{\wedge}\in \mathscr{L}(\bar{\mathbb{H}}_\gamma)$, $\mathcal A^{\vee} \in \mathscr{L}(\mathbb{H}_\gamma)$ and
\begin{equation*}
\|\mathcal Q^{\wedge}\|_{\mathscr{L}(\bar{\mathbb{H}}_\gamma)}\le \|\mathcal Q\|_{\mathscr{L}(\mathbb{H}_\gamma)},\quad
\|\mathcal A^{\vee}\|_{\mathscr{L}(\mathbb{H}_\gamma)} \leq\|\mathcal A\|_{\mathscr{L}(\bar{\mathbb{H}}_\gamma)}.
\end{equation*}
Similar to \cite{CF19}, we shall make the following assumption on $\gamma$ such that $\gamma^{\vee}$ is admissible with respect to all semigroups $\{S_\epsilon(\cdot )\}_{\epsilon>0}$. 
 \begin{asp}\label{Asp:gamma}
 Let $\gamma(z,k)=\vartheta(z)$ for every $(z,k)\in\Gamma$, where
$\vartheta:[0,\infty)\to(0,\infty)$ is second differentiable. Moreover,
\begin{equation}\label{eq:zvartheta}
|z\vartheta^{\prime\prime}(z)|+|\vartheta^{\prime}(z)|\le C\vartheta(z)\quad \text{ for all } z\in[0,\infty).
\end{equation}
 \end{asp}
 
 According to Hypothesis \ref{asp:H}($i$), there exist some positive constants $C_1,C_2>0$ such that 
$H(x)\le C_1|x|^2+C_2$, which implies that for any $K>0$,
$$\{x\in\R^2:H(x)\ge C_1K^2+C_2\}\subset\{x\in\R^2: |x|\ge K\}.$$
By Hypothesis \ref{asp:H}($iii$), there exists $C_3>0$ such that
$$\{x\in\R^2:|x|\ge C_3\}\subset\{x\in\R^2: H(x)\ge a_1|x|^2\}.$$
Hence for any $z\ge C_1C_3^2+C_2$,
$$\{x\in\R^2:H(x)\ge z\}\subset\{x\in\R^2:H(x)\ge C_1C_3^2+C_2\}\subset\{x\in\R^2: H(x)\ge a_1|x|^2\}.$$
In the sequel, we denote by 
\begin{equation*}
z_0:=\max\left\{\max_{1\le i\le m_1}H(x_i)+1,C_1C_3^2+C_2\right\},
\end{equation*}
where $\{x_i\}_{i=1}^{m_1}$ are the critical points of $H$.
Then $z_0>\max_{1\le i\le m_1}H(x_i)$ and
$H(x)\ge z_0$ implies $H(x)\ge a_1|x|^2$.

  \begin{exam}\label{Ex:gamma-S1}
 Let $\gamma(z,k)=\vartheta(z)$ for every $(z,k)\in\Gamma$, where
$\vartheta:[0,\infty)\to(0,\infty)$ is second differentiable. Moreover, there exist  $c_0>0$ and $\lambda>1$ such that
 $$\vartheta(z)=c_0z^{-\lambda},\quad z\ge z_0.$$ 
\end{exam}
 \begin{exam}\label{Ex:gamma-S}
 Let $\gamma(z,k)=\vartheta(z)$ for every $(z,k)\in\Gamma$, where
$\vartheta:[0,\infty)\to(0,\infty)$ is second differentiable. Moreover, there exist some constants $\lambda>0$ and $c_0>0$ such that
 $$\vartheta(z)=c_0e^{-\lambda(\sqrt{z}-\sqrt{2z_0})},\quad z\ge z_0.$$ 
 \end{exam}

The graph weighted functions $\gamma$ in Examples \ref{Ex:gamma-S1} and \ref{Ex:gamma-S} satisfy Assumption \ref{Asp:gamma} and \eqref{eq:IkTk}.
Indeed, the verification of \eqref{eq:zvartheta} is straightforward for $z\ge z_0\ge1$ by using the explicit formula of $\vartheta$. As for $z\le z_0$, \eqref{eq:zvartheta} follows from \eqref{eq:gamma} and the continuity of $z\vartheta^{\prime\prime}(z)$
and $\vartheta^\prime(z)$. 
In addition,
the integrability of $\gamma^{\vee}=\vartheta\circ H$ 
  follows from the assumption that $H$ grows quadratically at infinity. For instance, for $\gamma$ in Example \ref{Ex:gamma-S1}, we divide $\R^2=\mathcal E\cap \mathcal E^c$ with $\mathcal E:=\{H(x)\ge z_0\}\cap\{|x|\ge1\}$ to obtain
\begin{align*}
\int_{\R^2}\gamma^\vee(x)\ud x&=\int_{\mathcal E}\vartheta(H(x))\ud x+\int_{\mathcal E^c}\vartheta(H(x))\ud x\\
&\le \int_{\{H(x)\ge a_1|x|^2\}\cap\{|x|\ge1\}}c_0H(x)^{-\lambda}\ud x+C\\
&\le  a_1^{-\lambda}c_0\int_{\{|x|\ge1\}}|x|^{-2\lambda}\ud x+C<\infty,
\end{align*}
since $\mathcal E^c=\{H(x)\le z_0\}\cup\{|x|\le1\}$ is a bounded set and $\lambda>1$. Hence \eqref{eq:IkTk} follows from $\gamma^\vee\in L^1(\R^2)$ and \eqref{eq:L1}.

 \begin{lem}
 Under Assumption \ref{Asp:gamma}, there exists some positive constant $c$ such that
 $\Delta \gamma^{\vee}(x)\le c \gamma^{\vee}(x)$ for all $x\in\R^2$.
 \end{lem}
 \begin{proof}
  Under Assumption \ref{Asp:gamma}, $\gamma^{\vee}=\vartheta\circ H$.
 The positivity and continuity of $\vartheta$ imply that 
\begin{equation}\label{eq:gamma}\inf_{0\le H(x)\le z_0}\gamma^{\vee}(x)=\inf_{0\le H(x)\le z_0}\vartheta(H(x))=\inf_{0\le z\le z_0}\vartheta(z)=:C(z_0)>0.
\end{equation}
Since $\{x\in\R^2:0\le H(x)\le z_0\}$ is a bounded set, we have that for $ H(x)\le z_0$,
\begin{equation*}
\Delta \gamma^{\vee}(x)=\vartheta^{\prime \prime}(H(x))|\nabla H(x)|^2+\vartheta^{\prime}(H(x)) \Delta H(x)\le C\le CC(z_0)^{-1}\gamma^{\vee}(x).
\end{equation*}
For $H(x)\ge z_0$, we have $H(x)\ge a_1|x|^2$, $|\nabla H(x)|\le C|x|$ and $|\Delta H(x)|\le C$. Hence  it follows from Assumption \ref{Asp:gamma} that for $H(x)\ge z_0$,
\begin{align*}
\Delta \gamma^{\vee}(x)&=\vartheta^{\prime \prime}(H(x))|\nabla H(x)|^2+\vartheta^{\prime}(H(x)) \Delta H(x)\\
&\le C^2|\vartheta^{\prime \prime}(H(x))||x|^2+C|\vartheta^{\prime}(H(x))|\\
&\le C^2a_1^{-1}|\vartheta^{\prime \prime}(H(x))H(x)|+|\vartheta^{\prime}(H(x))|
\le C\gamma^{\vee}(x).
\end{align*}
The proof is completed.
 \end{proof}

Without pointing it out explicitly, we always assume that  Assumption \ref{Asp:gamma} holds throughout this section.
Then
 $\{S_\epsilon(t)\}_{\epsilon\in(0,1],t\in[0,\mathsf{T}]}$ is a bounded linear operator on $\mathbb{H}_\gamma$ (see \cite[Proposition 4.1]{CF19} or Lemma \ref{prop:Lrho2} below), i.e., 
\begin{equation}\label{eq:Seps}
\|S_\epsilon(t)\|_{\mathscr{L}(\mathbb{H}_\gamma)}\le C(\mathsf{T})\quad \forall~\epsilon\in(0,1],t\in[0,\mathsf{T}].
\end{equation}
According to \cite[Appendix B]{CF19}, under
Assumption \ref{Asp:gamma} and the following condition
\begin{equation}\label{eq:Tneq0}
\frac{\ud}{\ud z}T_k(z)\neq 0,\quad (z,k)\in\Gamma,
\end{equation}
it holds that
 for any $\varphi \in \mathbb{H}_{\gamma}$ and $0<\tau_0 < \mathsf{T}$, 
\begin{equation}\label{eq:Seps-S}
\lim _{\epsilon \rightarrow 0} \sup _{t \in[\tau_0, \mathsf{T}]}\left\|S_\epsilon(t) \varphi-\bar{S}(t)^{\vee} \varphi\right\|_{\mathbb{H}_\gamma}=\lim _{\epsilon \rightarrow 0} \sup _{t \in[\tau_0, \mathsf{T}]}\left\|\left(S_\epsilon(t) \varphi\right)^{\wedge}-\bar{S}(t) \varphi^{\wedge}\right\|_{\bar{\mathbb{H}}_\gamma}=0,
\end{equation}
and the semigroup $\{\bar{S}(t)\}_{t>0}$ generated by $\bar{\mathcal L}$ satisfies
\begin{equation}\label{eq:St}
\|\bar{S}(t)\|_{\mathscr{L}(\bar{\mathbb{H}}_\gamma)}\le C(\mathsf{T})\quad \forall~t\in[0,\mathsf{T}].
\end{equation}

Since the Hamiltonian $H$ grows essentially quadratically at infinity, a heuristic example of the Hamiltonian is $H(x)=|x|^2$, $x\in\R^2$, which implies that $\Gamma=[0,\infty)$.
 However, for $H(x)=|x|^2$, the function $T(z)\equiv \pi$ is a constant, which violates \eqref{eq:Tneq0}. To address this problem, we propose the following assumption for the case of $\Gamma=[0,\infty)$, which will be mainly used in the consistency analysis of the limiting scheme and the limiting equation \eqref{eq:utzk} in section \ref{S5}.

\begin{asp}\label{Asp:zeta}
Let $H(x)=|x|^2+\zeta(|x|^2)$, where $\zeta:[0,\infty)\to\R$ is a third differentiable function with $\zeta(0)=0$.
Assume that for some $0\le r_0<1$ and $\tilde{r}_0>0$,
\begin{gather}\label{eq:zeta1}
-r_0\le \zeta^\prime(z)\le \tilde{r}_0\quad\forall~z\in[0,\infty).
\end{gather}
Moreover, there exist $r_1<1-r_0$ and $r_2,r_3>0$ such that 
\begin{gather}\label{eq:zeta2}
|z\zeta^{\prime\prime}(z)|\le r_1\quad\text{ for $z$ large enough},\\\label{eq:zeta3}
0<|\zeta^{\prime\prime}(z)|\le r_2,\quad |z\zeta^{\prime\prime\prime}(z)|\le r_3\quad\forall~z\in[0,\infty).
\end{gather}
\end{asp}

The conditions \eqref{eq:zeta1}-\eqref{eq:zeta3} are mainly imposed such that $H(x)=|x|^2+\zeta(|x|^2)$ satisfies Hypothesis \ref{asp:H} and \eqref{eq:Tneq0}. Indeed, by \eqref{eq:zeta1} and
$\nabla H(x)=2x(1+\zeta^\prime(|x|^2))$, we have that $\nabla H(x)=0$ if and only if $x=0\in\R^2$, and thus
$x=0$ is the unique critical point of $H$.
Since
\begin{align*}
 \nabla^2H(x)
 &=2\left(1+\zeta^\prime(|x|^2)\right)I+4\zeta^{\prime\prime}(|x|^2)x\otimes x\\
 &=\left(\begin{array}{cc}2+4x_1^2 \zeta^{\prime\prime}(|x|^2)+2\zeta^\prime(|x|^2)& 4x_1x_2 \zeta^{\prime\prime}(|x|^2) \\4x_1x_2 \zeta^{\prime\prime}(|x|^2) & 2+4x_2^2 \zeta^{\prime\prime}(|x|^2)+2\zeta^\prime(|x|^2)\end{array}\right),
\end{align*}
one could verify by using \eqref{eq:zeta1}-\eqref{eq:zeta2} that
$\nabla^2H(x)$ is uniformly bounded in $\R^2$, $\nabla^2 H(0)=2(1+\zeta^\prime(0))I$
is non-degenerate, and
\begin{align*}
H(x)&
\ge(1-r_0)|x|^2,\quad x\in\R^2,\\
|\nabla H(x)|&=2|x|(1+\zeta^\prime(|x|^2))\ge 2(1-r_0)|x|,\quad x\in\R^2,\\
 \Delta H(x)&=4\left(1+|x|^2\zeta^{\prime\prime}(|x|^2)+\zeta^\prime(|x|^2)\right)\ge 4(1-r_0-r_1)>0\text{ for $|x|$ large enough}.
\end{align*}

By \eqref{eq:zeta1}, the map $z\mapsto (Id+\zeta)(z)=z+\zeta(z)$ is invertible, whose inverse is denoted by $F$, i.e., $F(z):=(Id+\zeta)^{-1}(z)$, $z\in[0,\infty)$.
Given $z\in[0,\infty)$, the level set of $H$ is
 \begin{align*}
 C(z)&=\left\{x\in\R^2:|x|=\sqrt{F(z)}\right\}\\
 &\phantom{:}=\left\{x=(x_1,x_2): x_1=\sqrt{F(z)}\cos\theta,x_2=\sqrt{F(z)}\sin\theta,\theta\in [0,2\pi]\right\},
 \end{align*} 
 on which $|\nabla H(x)|=2\sqrt{F(z)}\left(1+\zeta^\prime(F(z))\right)$ and $\ud l_z=\sqrt{F(z)}\ud \theta$. 
Moreover, 
\begin{gather}\label{eq:TA}
\left\{
\begin{split}
 T(z)&=\oint_{C(z)} \frac{1}{|\nabla H(x)|} \ud l_{z}
=\frac{\pi}{1+\zeta^\prime(F(z))},\\
 A(z)&=\frac12\oint_{C(z)} |\nabla H(x)| \ud l_{z}
=2\pi F(z)\left(1+\zeta^\prime(F(z))\right),\quad z\in[0,\infty).
\end{split}
\right.
\end{gather}
 By virtue of \eqref{eq:zeta1} and $F^\prime(z)=(1+\zeta^\prime(F(z)))^{-1}$, we have
\begin{align}\label{eq:T'}
T^\prime(z)
=-\frac{\pi\zeta^{\prime\prime}(F(z)) F^\prime(z)}{(1+\zeta^\prime(F(z)))^2}=-\frac{\pi\zeta^{\prime\prime}(F(z)) }{(1+\zeta^\prime(F(z)))^3}.
\end{align}
Therefore, a sufficient condition for \eqref{eq:Tneq0} is that
$\zeta^{\prime\prime}(z)\neq0$ for all $z\in[0,\infty)$. Hence the restriction $|\zeta^{\prime\prime}(z)|>0$ in \eqref{eq:zeta3} is imposed so that \eqref{eq:Tneq0} is fulfilled. 

It can be seen that  the Hamiltonian $H(x)=a|x|^2+a\zeta(|x|^2)$ with $a>0$ and $\zeta$ satisfying Assumption \ref{Asp:zeta} also fulfills Hypothesis \ref{asp:H} and \eqref{eq:Tneq0}.  Under Assumption \ref{Asp:zeta}, the limiting equation
 \eqref{eq:utzk} reduces to an SPDE on the half line
\begin{equation}\label{eq:utzk-p}
\left\{\begin{array}{l}
\partial_t \bar{u}(t, z)=\bar{\mathcal{L}} \bar{u}(t, z)+b(\bar{u}(t, z))+g(\bar{u}(t, z)) \partial_t \bar{\mathcal{W}}(t, z), \quad t\in[0,\mathsf T], \\
\bar{u}(0, z)=\psi^{\wedge}(z),\quad z\in[0,\infty)
\end{array}\right.
\end{equation}
with $\bar{\mathcal L}=\frac{A(z)}{T(z)}\partial_{zz}+\frac{A^\prime(z)}{T(z)}\partial_{z}$, where the functions $A$ and $T$ are given by \eqref{eq:TA}. 
\begin{exam}Below are some examples of $\zeta$  satisfying Assumption \ref{Asp:zeta}.
\begin{enumerate}
\item[1)]  Given $\alpha_0>0$, the function $\zeta(z)=1-e^{-\alpha_0 z}$, $z\in[0,\infty)$, satisfies Assumption \ref{Asp:zeta}. Indeed, it can be readily shown that 
\eqref{eq:zeta1} and \eqref{eq:zeta3} hold with $r_0=0$, $\tilde r_0=\alpha_0$, $r_2=\alpha_0^2$ and $r_3=\alpha_0^2e^{-1}$, while \eqref{eq:zeta2} is a consequence of the following fact 
$$|z\zeta^{\prime\prime}(z)|=\alpha_0^2ze^{-\alpha_0 z}\to 0\quad \text{as } z\to \infty.$$
\item[2) ] For any $\alpha_0\in(0,1)$ and $\beta_0>0$, the function $\zeta(z)=\beta_0(1+z)^{\alpha_0}-\beta_0$, $z\in[0,\infty)$ satisfies Assumption \ref{Asp:zeta}. Indeed, for any $z\in[0,\infty)$,
$\zeta^\prime(z)=\beta_0\alpha_0(1+z)^{\alpha_0-1}\in(0,\beta_0\alpha_0)$, which gives \eqref{eq:zeta1} with $r_0=0$ and $\tilde{r}_0=\beta_0\alpha_0$. The inequality \eqref{eq:zeta2} follows from the fact that
$|z\zeta^{\prime\prime}(z)|=\beta_0\alpha_0(1-\alpha_0)z(1+z)^{\alpha_0-2}\to 0$ as $z\to\infty$. Besides, the functions
$|\zeta^{\prime\prime}(z)|=\beta_0\alpha_0(1-\alpha_0)(1+z)^{\alpha_0-2}$ and $|z\zeta^{\prime\prime\prime}(z)|=\beta_0\alpha_0(1-\alpha_0)(2-\alpha_0)(1+z)^{\alpha_0-3}$ are uniformly bounded on $[0,\infty)$, which implies \eqref{eq:zeta3}.

\item[3)]
 Let $\alpha_0\in(0,1)$ and denote
 \begin{equation*}
\widetilde\zeta(z):=\begin{cases}\zeta_0 z+\frac12\zeta_1 z^2+\frac13\zeta_2z^3+\frac14\zeta_3z^4,&\quad 0\le z\le z_1,\\
z^{\alpha_0},&\quad z\ge z_1,
\end{cases}
\end{equation*}
where $z_1>0$, $\zeta_i$, $i=0,1,2,3$, are to be determined.
Assume that
\begin{equation}\label{eq:Azb}
\left(\begin{array}{cccc} z_1 &\frac12z_1^2&\frac13z_1^3&\frac14z_1^4\\1& z_1 & z_1^2 & z_1^3 \\0 & 1 & 2z_1 & 3z_1^2 \\0 & 0 & 2 & 6z_1\end{array}\right)\left(\begin{array}{c}\zeta_0 \\\zeta_1 \\\zeta_2 \\\zeta_3\end{array}\right)=\left(\begin{array}{c}z_1^{\alpha_0} \\ \alpha_0 z_1^{\alpha_0-1} \\\alpha_0(\alpha_0-1) z_1^{\alpha_0-2} \\ \alpha_0(\alpha_0-1) (\alpha_0-2) z_1^{\alpha_0-3}\end{array}\right),
\end{equation}
so that the function $\widetilde{\zeta}$ is third continuously differentiable.
For any fixed $z_1>0$ and $\alpha_0\in(0,1)$, one can find the solution of \eqref{eq:Azb} as follows 
\begin{align*}
\zeta_0&=\left(1-(\alpha_0-1)(\frac16\alpha_0^2-\frac43\alpha_0+3)\right)z_1^{\alpha_0-1},\\
\zeta_1&=(\alpha_0-4)(\alpha_0-3)(\alpha_0-1)z_1^{\alpha_0-2}<0,\\
\zeta_2&=\frac{3}{2}(4-\alpha_0)(\alpha_0-1)(\alpha_0-2)z_1^{\alpha_0-3}>0,\\
\zeta_3&=\frac{2}{3}(\alpha_0-1)(\alpha_0-2)(\alpha_0-3)z_1^{\alpha_0-4}<0.
\end{align*}
The function $\widetilde\zeta^{\prime\prime}(z)=3\zeta_3z^2+2\zeta_2z+\zeta_1$, $z\in[0,\infty)$, attains its maximum 
$$\frac{3\zeta_1\zeta_3-\zeta_2^2}{3\zeta_3}=\frac{\alpha_0(\alpha_0-4)(\alpha_0-1)^2(\alpha_0-2)(6-\alpha_0)}{12\zeta_3}z_1^{2\alpha_0-6}<0,$$
which ensures that $\widetilde\zeta^{\prime\prime}(z)$ remains negative on $[0,z_1]$.

For any $\alpha_0\in(0,1)$, there exist $\beta_0,z_1>0$ such that $\zeta(z)=\zeta_{\beta_0,z_1}(z):=\beta_0\widetilde\zeta(z)$ satisfies Assumption \ref{Asp:zeta}. Indeed, we can first fix
$z_1>1$ sufficiently large such that
\begin{gather*}
|\widetilde\zeta^\prime(z)|=\alpha_0 z^{\alpha_0-1}\le\alpha_0 z_1^{\alpha_0-1}<1\quad  \forall~z\ge z_1.
\end{gather*}
 Then fix a sufficiently small
 $\beta_0\in(0,1)$ such that for any $z\in[0,z_1]$,
\begin{gather*}
|\zeta_{\beta_0,z_1}^\prime(z)|\le\beta_0\sup_{z\in[0,z_1]}|\widetilde\zeta^\prime(z)|\le r_0:=\alpha_0 z_1^{\alpha_0-1},
\end{gather*}
which implies that $\zeta=\zeta_{\beta_0,z_1}$ satisfies \eqref{eq:zeta1}. Notice that for any $z\ge z_1$,
$$|z\widetilde\zeta^{\prime\prime}(z)|=\alpha_0(1-\alpha_0) z^{\alpha_0-1} \le \alpha_0(1-\alpha_0) z_1^{\alpha_0-1} <1-r_0,$$
that is, \eqref{eq:zeta2} holds for $\zeta=\zeta_{\beta_0,z_1}$ with  $r_1=\alpha_0(1-\alpha_0) z_1^{\alpha_0-1}$. Moreover, one can also verify that 
 \eqref{eq:zeta3} holds for $\zeta=\zeta_{\beta_0,z_1}$. 
\end{enumerate}
\end{exam}

\subsection{Stochastic RDA equation}\label{S:SRDA}
Following the framework of \cite{DP14}, we recast the stochastic RDA equation \eqref{eq:SPDE} into the following compact form
\begin{equation*}
\ud u_{\epsilon}(t)=\mathcal L_\epsilon u_{\epsilon}(t)\ud t
 +B(u_{\epsilon}(t))\ud t+G(u_{\epsilon}(t))\ud\mathcal{W}(t), \quad t\in[0,\mathsf T]
\end{equation*}
with the initial value $u_{\epsilon}(0)=\psi$. Here,
$B$ (resp. $G$) is the Nemytskii operator
associated with $b$ (resp. $g$), that is, for $v_1,v_2\in \mathbb{H}_\gamma$,
\begin{align*}
B(v_1)(x)=b(v_1(x)),\quad G(v_1)(v_2)(x)=g(v_1(x))v_2(x),\quad x\in\R^2.
\end{align*}

 Assume that the random force $\mathcal W$ in the stochastic RDA equation \eqref{eq:SPDE} is a spatially homogeneous Wiener process 
 with a finite non-negative definite symmetric spectral measure $\mu$. Namely, $\mathcal W$ is a Gaussian random field on $[0,\infty)\times\R^2$ defined on the stochastic basis $(\Omega,\mathscr F,\{\mathscr F_t\}_{t\ge0},\mathbb P)$ such that 
 \begin{enumerate}
 \item[$1$)] the mapping $(t,x)\mapsto\mathcal W(t,x)$ is continuous with respect to $t$ and measurable with respect to $(t,x)$, $\mathbb P$-a.s.;
 \item[$2$)] for every $x\in\R^2$, $\{\mathcal{W}(t,x),t\ge0\}$ is a 1-dimensional Brownian motion;
 \item[$3$)] for every  $t,s\ge 0$ and $x,y\in\R^2$,
 \begin{gather}\label{eq:Lam}
 \E\left[\mathcal W(t,x)\mathcal W(s,y)\right]= (t\wedge s)\Lambda(x-y),
  \end{gather}
  where $\Lambda$ is the Fourier transform of $\mu$, i.e.,
 $ \Lambda(x)=\int_{\R^2}e^{\textup{i}x\cdot\xi }\mu(\ud \xi)$ for $x\in\R^2$.
 Here, $\textup{i}=\sqrt{-1}$ is the imaginary unit. 
 \end{enumerate}

 \begin{exam} Below are some examples of the spectral measure $\mu$ (see e.g., \cite{BT10}).
\begin{enumerate}
\item[](i)
 If $\mu(\ud \xi)=|\xi|^{-r} \ud \xi$ for some $0<r<2$, then $\Lambda$ is the Riesz kernel of order $r$:
$$
\Lambda(x)=\frac{\mathsf{\Gamma}((2-r) / 2)}{2^{r}\mathsf{\Gamma}(r / 2)\pi}|x|^{-2+r},\quad x\in\R^2,
$$
where $\mathsf{\Gamma}$ denotes the Gamma function.

\item[] (ii) If $\mu(\ud \xi)=(1+|\xi|^2)^{-\frac{r}{2}} \ud \xi$ for some $r>0$, then $\Lambda$ is the Bessel kernel of order $r$:
$$
\Lambda(x)=(4 \pi)^{r / 2} \mathsf{\Gamma}(r / 2) \int_0^{\infty} y^{(r-2) / 2-1} e^{-y} e^{-|x|^2 /(4 y)} \ud y,\quad x\in\R^2.
$$

\item[] (iii) If $\mu(\ud \xi)=e^{-\frac12\pi^2 r|\xi|^2} \ud \xi$ for some $r>0$, then $\Lambda$ is the heat kernel of order $r$:
$$
\Lambda(x)=(2 \pi r)^{-1} e^{-\frac{|x|^2}{2 r}},\quad x\in\R^2 .
$$

\item[] (iv) If $\mu(\ud \xi)=e^{-4 \pi^2 r|\xi|} \ud \xi$ for some $r>0$, then $\Lambda$ is the Poisson kernel of order $r$:
$$
\Lambda(x)=\pi^{-1} r(|x|^2+r^2)^{-1},\quad x\in\R^2.
$$
\end{enumerate}
\end{exam}

Let $\mathcal S(\R^2)$ be the set of smooth functions on $\R^2$ with rapid decrease.
The operator $\mathscr{Q}:\mathcal S(\R^2)\times\mathcal S(\R^2)\to\R$ defined by
$$\mathscr{Q}(\varphi_1,\varphi_2):=\int_{\R^2}\Lambda(x)\int_{\R^2}\varphi_1(y)\varphi_2(y-x)\ud y\ud x,\quad\varphi_1,\varphi_2\in\mathcal S(\R^2),$$
is called the covariance form of $\mathcal W$. We denote by $RK$ the reproducing kernel space of $\mathcal W$.
By \cite[Section 1.2]{PZ97}, the space $RK$ can be identified with the dual of $\mathcal{S}_q $, where $\mathcal{S}_q $ is the completion of the set $\mathcal S(\R^2)/ker \mathscr{Q}$ with respect to the norm $q([\varphi]):=\sqrt{\mathscr{Q}(\varphi,\varphi)}$.
Denote by $L^2_{(s)}(\R^2,\ud \mu)$ the subspace of the Hilbert space $L^2(\R^2,\ud \mu;\mathbb{C})$ consisting of all functions $\varphi$ such that $ \overline{\varphi(-x)}=\varphi(x)$ for $x \in\R^2$. 
 An orthonormal basis for the reproducing kernel space $RK$  is given by $\{\widehat{\mathfrak{u}_j\mu}\}_{j\in\mathbb{N}_+}$ (cf. \cite[Proposition 1.2]{PZ97}), where $\{\mathfrak{u}_j\}_{j\in\mathbb{N}_+}$ is a complete orthonormal basis of the Hilbert space $L^2_{(s)}(\R^2,\ud \mu)$ and
$$\widehat{\mathfrak{u}_j\mu}(x)=\int_{\R^2}e^{\textup{i}x\cdot\xi} \mathfrak{u}_j(\xi)\mu(\ud\xi),\quad x\in\R^2.$$
This implies that $\mathcal W(t,x)$ has the following Karhunen--Loève expansion
\begin{align}\label{eq:KL}\mathcal W(t,x)=\sum_{j=1}^\infty\widehat{\mathfrak{u}_j\mu}(x)\beta_j(t),\quad t \geq 0,~x\in\R^2,
\end{align}
where $\{\beta_j\}_{j\in\mathbb N_+}$ is a sequence of independent Brownian motions defined on the stochastic basis $(\Omega,\mathscr F,\{\mathscr F\}_{t\ge0},\mathbb P)$. 

For a Hilbert space $(V,\|\cdot\|_V)$, let $\mathscr{L}_2(RK,V)$ be the space of Hilbert--Schmidt operators from $RK$ to $V$ endowed with the norm $\|\cdot\|_{\mathscr{L}_2(RK,V)}:=(\sum_{j=1}^\infty\|\cdot \widehat{\mathfrak{u}_j\mu}\|^2_{V})^{\frac12}$. 
 For any $p\ge2$, $0\le s_1<s_2\le\mathsf{T}$ and any predictable stochastic process $\Phi:[0,\mathsf{T}]\times\Omega\to\mathscr{L}_2(RK,\mathbb{H}_\gamma)$, we have the following Burkholder inequality (see e.g., \cite{DP14})
\begin{align}\label{eq:BDG}
\E\left[\Big\|\int_{s_1}^{s_2}\Phi(\sigma) \ud \mathcal{W}(\sigma)\Big\|_{\mathbb{H}_\gamma}^p\right]\lesssim\E\left[\Big(\int_{s_1}^{s_2}\left\| \Phi(\sigma) \right\|_{\mathscr{L}_2(RK,\mathbb{H}_\gamma)}^2\ud \sigma\Big)^{\frac{p}{2}}\right],
\end{align}
provided that the right-hand side is finite.

The Lipschitz continuity of $b$ implies
$$\|B(v_1)-B(v_2)\|_{\mathbb{H}_\gamma}\le C\|v_1-v_2\|_{\mathbb{H}_\gamma}\quad\forall~ v_1,v_2\in \mathbb{H}_\gamma.$$
Since
$\{\mathfrak{u}_j\}_{j\in\mathbb{N}_+}$ is a complete orthonormal basis of the Hilbert space $L^2_{(s)}(\R^2,\ud \mu)$, by the Parseval inequality, for any $x\in\R^2$,
\begin{equation}\label{eq:ujmu0}\sum_{j=1}^\infty|\widehat{\mathfrak{u}_j\mu}(x)|^2=\sum_{j=1}^\infty\left|\int_{\R^2}e^{\textup{i}x\cdot\xi}\mathfrak{u}_j(\xi)\mu(\ud\xi)\right|^2=\mu(\R^2)<\infty.
\end{equation}
Hence 
 for any $v_1,v_2\in \mathbb{H}_\gamma$,
\begin{align*}
 \|G(v_1)-G(v_2)\|_{\mathscr{L}_2(RK,\mathbb{H}_\gamma)}^2&=\sum_{j=1}^\infty\int_{\R^2}\big|(g(v_1(x))-g(v_2(x)))\widehat{\mathfrak{u}_j\mu}(x)\big|^2\gamma^{\vee}(x)\ud x\\
&\le\mu(\R^2)\int_{\R^2}\big|g(v_1(x))-g(v_2(x))\big|^2\gamma^{\vee}(x)\ud x\\
&\lesssim\|v_1-v_2\|_{\mathbb{H}_\gamma}^2,
\end{align*}
where we also used the Lipschitz continuity of $g$. Since $\gamma^{\vee}\in L^1(\R^2)$, it can be seen that $G$ maps $\mathbb{H}_\gamma$ into $\mathscr{L}_2(RK,\mathbb{H}_\gamma)$. Indeed,
\begin{align}\label{eq:Gv}
\|G(v_1)\|_{\mathscr{L}_2(RK,\mathbb{H}_\gamma)}^2
&=\mu(\R^2)\int_{\R^2}|g(v_1(x))|^2\gamma^{\vee}(x)\ud x\\\notag
&\lesssim\int_{\R^2}\left(1+|v_1(x)|^2\right)\gamma^{\vee}(x)\ud x\lesssim1+\|v_1\|_{\mathbb{H}_\gamma}^2\quad\forall~v_1\in \mathbb{H}_\gamma.
\end{align}

Taking advantage of  the Lipschitz continuity of $B$ and $G$, one has that
for any initial value $\psi\in \mathbb{H}_\gamma$, 
there exists a unique mild solution $u_\epsilon$ of \eqref{eq:SPDE} (see also \cite[Theorem 2.1]{PZ97}), i.e., for every $t\ge0$,
$$
u_\epsilon(t)=S_\epsilon(t)\psi+\int_0^t S_\epsilon(t-s) B(u_\epsilon(s)) \ud s+\int_0^t S_\epsilon(t-s) G(u_\epsilon(s)) \ud \mathcal{W}(s).
$$
Moreover,  for every $\epsilon>0$ and $p\ge1$ (cf. \cite[formula (5.8)]{CF19}),
\begin{align}\label{eq:uHp}
\mathbb{E}\bigg[\sup_{t\in[0,\mathsf{T}]}\|u_{\epsilon}(t)\|^p_{\mathbb{H}_\gamma}\bigg]\le C(p,\mathsf{T},c).
\end{align}
\subsection{SPDE on graph}\label{S:SPDE}

We rewrite the SPDE on graph \eqref{eq:utzk} into the following compact form
\begin{equation*}
\ud \bar{u}(t)=\bar{\mathcal L}\bar{u}(t)\ud t
 +B(\bar{u}(t))\ud t+G(\bar{u}(t))\ud\bar{\mathcal{W}}(t), \quad t\in[0,\mathsf T]
\end{equation*}
with the initial value $\bar{u}(0)=\psi^\wedge$. Here,
$B$ (resp. $G$) is the Nemytskii operator
associated with $b$ (resp. $g$), that is, for $\bar{v}_1,\bar{v}_2\in \bar{\mathbb{H}}_\gamma$,
\begin{align*}
B(v_1)(z,k)=b(v_1(z,k)),\quad G(v_1)(v_2)(z,k)=g(v_1(z,k))v_2(z,k),\quad (z,k)\in\Gamma.
\end{align*}
Concerning the noisy forcing $\bar{\mathcal{W}}$, it is defined by 
$$
\bar{\mathcal{W}}(t, z, k)=\sum_{j=1}^{\infty}(\widehat{\mathfrak{u}_j \mu})^{\wedge}(z, k) \beta_j(t), \quad t \geq 0 ,\quad(z, k) \in \Gamma
$$
so that $\mathcal{W}(t,\cdot)^{\wedge}(z,k)=\bar{\mathcal{W}}(t,z,k)$ formally. By the Karhunen--Loève expansion of $\bar{\mathcal{W}}$, the sequence $\{(\widehat{\mathfrak{u}_j \mu})^{\wedge}\}_{j=1}^\infty$ forms an orthonormal basis of the reproducing kernel space $\overline{RK}$ of $\bar{\mathcal{W}}$.
We note that the Burkholder inequality \eqref{eq:BDG} also holds with $RK$, $\mathbb{H}_\gamma$ and $\mathcal W$  replaced by $\overline{RK}$, $\bar{\mathbb{H}}_\gamma$ and $\bar{\mathcal W}$, respectively.

Using the Cauchy--Schwarz inequality, $(\widehat{\mathfrak{u}_j\mu})^\wedge(z,k)$=$\oint_{C_k(z)}\widehat{\mathfrak{u}_j\mu}(x)\ud\mu_{z,k}$ and
 \eqref{eq:ujmu0},
\begin{align}\label{eq:ujmu1} 
\sum_{j=1}^\infty|(\widehat{\mathfrak{u}_j\mu})^\wedge(z,k)|^2
\le \sum_{j=1}^\infty\oint_{C_k(z)}\left|\widehat{\mathfrak{u}_j\mu}(x)\right|^2\ud\mu_{z,k}=
\mu(\R^2),\quad (z,k)\in\Gamma,
\end{align}
since $\mu_{z,k}$ is a probability measure supported on $C_k(z)$. By \eqref{eq:ujmu1}, for any $\bar{v}_1,\bar{v}_2\in\bar{\mathbb{H}}_\gamma$,
\begin{align*}
&\quad 
\sum_{j=1}^{\infty}\|(G(\bar{v}_1)-G(\bar{v}_2))(\widehat{\mathfrak{u}_j\mu})^{\wedge}\|^2_{\bar{\mathbb{H}}_\gamma}\\
&=\sum_{j=1}^{\infty}\sum_{k=1}^m\int_{I_k}\left|(g(\bar{v}_1(z,k))-g(\bar{v}_2(z,k)))(\widehat{\mathfrak{u}_j\mu})^\wedge(z,k)\right|^2\gamma(z,k)T_k(z)\ud z\\
&\le\mu(\R^2)\sum_{k=1}^m\int_{I_k}|g(\bar{v}_1(z,k))-g(\bar{v}_2(z,k))|^2\gamma(z,k)T_k(z)\ud z\\
&\lesssim
\mu(\R^2)\sum_{k=1}^m\int_{I_k}|\bar{v}_1(z,k)-\bar{v}_2(z,k))|^2\gamma(z,k)T_k(z)\ud z=
\mu(\R^2)\|\bar{v}_1-\bar{v}_2\|_{\bar{\mathbb{H}}_\gamma}^2,
\end{align*}
due to the Lipschitz continuity of $g$, which proves the Lipschitz continuity of $G:\bar{\mathbb{H}}_{\gamma}\to\mathscr{L}_2(\overline{RK},\bar{\mathbb{H}}_{\gamma})$. In virtue of the Lipschitz continuity of $b$, the map $B:\bar{\mathbb{H}}_{\gamma}\to\bar{\mathbb{H}}_{\gamma}$ is also Lipschitz continuous.
Thanks to \eqref{eq:IkTk}, for any $\bar{v}\in\bar{\mathbb{H}}_\gamma$,
\begin{equation}\label{eq:barGlin-growth}
\sum_{j=1}^{\infty}\|G(\bar{v})(\widehat{\mathfrak{u}_j\mu})^{\wedge}\|^2_{\bar{\mathbb{H}}_\gamma}
\lesssim 1+\|\bar{v}\|_{\bar{\mathbb{H}}_\gamma}^2.
\end{equation}
By \eqref{eq:St} and the above Lipschitz continuity of $B$ and $G$,
there exists a unique mild solution $\bar{u}$ to \eqref{eq:utzk} in $L^p(\Omega, \mathcal C([0, \mathsf{T}] ; \bar{\mathbb{H}}_\gamma))$ for any $p \geq 1$. Namely, $\bar{u}$ is a predictable process satisfying that for any $t\in[0,\mathsf{T}]$,
\begin{equation}\label{eq:ut}
\bar{u}(t)=\bar{S}(t) \psi^{\wedge}+\int_0^t \bar{S}(t-s) B(\bar{u}(s)) \ud s+\int_0^t \bar{S}(t-s) G(\bar{u}(s)) \ud \bar{\mathcal{W}}(s).
\end{equation}
Moreover,  for every $p\ge1$ (cf. \cite[formula (5.13)]{CF19}),
\begin{align*}
\mathbb{E}\bigg[\sup_{t\in[0,\mathsf{T}]}\|\bar{u}(t)\|^p_{\bar{\mathbb{H}}_\gamma}\bigg]\le C(p,\mathsf{T},c).
\end{align*}

\section{Exponential Euler approximation}\label{S3}
In the absence of the explicit expression of the particle density in stochastic incompressible flow, it is important to resort to numerical approximations to understand and simulate the stochastic RDA equation \eqref{eq:SPDE}. Considering the multiscale nature, an effective numerical approximation should not only be consistent with \eqref{eq:SPDE} but also enable the numerical investigation of the asymptotical behavior \eqref{eq:u-u} as $\epsilon$ tends to $0$.
We notice that the asymptotical behavior \eqref{eq:u-u} highly relies on the 
asymptotical behavior of the semigroup $S_\epsilon(t)$. To fully exploit the properties of the semigroup $S_\epsilon(t)$, we discretize \eqref{eq:SPDE} in time direction by the exponential Euler approximation.
 As a popular time-discrete approximation of SPDEs,
the exponential Euler approximation not only does not require a restriction on stepsize such as the CFL-type condition, but also has implementation advantages over the implicit scheme (see e.g., \cite{HO10,ACQ20}). 

In this section, we focus on the error estimate of the exponential Euler approximation \eqref{eq:EEM} for \eqref{eq:SPDE}, to clarify the computational accuracy of \eqref{eq:EEM} and to track how errors propagate across different scales $0<\epsilon\ll1$.
Note that the approximation error between $u_{\epsilon}(t_n)$ and $U_{\epsilon}^{n}$ satisfies that for any $1\le n\le N$,
\begin{align*}
u_{\epsilon}(t_n)-U_{\epsilon}^{n}&
=\sum_{j=0}^{n-1}\int_{t_{j}}^{t_{j+1}}\left(S_\epsilon(t_n-s)-S_\epsilon(t_n-t_j)\right)B(u_{\epsilon}(s))\ud s\\
&\quad+\sum_{j=0}^{n-1}\int_{t_{j}}^{t_{j+1}}S_\epsilon(t_n-t_j)\left(B(u_{\epsilon}(s))-B(u_{\epsilon}(t_{j}))\right)\ud s\\
&\quad+\sum_{j=0}^{n-1}\int_{t_{j}}^{t_{j+1}}S_\epsilon(t_n-t_j)\left(B(u_{\epsilon}(t_{j}))-B(U_{\epsilon}^{j})\right)\ud s\\
&\quad+\sum_{j=0}^{n-1}\int_{t_{j}}^{t_{j+1}}\left(S_\epsilon(t_n-s)-S_\epsilon(t_n-t_j)\right)G(u_{\epsilon}(s))\ud\mathcal{W}(s)\\
&\quad+\sum_{j=0}^{n-1}\int_{t_{j}}^{t_{j+1}}S_\epsilon(t_n-t_j)\left(G(u_{\epsilon}(s))-G(u_{\epsilon}(t_{j}))\right)\ud\mathcal{W}(s)\\
&\quad+\sum_{j=0}^{n-1}\int_{t_{j}}^{t_{j+1}}S_\epsilon(t_n-t_j)\left(G(u_{\epsilon}(t_{j}))-G(U_{\epsilon}^{j})\right)\ud\mathcal{W}(s).
\end{align*}
By \eqref{eq:Seps}, the Burkholder inequality, and the Lipschitz continuity of $B$ and $G$,
\begin{align}\label{eq:u-U}
&\mathbb{E}\left[\|u_{\epsilon}(t_n)-U_{\epsilon}^{n}\|^2_{\mathbb{H}_\gamma}\right]\\\notag
&\lesssim \sum_{j=0}^{n-1}\int_{t_{j}}^{t_{j+1}}\mathbb{E}\left[\left\|\left(S_\epsilon(t_n-s)-S_\epsilon(t_n-t_j)\right)B(u_{\epsilon}(s))\right\|^2_{\mathbb{H}_\gamma}\right]\ud s\\\notag
&\quad+ \sum_{j=0}^{n-1}\int_{t_{j}}^{t_{j+1}}\mathbb{E}\left[\left\|\left(S_\epsilon(t_n-s)-S_\epsilon(t_n-t_j)\right)G(u_{\epsilon}(s))\right\|^2_{\mathscr{L}_2(RK,\mathbb{H}_\gamma)}\right]\ud s\\\notag
&\quad+\sum_{j=0}^{n-1}\int_{t_{j}}^{t_{j+1}}\mathbb{E}\left[\|u_{\epsilon}(s)-u_{\epsilon}(t_{j})\|_{\mathbb{H}_\gamma}^2\right]\ud s+\sum_{j=0}^{n-1}\int_{t_{j}}^{t_{j+1}}\mathbb{E}\left[\|u_{\epsilon}(t_{j})-U_{\epsilon}^{j}\|_{\mathbb{H}_\gamma}^2\right]\ud s.
\end{align}

From the above analysis, the mean square convergence error of \eqref{eq:EEM}
depends on the temporal H\"older regularity estimates of the semigroup $\{S_\epsilon(t)\}_{t\in[0,\mathsf T]}$ and the solution $\{u_\epsilon(t)\}_{t\in[0,\mathsf T]}$.
These will be handled in Propositions \ref{prop:S-I-1} and \ref{prop:Holder}, respectively.

\subsection{Regularity of semigroup}
To derive the  H\"older regularity estimate of $\{S_\epsilon(t)\}_{t\in[0,\mathsf T]}$, we
first present a priori estimates of the corresponding linear equation $$\partial_t S_\epsilon(t)\varphi=\mathcal L_\epsilon S_\epsilon(t)\varphi,\quad t\in[0,\mathsf{T}]$$
 in the weighted space $\mathbb{H}_\gamma$
with the initial value $\varphi:\R^2\to \R$. It should be noticed that the following gradient estimate of $S_\epsilon(t)$ highly depends on the skew-symmetry of the symplectic matrix and the divergence-free property of $\nabla^{\perp} H(x)$.

\begin{lem}\label{prop:Lrho2}
Let Assumption \ref{Asp:gamma} hold. Then for any $\epsilon\in(0,1]$ and $t\in[0,\mathsf{T}]$.
 \begin{align}\label{eq:UL2rho}
&\|S_\epsilon(t)\varphi\|_{\mathbb{H}_\gamma}^2+\int_0^t\|\nabla S_\epsilon(t)\varphi\|_{\mathbb{H}_\gamma}^2\ud s\lesssim\|\varphi\|_{\mathbb{H}_\gamma}^2\quad\forall~\varphi\in \mathbb{H}_\gamma.
\end{align}
Moreover, for any
$\epsilon\in(0,1)$ and $t\in[0,\mathsf{T}]$,
\begin{align}\label{eq:D2Urho}
\|\nabla S_\epsilon(t)\varphi\|^2_{\mathbb{H}_\gamma}+\int_0^t\|\nabla^2S_\epsilon(t)\varphi\|^2_{\mathbb{H}_\gamma}\ud s
\lesssim\|\nabla\varphi\|^2_{\mathbb{H}_\gamma}+\epsilon^{-1}\|\varphi\|^2_{\mathbb{H}_\gamma},
\end{align}
provided that the right-hand side is finite, where $\nabla^2S_\epsilon(t)\varphi$ denotes the Hessian matrix of $S_\epsilon(t)\varphi$.

\end{lem}
\begin{proof}
By the relations $\nabla{|\phi|^2}=2\phi \nabla \phi$ and
$2\phi\Delta \phi=\Delta(\phi^2)-2|\nabla \phi|^2$,
we obtain that $U_\varphi(t):=S_\epsilon(t)\varphi$ satisfies
\begin{align}\label{eq:Uphi}
\partial_t|U_\varphi(t)|^2&=\frac12\Delta|U_\varphi(t)|^2-|\nabla U_\varphi(t)|^2+\frac{1}{\epsilon}\langle \nabla^\perp H,\nabla |U_\varphi(t)|^2\rangle,
\end{align}
and
\begin{align}\label{eq:Uphi1}
\partial_t|\nabla U_\varphi(t)|^2
&=\frac{1}{2}\Delta|\nabla U_\varphi(t)|^2-|\nabla^2U_\varphi(t)|^2+\frac{2}{\epsilon}\sum_j\partial_jU_\varphi(t)\langle \partial_j\nabla^\perp H,\nabla U_\varphi(t)\rangle\\\notag
&\quad +\frac{1}{\epsilon}\langle \nabla^\perp H,\nabla| \nabla U_\varphi(t)|^2\rangle.
\end{align}
 
Thanks to Assumption \ref{Asp:gamma}, $\gamma^{\vee}=\vartheta\circ H$, which ensures that
$\langle \nabla^\perp H,\nabla\gamma^{\vee}\rangle=\langle \nabla^\perp H,\vartheta^\prime(H)\nabla H\rangle=0$ and $\textrm{div}(\nabla^\perp H)=0$. Hence
\begin{align}\label{eq:perpH0}
\int_{\R^2}\langle \nabla^\perp H,\nabla \chi\rangle\gamma^{\vee}\ud x=-\int_{\R^2}\textrm{div}( \nabla^\perp H) \chi\gamma^{\vee}\ud x-\int_{\R^2}\langle \nabla^\perp H,\nabla\gamma^{\vee}\rangle\chi\ud x=0
\end{align}
for any suitable regular $\chi:\R^2\to \R$. Integrating with respect to the measure $\gamma^{\vee}(x)\ud x$ on both sides of \eqref{eq:Uphi} and then using 
 \eqref{eq:perpH0} with $\chi=|U_\varphi(t)|^2$, 
\begin{align}\label{eq:Uphirho}
&\quad\ \frac{\ud}{\ud t}\int_{\R^2}|U_\varphi(t)|^2\gamma^{\vee}\ud x+\int_{\R^2}|\nabla U_\varphi(t)|^2\gamma^{\vee}\ud x\\\notag
&=\frac12\int_{\R^2}\Delta|U_\varphi(t)|^2\gamma^{\vee}\ud x=\frac12\int_{\R^2}|U_\varphi(t)|^2\Delta\gamma^{\vee}\ud x\le \frac12c\int_{\R^2}|U_\varphi(t)|^2\gamma^{\vee}\ud x
\end{align}
due to  $\Delta \gamma^{\vee}\le c\gamma^{\vee}$.
Then the Gronwall inequality leads to \eqref{eq:UL2rho}.

Similarly, according to \eqref{eq:Uphi1} and \eqref{eq:perpH0} with $\chi=|\nabla U_\varphi(t)|^2$,
\begin{align}\label{eq:Uh1rho}
&\quad\ \frac{\ud}{\ud t}\int_{\R^2}|\nabla U_\varphi(t)|^2\gamma^{\vee}\ud x\\\notag
&\le \frac12\int_{\R^2}|\nabla U_\varphi(t)|^2\Delta\gamma^{\vee}\ud x-\int_{\R^2}|\nabla^2U_\varphi(t)|^2\gamma^{\vee}\ud x+C\epsilon^{-1}\int_{\R^2}|\nabla U_\varphi(t)|^2\gamma^{\vee}\ud x\\\notag
&\le \frac12c\|\nabla U_\varphi(t)\|^2_{\mathbb{H}_\gamma}-\|\nabla^2U_\varphi(t)\|^2_{\mathbb{H}_\gamma}+C\epsilon^{-1}\|\nabla U_\varphi(t)\|^2_{\mathbb{H}_\gamma},
\end{align}
since the second-order derivative of $H$ is bounded.
 Further, applying the Gronwall inequality to \eqref{eq:Uh1rho} and then using \eqref{eq:UL2rho} give that for any $t\in[0,\mathsf T]$,
\begin{align*}
\|\nabla U_\varphi(t)\|^2_{\mathbb{H}_\gamma}
&\le e^{\frac12 ct}\|\nabla\varphi\|^2_{\mathbb{H}_\gamma}+\int_0^te^{\frac12 c(t-s)}C\epsilon^{-1}\|\nabla U_\varphi(s)\|^2_{\mathbb{H}_\gamma}\ud s\\\notag
&\le e^{\frac12 ct}\|\nabla\varphi\|^2_{\mathbb{H}_\gamma}+e^{\frac12 c\mathsf{T}}C\epsilon^{-1}\int_0^t\|\nabla U_\varphi(s)\|^2_{\mathbb{H}_\gamma}\ud s\\\notag
&\lesssim\|\nabla\varphi\|^2_{\mathbb{H}_\gamma}+\epsilon^{-1}\|\varphi\|^2_{\mathbb{H}_\gamma}.
\end{align*}
Integrating \eqref{eq:Uh1rho} with respect to time completes the proof of \eqref{eq:D2Urho}.
\end{proof}

In addition to $\gamma$, we also consider another graph weighted function $\sqrt{\gamma}:\Gamma\to (0,+\infty)$ given by 
$\sqrt{\gamma}(z,k)=\sqrt{\vartheta(z)}$ for every $(z,k)\in\Gamma$. The following assumption will be used to analyze the convergence rate of the proposed approximations in this section and section \ref{S5}.

\begin{asp}\label{Asp:gamma-S}
Let $\gamma,\vartheta$ satisfy Assumption \ref{Asp:gamma} and $z^2\vartheta(z)$ be uniformly bounded for $z\in[0,\infty)$. Moreover, Assumption \ref{Asp:gamma} and \eqref{eq:IkTk} hold with $\gamma$ and $\vartheta$ replaced by $\sqrt{\gamma}$ and $\sqrt{\vartheta}$, respectively.

\end{asp}
\begin{rem}\label{rem:A3}
\begin{enumerate}
\item[(i)]
Notice that in  Example \ref{Ex:gamma-S1} with $\lambda>2$ or Example \ref{Ex:gamma-S},
  $\sqrt{\vartheta}$ shares similar properties as $\vartheta$, namely, $\sqrt{\vartheta}:[0,\infty)\to(0,\infty)$ is second order differentiable and $\sqrt{\vartheta(z)}=\sqrt{c_0}e^{-\frac12\lambda(\sqrt{z}-\sqrt{2z_0})}$ or $\sqrt{\vartheta(z)}=\sqrt{c_0}z^{-\frac\lambda2}$ for $z\ge z_0$. Hence the  functions $\gamma$ in Example \ref{Ex:gamma-S1} with $\lambda>2$ and Example \ref{Ex:gamma-S} satisfy Assumption \ref{Asp:gamma-S}. 
  
 \item[(ii)] We note that under Assumption \ref{Asp:gamma-S}, the estimates under the norm $\|\cdot\|_{\mathbb{H}_\gamma}$ with $\gamma$ satisfying Assumption \ref{Asp:gamma}, such as \eqref{eq:UL2rho} and \eqref{eq:D2Urho}, also hold with $\gamma$ replaced by $\sqrt\gamma$.
 \end{enumerate}
 \end{rem}

With slight abuse of notation, we use the notation $\mathbb{H}_{\sqrt\gamma}$ (resp.\ $\bar{\mathbb{H}}_{\sqrt\gamma}$) to denote the space $\mathbb{H}_{\gamma}$ (resp.\ $\bar{\mathbb{H}}_{\gamma}$) with $\gamma$ replaced by $\sqrt\gamma$. Under Assumption \ref{Asp:gamma-S}, the function $\vartheta$ is uniformly bounded on $[0,\infty)$, which implies $\mathbb{H}_{\sqrt{\gamma}}\subset \mathbb{H}_\gamma$. Indeed, 
 $$\|\varphi\|^2_{\mathbb{H}_\gamma}=\int_{\R^2}|\varphi(x)|^2\vartheta(H(x))\ud x\lesssim\int_{\R^2}|\varphi(x)|^2\sqrt{\vartheta(H(x))}\ud x\lesssim \|\varphi\|^2_{\mathbb{H}_{\sqrt{\gamma}}}\quad\forall~\varphi\in \mathbb{H}_{\sqrt{\gamma}}.$$

We are ready to give the temporal H\"older continuity of the semigroup $\{S_\epsilon(t)\}_{t\in[0,\mathsf{T}]}$.

\begin{prop}\label{prop:S-I-1}
Under Assumption \ref{Asp:gamma-S},
for any $\epsilon\in(0,1]$ and $t\in(0,\mathsf{T}]$,
\begin{align*}
 \|(S_\epsilon(t)-I)\varphi\|_{\mathbb{H}_\gamma}
 \lesssim t^{\frac12}\left(\|\nabla\varphi\|_{\mathbb{H}_\gamma}+\epsilon^{-\frac12}\|\varphi\|_{\mathbb{H}_\gamma}+\epsilon^{-1} \|\varphi\|_{\mathbb{H}_{\sqrt\gamma}}\right),
 \end{align*}
 provided that the right-hand side is finite.
\end{prop}

\begin{proof}
Under Assumption \ref{Asp:gamma-S},
 $z\sqrt{\vartheta(z)}$ is uniformly bounded for $z\in[0,\infty)$, and thus
 the map $x\mapsto |\nabla H(x)|^2\sqrt{\vartheta(H(x))}$ is bounded in $\R^2$. Hence one has
 \begin{equation*}
|\nabla^{\perp} H(x)|^2\gamma^{\vee}(x)
=\sqrt\gamma^{\vee}(x)|\nabla H(x)|^2\sqrt{\vartheta(H(x))}
\lesssim\sqrt\gamma^{\vee}(x),
\end{equation*}
from which it follows that $U_\varphi(t)=S_\epsilon(t)\varphi$ satisfies
\begin{align}\label{eq:nablaUL2}
\|\langle \nabla^\perp H,\nabla U_\varphi(s) \rangle \|^2_{\mathbb{H}_\gamma}&=\int_{\R^2}|\langle \nabla^\perp H,\nabla U_\varphi(s) \rangle|^2\gamma^{\vee}(x)\ud x\\\notag
&\lesssim \int_{\R^2}\sqrt\gamma^{\vee}(x)|\nabla U_\varphi(s)|^2\ud x\lesssim\|\nabla U_\varphi(s)\|^2_{\mathbb{H}_{\sqrt\gamma}}.
 \end{align}
 In view of \eqref{eq:UL2rho} with $\gamma$ replaced by $\sqrt\gamma$, \eqref{eq:D2Urho}, and \eqref{eq:nablaUL2}, for any $t\in(0,\mathsf{T}]$,
\begin{align*}
 \|(S_\epsilon(t)-I)\varphi\|_{\mathbb{H}_\gamma}&\le \int_0^t\|\frac{1}{2}\Delta U_\varphi(s)+\frac{1}{\epsilon}\langle \nabla^\perp H,\nabla U_\varphi(s) \rangle \|_{\mathbb{H}_\gamma}\ud s\\
 &\lesssim t^{\frac12}\left(\int_0^t\|\Delta U_\varphi(s)\|^2_{\mathbb{H}_\gamma}\ud s\right)^{\frac12}+\frac{1}{\epsilon}t^{\frac12} \left(\int_0^t\|\nabla U_\varphi(s)\|^2_{\mathbb{H}_{\sqrt\gamma}}\ud s\right)^{\frac12}\\
 &\lesssim t^{\frac12}\left(\|\nabla\varphi\|^2_{\mathbb{H}_\gamma}+\epsilon^{-1}\|\varphi\|^2_{\mathbb{H}_\gamma}\right)^{\frac12}+\epsilon^{-1}t^{\frac12} \|\varphi\|_{\mathbb{H}_{\sqrt\gamma}},
 \end{align*}
 which completes the proof.  
 \end{proof}

\subsection{Approximation error for stochastic RDA equation}
In this part, we present the convergence analysis of the exponential Euler approximation \eqref{eq:EEM} for the stochastic RDA equation \eqref{eq:SPDE}. 
For preparation, we begin with the gradient estimate of \eqref{eq:SPDE}, which shows that the second moment of the gradient of the solution depends on $\epsilon^{-1}$ linearly.

\begin{lem}\label{lem:nablau}
Let $b,g$ be continuously differentiable with bounded derivatives and Assumption \ref{Asp:gamma} hold. Suppose that
 $\psi,\nabla\psi\in \mathbb{H}_\gamma $ and 
 \begin{equation}\label{eq:noise}
\int_{\R^2}|\xi|^2\mu(\ud \xi)< \infty.
\end{equation}
 Then for any $\epsilon\in(0,1]$ and  $t\in[0,\mathsf T]$, 
 \begin{align}\label{eq:nablau}
 \E\left[\|\nabla u_\epsilon(t)\|^2_{\mathbb{H}_\gamma}\right]+\int_0^t\E\left[\|\nabla^2u_\epsilon(r)\|^2_{\mathbb{H}_\gamma}\right]\ud r
\lesssim 1+\epsilon^{-1}.
\end{align}
 \end{lem}
\begin{proof}
Applying the It\^o formula (see e.g., \cite[Theorem 4.32]{DP14}) to $\|u_\epsilon(t)\|_{\mathbb{H}_\gamma}^2$ and using a similar argument of \eqref{eq:Uphirho}, we obtain
\begin{align*}
\| u_\epsilon(t)\|^2_{\mathbb{H}_\gamma}&=\| \psi\|^2_{\mathbb{H}_\gamma}+2\int_0^t\Big\langle u_\epsilon(r), \frac12\Delta u_{\epsilon}(r)+\frac{1}{\epsilon}\langle \nabla^{\perp} H(x),\nabla u_{\epsilon}(r)\rangle\Big\rangle_{\mathbb{H}_\gamma}\ud r\\\notag
&\quad+2\int_0^t\left\langle u_\epsilon(r), B(u_\epsilon(r))\right\rangle_{\mathbb{H}_\gamma}\ud r+2\int_0^t\Big\langle u_\epsilon(r), G(u_\epsilon(r))\ud \mathcal W(r)\Big\rangle_{\mathbb{H}_\gamma}\\\notag
&\quad+\int_0^t\|G(u_\epsilon(r))\|_{\mathscr{L}_2(RK,\mathbb{H}_\gamma)}^2\ud r\\
&\le\| \psi\|^2_{\mathbb{H}_\gamma}+\frac12c\int_0^t\|u_\epsilon(r)\|^2_{\mathbb{H}_\gamma}\ud r-\int_0^t\|\nabla u_\epsilon(r)\|^2_{\mathbb{H}_\gamma}\ud r\\\notag
&\quad+C\int_0^t1+ \|u_\epsilon(r)\|^2_{\mathbb{H}_\gamma}\ud r+2\int_0^t\Big\langle u_\epsilon(r), G(u_\epsilon(r))\ud \mathcal W(r)\Big\rangle_{\mathbb{H}_\gamma},
\end{align*}
thanks to the linear growth of $b,g$ and \eqref{eq:Gv}.
Taking expectations on both sides of the above 
inequality and using \eqref{eq:uHp}, it follows that for any $t\in[0,\mathsf{T}]$,
\begin{equation}\label{eq:D1uHp}
\int_0^t\E\left[\|\nabla u_\epsilon(r)\|^2_{\mathbb{H}_\gamma}\right]\ud r\le C(\mathsf T).
\end{equation}
Then applying the It\^o formula to $\|\nabla u_\epsilon(t)\|^2_{\mathbb{H}_\gamma}$ yields
\begin{align*}
\|\nabla u_\epsilon(t)\|^2_{\mathbb{H}_\gamma}-\|\nabla \psi\|^2_{\mathbb{H}_\gamma}&=2\int_0^t\Big\langle\nabla u_\epsilon(r), \nabla\big(\frac12\Delta u_{\epsilon}(r)+\frac{1}{\epsilon}\langle \nabla^{\perp} H(x),\nabla u_{\epsilon}(r)\rangle\big)\Big\rangle_{\mathbb{H}_\gamma}\ud r\\\notag
&\quad+2\int_0^t\left\langle\nabla u_\epsilon(r), \nabla B(u_\epsilon(r))\right\rangle_{\mathbb{H}_\gamma}\ud r\\\notag
&\quad+2\int_0^t\Big\langle\nabla u_\epsilon(r), \nabla\big(G(u_\epsilon(r))\ud \mathcal W(r)\big)\Big\rangle_{\mathbb{H}_\gamma}\ud r\\\notag
&\quad+\int_0^t\sum_{j=1}^\infty\|\nabla(G(u_\epsilon(r))\widehat{\mathfrak{u}_j\mu})\|_{\mathbb{H}_\gamma}^2\ud r=:I+II+III+IV.
\end{align*}
Repeating the derivation of \eqref{eq:Uh1rho}, it can be shown that
\begin{align*}
I\le \frac12c\int_0^t\|\nabla u_\epsilon(r)\|^2_{\mathbb{H}_\gamma}\ud r-\int_0^t\|\nabla^2u_\epsilon(r)\|^2_{\mathbb{H}_\gamma}\ud r+C\epsilon^{-1}\int_0^t\|\nabla u_\epsilon(r)\|^2_{\mathbb{H}_\gamma}\ud r.
\end{align*}
By the boundedness of $b^\prime$, one has
$$II=2\int_0^t\left\langle\nabla u_\epsilon(r), b^\prime(u_\epsilon(r))\nabla u_\epsilon(r)\right\rangle_{\mathbb{H}_\gamma}\ud r\lesssim\int_0^t\left\|\nabla u_\epsilon(r)\right\|_{\mathbb{H}_\gamma}^2\ud r.$$

Replacing $g(v_1(x))$ in \eqref{eq:Gv} by $\nabla g(u_{\epsilon}(r))$ yields
\begin{align*}
\sum_{j=1}^\infty\int_{\R^2}| \nabla g(u_{\epsilon}(r))\widehat{\mathfrak{u}_j\mu}|^2\gamma^{\vee}\ud x\lesssim \| \nabla g(u_{\epsilon}(r))\|^2_{\mathbb{H}_\gamma}\lesssim \| \nabla u_{\epsilon}(r)\|^2_{\mathbb{H}_\gamma}.
\end{align*}
Since $\{\mathfrak{u}_j\}_{j\in\mathbb{N}_+}$ is a complete orthonormal basis of the Hilbert space $L^2_{(s)}(\R^2,\ud \mu)$,
\begin{align*}
\sum_{j=1}^\infty\int_{\R^2}|g(u_{\epsilon}(r)) \nabla\widehat{\mathfrak{u}_j\mu}|^2\gamma^{\vee}\ud x
&=\sum_{j=1}^\infty\int_{\R^2}\Big|g(u_{\epsilon}(r)) \int_{\R^2}\textup{i}\xi e^{\textup{i}x\cdot\xi} \mathfrak{u}_j(\xi)\mu(\ud\xi)\Big|^2\gamma^{\vee}(x)\ud x\\
&=\int_{\R^2}\int_{\R^2}|g(u_{\epsilon}(r,x)) \textup{i}\xi e^{-\textup{i}x\cdot\xi} |^2\mu(\ud \xi)\gamma^{\vee}(x)\ud x\\
&\lesssim \int_{\R^2}|g(u_{\epsilon}(r,x)) |^2\int_{\R^2}|\xi|^2\mu(\ud \xi)\gamma^{\vee}(x)\ud x\\
&\lesssim 1+\|u_{\epsilon}(r)\|_{\mathbb{H}_\gamma}^2,
\end{align*}
thanks to \eqref{eq:noise}. 
Utilizing the boundedness of $g^\prime$, one obtains
\begin{align}\label{eq:Gu}
&\sum_{j=1}^\infty\|\nabla(G(u_{\epsilon}(r))\widehat{\mathfrak{u}_j\mu})\|^2_{\mathbb{H}_\gamma}=\sum_{j=1}^\infty\int_{\R^2}|\nabla(g(u_{\epsilon}(r)) \widehat{\mathfrak{u}_j\mu})|^2\gamma^{\vee}\ud x\\\notag
&\lesssim\sum_{j=1}^\infty\int_{\R^2}| \nabla g(u_{\epsilon}(r))\widehat{\mathfrak{u}_j\mu}|^2\gamma^{\vee}\ud x+\sum_{j=1}^\infty\int_{\R^2}|g(u_{\epsilon}(r)) \nabla\widehat{\mathfrak{u}_j\mu}|^2\gamma^{\vee}\ud x\\\notag
&
\lesssim \|\nabla u_\epsilon(r) \|_{\mathbb{H}_\gamma}^2+1+\|u_\epsilon(r)\|_{\mathbb{H}_\gamma}^2.
\end{align}
Collecting the above estimates, we obtain \eqref{eq:nablau} from \eqref{eq:uHp} and \eqref{eq:D1uHp}.
\end{proof}

Utilizing Proposition \ref{prop:S-I-1} and Lemma \ref{lem:nablau}, we obtain the temporal H\"older continuity of the solution to the stochastic RDA equation \eqref{eq:SPDE}. 
 
 \begin{prop}\label{prop:Holder}
Let $b,g$ be continuously differentiable with bounded derivatives and Assumption \ref{Asp:gamma-S} hold. If
 $\psi\in \mathbb{H}_{\sqrt\gamma},\nabla\psi\in \mathbb{H}_\gamma $ and \eqref{eq:noise} hold, then for any $\epsilon\in(0,1]$,
$$\mathbb{E}\left[\|u_{\epsilon}(t)-u_{\epsilon}(s)\|^2_{\mathbb{H}_\gamma}\right]\lesssim (t-s)(1+\epsilon^{-2})\quad\forall~0\le s<t\le \mathsf{T}.$$
\end{prop}
\begin{proof}
For $0\le s<t\le \mathsf{T}$,
\begin{align*}
u_{\epsilon}(t)-u_{\epsilon}(s)&=(S_\epsilon(t-s)-I)u_{\epsilon}(s)+\int_s^tS_\epsilon(t-r)B(u_{\epsilon}(r))\ud r\\
&\quad+\int_s^tS_\epsilon(t-r)G(u_{\epsilon}(r))\ud\mathcal{W}(r).
\end{align*}
By 
$\psi\in  \mathbb{H}_{\sqrt\gamma}\subset\mathbb{H}_\gamma$, \eqref{eq:uHp} and Assumption \ref{Asp:gamma-S} (see Remark \ref{rem:A3}(ii)), we obtain 
\begin{equation}\label{eq:ulambda}
\mathbb{E}\left[\|u_{\epsilon}(s) \|_{\mathbb{H}_\gamma}^2\right]+\mathbb{E}\left[\|u_{\epsilon}(s) \|^2_{\mathbb{H}_{\sqrt\gamma}}\right]\le C(\mathsf{T},\lambda,c_0).
\end{equation}
Taking advantage of \eqref{eq:Seps}, the Burkholder inequality, and \eqref{eq:ulambda},
\begin{align*}
 \mathbb{E}\left[\left\|\int_s^tS_\epsilon(t-r)G(u_{\epsilon}(r))\ud\mathcal{W}(r)\right\|^2_{\mathbb{H}_\gamma}\right]
\lesssim \int_s^t1+\E\left[\|u_{\epsilon}(r)\|^2_{\mathbb{H}_\gamma}\right]\ud r\lesssim t-s.
\end{align*}
By the H\"older inequality, \eqref{eq:Seps}, and \eqref{eq:ulambda},
\begin{align*}
&\quad\ \mathbb{E}\left[\left\|\int_s^tS_\epsilon(t-r)B(u_{\epsilon}(r))\ud r\right\|^2_{\mathbb{H}_\gamma}\right]
\lesssim (t-s) \int_s^t1+\E\left[\|u_{\epsilon}(r)\|^2_{\mathbb{H}_\gamma}\right]\ud r\lesssim (t-s)^2.
\end{align*}
In view of Proposition \ref{prop:S-I-1}, Lemma \ref{lem:nablau}, and \eqref{eq:ulambda},
\begin{align*}
&\E\left[ \|(S_\epsilon(t-s)-I)u_{\epsilon}(s)\|^2_{\mathbb{H}_\gamma}\right]\\\notag
 &\lesssim (t-s)\left(\E[\|\nabla u_{\epsilon}(s)\|^2_{\mathbb{H}_\gamma}]+\epsilon^{-1}\E[\|u_{\epsilon}(s)\|^2_{\mathbb{H}_\gamma}]+\epsilon^{-2} \E[\|u_{\epsilon}(s)\|^2_{\mathbb{H}_{\sqrt\gamma}}]\right)\\
 & \lesssim (t-s)(1+\epsilon^{-2} ).
\end{align*}
Combining the previous three inequalities, we obtain the required result.
\end{proof}

We end this section by proving the mean square convergence order $\frac12$ of the exponential Euler approximation \eqref{eq:EEM} for the stochastic RDA equation \eqref{eq:SPDE}, where the error bound depends on $\epsilon^{-1}$ linearly.

\begin{tho}\label{theo:MS-2D}
Let $b,g$ be continuously differentiable with bounded derivatives and Assumption \ref{Asp:gamma-S} hold. If
 $\psi\in \mathbb{H}_{\sqrt\gamma},\nabla\psi\in \mathbb{H}_\gamma $ and \eqref{eq:noise} hold, then for any $\epsilon\in(0,1]$, 
 \begin{align*}
\mathbb{E}\left[\|u_{\epsilon}(t_n)-U_{\epsilon}^{n}\|^2_{\mathbb{H}_\gamma}\right]
\lesssim \tau(1+\epsilon^{-2}),\quad n=0,1,\ldots,N.
\end{align*}
\end{tho}
\begin{proof}
Taking advantage of \eqref{eq:Seps} and the semigroup property of $\{S_\epsilon(t)\}_{t\in[0,\mathsf T]}$,
\begin{align*}
& \mathbb{E}\left[\|\left(S_\epsilon(t_n-s)-S_\epsilon(t_n-t_j)\right)G(u_{\epsilon}(s))\|^2_{\mathscr{L}_2(RK,\mathbb{H}_\gamma)}\right]\\
&\lesssim \sum_{l=1}^\infty\mathbb{E}\left[\left\|\left(S_\epsilon(s-t_j)-I\right)G(u_{\epsilon}(s))\widehat{\mathfrak{u}_l\mu}\right\|^2_{\mathbb{H}_\gamma}\right]=:J.
\end{align*}
Proposition \ref{prop:S-I-1}, \eqref{eq:Gu} and \eqref{eq:Gv} 
indicate that 
\begin{align*} 
J&\lesssim (s-t_j)\sum_{l=1}^\infty\mathbb{E}\left[\|\nabla(G(u_{\epsilon}(s))\widehat{\mathfrak{u}_l\mu})\|^2_{\mathbb{H}_\gamma}\right]+\frac{1}{\epsilon}(s-t_j)\mathbb{E}\left[\|G(u_{\epsilon}(s))\|_{\mathscr{L}_2(RK,\mathbb{H}_\gamma)}^2\right]\\
&\quad+\frac{1}{\epsilon^2}(s-t_j)\mathbb{E}\left[\|G(u_{\epsilon}(s)) \|_{\mathscr{L}_2(RK,\mathbb{H}_{\sqrt\gamma})}^2\right]\\
&\lesssim (s-t_j)\mathbb{E}\left[\|\nabla u_{\epsilon}(s)\|^2_{\mathbb{H}_\gamma}\right]+\left(1+\frac{1}{\epsilon}\right)(s-t_j)\left(\mathbb{E}\left[\|u_{\epsilon}(s)\|^2_{\mathbb{H}_\gamma}\right]+1\right)\\
&\quad+\frac{1}{\epsilon^2}(s-t_j)\left(\mathbb{E}\left[\|u_{\epsilon}(s) \|^2_{\mathbb{H}_{\sqrt\gamma}}\right]+1\right).
\end{align*}
By \eqref{eq:ulambda} and Lemma \ref{lem:nablau}, we have that for any $s\in(t_j,t_{j+1}]$ with $j=0,1,\ldots,N-1$,
$$J\lesssim(1+\epsilon^{-2})(s-t_j)\lesssim (1+\epsilon^{-2})\tau.$$
Similarly, by virtue of \eqref{eq:Seps}, Proposition \ref{prop:S-I-1}, Lemma \ref{lem:nablau}, and \eqref{eq:ulambda},
\begin{align*}
 \mathbb{E}\left[\|\left(S_\epsilon(t_n-s)-S_\epsilon(t_n-t_j)\right)B(u_{\epsilon}(s))\|^2_{\mathbb{H}_\gamma}\right]
\lesssim (1+\epsilon^{-2})\tau
\end{align*}
for any $s\in(t_j,t_{j+1}]$ with $j=0,1,\ldots,N-1$.
Inserting the above estimates into 
\eqref{eq:u-U} and using Proposition \ref{prop:Holder}, we arrive at
\begin{align*}
\mathbb{E}\left[\|u_{\epsilon}(t_n)-U_{\epsilon}^{n}\|^2_{\mathbb{H}_\gamma}\right]\lesssim \tau(1+\epsilon^{-2})
+\tau\sum_{j=0}^{n-1}\mathbb{E}\left[\|u_{\epsilon}(t_{j})-U_{\epsilon}^{j}\|^2_{\mathbb{H}_\gamma}\right],\end{align*}
which along with the discrete Gronwall inequality leads to the desired result.
 \end{proof}

\section{Fast advection asymptotics}\label{S4}
It has been shown in \cite[Theorem 5.3]{CF19} that under Assumption \ref{Asp:gamma} and the condition \eqref{eq:Tneq0}, the limiting process of $u_\epsilon^\wedge$, as $\epsilon\to 0$,
is the solution $\bar{u}$ of the SPDE \eqref{eq:utzk} on the graph $\Gamma$. Specifically,
 for any $\psi\in \mathbb{H}_\gamma$, the asymptotical behavior \eqref{eq:u-u} of the stochastic RDA equation \eqref{eq:SPDE} holds. As shown in section \ref{S3}, for a fixed $\epsilon>0$, the exponential Euler approximation \eqref{eq:EEM} converges in the mean square sense to \eqref{eq:SPDE} as the time stepsize $\tau$ tends to $0$. A natural question arises: 
\begin{itemize}
\item[] \textit{Does the exponential Euler approximation preserve the asymptotical behavior \eqref{eq:u-u}?} 
\end{itemize}
In section \ref{S4.1}, we provide an affirmative answer to the above equation
 by employing the exponential Euler approximation to \eqref{eq:utzk} and establishing the fast advection asymptotics of \eqref{eq:EEM} in terms of the exponential Euler approximation of \eqref{eq:utzk}. 

\subsection{Exponential Euler approximation for SPDE on graph}\label{S4.1}

Let us first introduce the exponential Euler approximation of \eqref{eq:utzk}, which reads 
\begin{equation}\label{eq:Un}
\bar{U}^{n}=\bar{S}(\tau) \bar{U}^{n-1}+\bar{S}(\tau) B(\bar{U}^{n-1}) \tau+\bar{S}(\tau) G(\bar{U}^{n-1}) \delta\bar{\mathcal{W}}_{n-1},\quad n=1,\ldots,N
\end{equation}
with the initial value $\bar{U}^{0}=\psi^{\wedge}$. Here, $\delta\bar{\mathcal{W}}_{n-1}:=\bar{\mathcal{W}}(t_n)-\bar{\mathcal{W}}(t_{n-1})$ is the increment of $\bar{\mathcal{W}}$. One has by iteration that
\begin{align*}
\bar{U}^{n}=\bar{S}(t_n) \psi^\wedge+\int_0^{t_n}\bar{S}(t_n- \lfloor s\rfloor)B(\bar{U}^{\frac{\lfloor s\rfloor}{\tau}}) \ud s+\int_0^{t_n}\bar{S}(t_n- \lfloor s\rfloor) G(\bar{U}^{\frac{\lfloor s\rfloor}{\tau}}) \ud\bar{\mathcal{W}}(s),
\end{align*}
where $\lfloor s\rfloor=t_i$ for $s\in[t_i,t_{i+1})$ with $i=0,1,\ldots,N-1$.
By means of \eqref{eq:St} and \eqref{eq:barGlin-growth}, 
 it can be verified that for any $p\ge1$,
\begin{equation}\label{eq:UnHp}
 \sup_{0\le n\le N}\E\left[\|\bar U^n\|_{\bar{\mathbb{H}}_\gamma}^p\right]
\le C.
 \end{equation}

To characterize the asymptotical behavior of $\{U_{\epsilon}^{n}\}_{1\le n\le N}$ in terms of $\{\bar{U}^{n}\}_{1\le n\le N}$, we regard the initial time step as an intuition, which leads us to the subsequent result.

\begin{lem}\label{prop:AC}
Let Assumption \ref{Asp:gamma} and \eqref{eq:Tneq0} 
hold, and $\psi\in\mathbb{H}_\gamma$. Then $U_{\epsilon}^{1}$ converges to $(\bar{U}^{1})^\vee$ in $L^2(\Omega,\mathbb{H}_\gamma)$ as $\epsilon$ tends to zero (for a fixed $\tau>0$) if and only if
\begin{gather}\label{eq:StauG}
\|\bar{S}(\tau)^\vee (G(\psi)-G((\psi^{\wedge})^{\vee}))\|_{\mathscr{L}_2(RK,\mathbb{H}_\gamma)}+\|\bar{S}(\tau)^\vee (B(\psi)-B((\psi^{\wedge})^{\vee}))\|_{\mathbb{H}_\gamma}=0.
\end{gather}
 In particular, $\psi=(\psi^{\wedge})^{\vee}$ is a sufficient condition for \eqref{eq:StauG}.
\end{lem}
\begin{proof}
For any $f\in \bar{\mathbb{H}}_\gamma$ and $\varphi\in \mathbb{H}_\gamma$, $(f^{\vee}\varphi)^{\wedge}=f\varphi^\wedge$ (see \cite[Formula (3.5)]{CF19}. 
Note that 
$g((\psi^{\wedge})^{\vee})=(g(\psi^{\wedge}))^{\vee}$ implies that $G((\psi^{\wedge})^{\vee})=(G(\psi^{\wedge}))^{\vee}$,
which together with $\delta\bar{\mathcal{W}}_0=(\delta\mathcal{W}_0)^\wedge$ yields  \begin{align*}
(\bar{S}(\tau)G(\psi^{\wedge})\delta\bar{\mathcal{W}}_0)^{\vee}&=(\bar{S}(\tau)G(\psi^{\wedge})(\delta\mathcal{W}_0)^\wedge)^{\vee}\\
&=(\bar{S}(\tau)(G((\psi^{\wedge})^{\vee})\delta\mathcal{W}_0)^\wedge)^{\vee}=\bar{S}(\tau)^{\vee}(G((\psi^{\wedge})^{\vee})\delta\mathcal{W}_0).
\end{align*}
Hence, for the first step, we have 
\begin{align}\label{eq:U1-bU1}
U_{\epsilon}^{1}-(\bar{U}^{1})^\vee&=S_\epsilon(\tau)\psi-\bar{S}(\tau)^\vee \psi+\tau\left[S_\epsilon(\tau)B(\psi)-\bar{S}(\tau)^\vee B((\psi^{\wedge})^{\vee})\right]\\\notag
&\quad +\left[S_\epsilon(\tau)G(\psi)-\bar{S}(\tau)^\vee G((\psi^{\wedge})^{\vee})\right]\delta\mathcal{W}_0.
\end{align}
Due to   \eqref{eq:Seps-S}, it holds that
\begin{align}\label{eq:St-St}\notag
\lim_{\epsilon\to0}\|(S_\epsilon(\tau)-\bar{S}(\tau)^\vee )G(\psi)\|_{\mathscr{L}_2(RK,\mathbb{H}_\gamma)}&=\lim_{\epsilon\to0}\left\|(S_\epsilon(\tau)-\bar{S}(\tau)^\vee )B(\psi)\right\|_{\mathbb{H}_\gamma}\\
&=\lim_{\epsilon\to0}\|S_\epsilon(\tau)\psi-\bar{S}(\tau)^\vee \psi\|_{\mathbb{H}_\gamma}=0.
\end{align} 
In view of the It\^o isometry,
\begin{align}\label{eq:EU1}\notag
\E\left[\|U_{\epsilon}^{1}-(\bar{U}^{1})^\vee\|_{\mathbb{H}_\gamma}^2\right]&=\|S_\epsilon(\tau)\psi-\bar{S}(\tau)^\vee \psi+\tau\left[S_\epsilon(\tau)B(\psi)-\bar{S}(\tau)^\vee B((\psi^{\wedge})^{\vee})\right]\|_{\mathbb{H}_\gamma}^2\\
&\quad+\tau\|S_\epsilon(\tau)G(\psi)-\bar{S}(\tau)^\vee G((\psi^{\wedge})^{\vee})\|_{\mathscr{L}_2(RK,\mathbb{H}_\gamma)}^2.
\end{align}
By the elementary inequality $(a+b)^2\ge\frac12 a^2-b^2$ for $a,b\in\R$, we obtain
\begin{align*}
\E\left[\|U_{\epsilon}^{1}-(\bar{U}^{1})^\vee\|_{\mathbb{H}_\gamma}^2\right]&\ge\tau\|S_\epsilon(\tau)G(\psi)-\bar{S}(\tau)^\vee G((\psi^{\wedge})^{\vee})\|_{\mathscr{L}_2(RK,\mathbb{H}_\gamma)}^2\\
&\ge \frac12\tau\|\bar{S}(\tau)^\vee (G(\psi)-G((\psi^{\wedge})^{\vee}))\|_{\mathscr{L}_2(RK,\mathbb{H}_\gamma)}^2\\
&\quad-\tau\|(S_\epsilon(\tau)-\bar{S}(\tau)^\vee )G(\psi)\|_{\mathscr{L}_2(RK,\mathbb{H}_\gamma)}^2.
\end{align*}
Taking \eqref{eq:St-St} further into account,  we deduce that a necessary condition for $\lim_{\epsilon\to0}\E[\|U_{\epsilon}^{1}-(\bar{U}^{1})^\vee\|_{\mathbb{H}_\gamma}^2]=0$ is
\begin{equation}\label{eq:Stau}
\|\bar{S}(\tau)^\vee (G(\psi)-G((\psi^{\wedge})^{\vee}))\|_{\mathscr{L}_2(RK,\mathbb{H}_\gamma)}=0.
\end{equation}

On the other hand, by \eqref{eq:EU1}, we also have
\begin{align*}
&\E\left[\|U_{\epsilon}^{1}-(\bar{U}^{1})^\vee\|_{\mathbb{H}_\gamma}^2\right]\\
&\ge\|S_\epsilon(\tau)\psi-\bar{S}(\tau)^\vee \psi+\tau\left[S_\epsilon(\tau)B(\psi)-\bar{S}(\tau)^\vee B((\psi^{\wedge})^{\vee})\right]\|_{\mathbb{H}_\gamma}^2\\
&\ge\frac12\tau^2\left\|S_\epsilon(\tau)B(\psi)-\bar{S}(\tau)^\vee B((\psi^{\wedge})^{\vee})\right\|_{\mathbb{H}_\gamma}^2-\|S_\epsilon(\tau)\psi-\bar{S}(\tau)^\vee \psi\|_{\mathbb{H}_\gamma}^2\\
&\ge\frac14\tau^2\left\|\bar{S}(\tau)^\vee (B(\psi)-B((\psi^{\wedge})^{\vee}))\right\|_{\mathbb{H}_\gamma}^2
-\frac12\tau^2\left\|(S_\epsilon(\tau)-\bar{S}(\tau)^\vee )B(\psi)\right\|_{\mathbb{H}_\gamma}^2\\
&\quad-\|S_\epsilon(\tau)\psi-\bar{S}(\tau)^\vee \psi\|_{\mathbb{H}_\gamma}^2.
\end{align*}
Thanks to \eqref{eq:St-St}, another 
necessary condition for $\lim_{\epsilon\to0}\E[\|U_{\epsilon}^{1}-(\bar{U}^{1})^\vee\|_{\mathbb{H}_\gamma}^2]=0$ is
\begin{equation*}
\|\bar{S}(\tau)^\vee (B(\psi)-B((\psi^{\wedge})^{\vee}))\|_{\mathbb{H}_\gamma}=0,
\end{equation*}
which together with \eqref{eq:Stau} yields the necessity of  \eqref{eq:StauG}. Finally, the sufficiency of \eqref{eq:StauG} follows from \eqref{eq:EU1}, \eqref{eq:St-St}, and the triangle inequality.
\end{proof}

As indicated by Lemma \ref{prop:AC},
 compared to the continuous case in \cite[Theorem 5.3]{CF19}, 
 we need to impose an additional condition \eqref{eq:StauG} such that 
 $U_{\epsilon}^{1}$ converges to $(\bar{U}^{1})^\vee$ in $L^2(\Omega,\mathbb{H}_\gamma)$.
This is primarily because the numerical solution $U_\epsilon^1$ depends solely on the initial value $U_\epsilon^0$, rather than on the solution process over the historical time interval $[0,\tau)$, as is the case in continuous case.
In fact,
for the continuous case, instead of \eqref{eq:U1-bU1}, one has that  
\begin{align*}
u_\epsilon(\tau)-(\bar{u}(\tau))^\vee
&=S_\epsilon(\tau)\psi-\bar{S}(\tau)^\vee \psi\\\notag
&\quad+\int_0^{\tau}\left[S_\epsilon(\tau-r)B(u_\epsilon(r))-\bar{S}(\tau-r)^\vee B((\bar{u}(r))^{\vee})\right]\ud r\\\notag
&\quad+\int_0^{\tau} \left[S_\epsilon(\tau-r)G(u_\epsilon(r))-\bar{S}(\tau-r)^\vee G((\bar{u}(r))^{\vee})\right]\ud\mathcal{W}(r)\\
&=:\mathcal{K}_{\epsilon,1}(\tau)+\mathcal{K}_{\epsilon,2}(\tau)+\mathcal{K}_{\epsilon,3}(\tau).
\end{align*}
As shown in \cite[Lemma 5.4]{CF19}, for any $0<\tau_0\le t\le \tau$,
\begin{equation*}
\E\left[\sup_{s\in[0,t]}\|\mathcal{K}_{\epsilon,2}(t)\|_{\mathbb{H}_\gamma}^2\right]\le C\int_{\tau_0}^t\E\left[\sup_{r\in[\tau_0,s]}\|u_\epsilon(r)-(\bar{u}(r))^\vee\|_{\mathbb{H}_\gamma}^2\right]\ud s+R_{\mathsf{\tau}}(\tau_0,\epsilon)
\end{equation*}
for some constant $R_{\mathsf{\tau}}(\tau_0,\epsilon)$ satisfying $\lim_{\epsilon,\tau_0\to 0}R_{\mathsf{\tau}}(\tau_0,\epsilon)=0$. The above inequality enables the use of the Gronwall inequality to control $\|\mathcal{K}_{\epsilon,2}(t)\|_{\mathbb{H}_\gamma}$, and $\|\mathcal{K}_{\epsilon,3}(t)\|_{\mathbb{H}_\gamma}$ can be treated similarly.

The following result reveals that the condition \eqref{eq:StauG} is also sufficient under which 
the exponential Euler approximation of \eqref{eq:SPDE}  can be approximately characterized by that of \eqref{eq:utzk} in the fast advection limit.

 \begin{tho}\label{theo:asy}
Let Assumption \ref{Asp:gamma} and \eqref{eq:Tneq0} 
hold, and $\psi\in\mathbb{H}_\gamma$. Then under the condition \eqref{eq:StauG},  for any fixed $\tau>0$,
\begin{equation}\label{eq:Un-Un}
\lim_{\epsilon\to 0}\E\left[\sup_{1\le n\le N}\|U_{\epsilon}^{n}-(\bar U^{n})^{\vee}\|_{\mathbb{H}_\gamma}^p\right]=0\quad \forall~p\ge2.
\end{equation}
 In particular,   \eqref{eq:Un-Un} holds when
 $\psi=(\psi^{\wedge})^{\vee}$.
\end{tho}
The proof of Theorem \ref{theo:asy} is postponed in Appendix \ref{Sec:A3}. Theorem \ref{theo:asy} reveals that the exponential Euler approximation \eqref{eq:EEM} is capable of preserving the asymptotical behavior \eqref{eq:u-u} of the stochastic RDA equation \eqref{eq:SPDE} as $\epsilon$ tends to $0$.

\subsection{Euler--Maruyama approximation for SPDE on graph}\label{S4.2}
In contrast to Theorem \ref{theo:asy}, by taking the standard Euler--Maruyama approximation as an example, we illustrate in this part that common numerical discretizations may fail to preserve the fast advection asymptotics of \eqref{eq:SPDE}.

  To simplify the analysis, we will discuss under Assumption \ref{Asp:zeta}. Under Assumption \ref{Asp:zeta}, a function $\varphi:\R^2\to \R$ satisfies $\varphi=(\varphi^\wedge)^{\vee}$ 
  if and only if $\varphi$ is a radial function. Indeed, if $\varphi=(\varphi^\wedge)^{\vee}=(\varphi^\wedge)\circ H$, then
 $\varphi$ is radial since $H$ is radial. Conversely, for any radial function 
 $\varphi(x)=R(|x|^2)$, we can rewrite $\varphi(x)=R( F(H(x)))$ with $F=(Id+\zeta)^{-1}$, which means that $\varphi=(\varphi^\wedge)^{\vee}$ with $\varphi^\wedge=R\circ F$.

 \begin{lem}\label{lemL-L}
 Let Assumption \ref{Asp:zeta} hold and $\varphi:\R^2\to \R$ be a second continuously differentiable radial function.
 Then 
 $(\mathcal L_\epsilon \varphi)^\wedge=\bar{\mathcal L} \varphi^\wedge$ for any $\epsilon>0$. 
 \end{lem}
 \begin{proof}
 For any radial function $\varphi=R(|x|^2)$ with some $R:[0,\infty)\to\R$,
\begin{equation*}
\mathcal L_\epsilon \varphi=\frac12 \Delta \varphi=2  |x|^2R^{\prime\prime}(|x|^2)+2 R^\prime(|x|^2).
\end{equation*}
 In virtue of \eqref{eq:TA}, one has $A^\prime(z)
=
2\pi+2\pi F(z)\zeta^{\prime\prime}(F(z))F^{\prime}(z)$, and 
 \begin{gather*} 
\frac{A(z)}{T(z)}=2F(z)\big(1+\zeta^\prime(F(z))\big)^2,\quad
\frac{A^\prime(z)}{T(z)}=2\big(1+\zeta^\prime(F(z))+F(z)\zeta^{\prime\prime}(F(z))\big).
\end{gather*}
 Under Assumption \ref{Asp:zeta},
 $F^\prime(z)=\frac{1}{1+\zeta^\prime(F(z))}$ and $F^{\prime\prime}(z)=-\frac{\zeta^{\prime\prime}(F(z))}{(1+\zeta^\prime(F(z)))^3}.$
 Then direct calculations give that
 \begin{gather*}
 \varphi^{\wedge}(z)=\frac{1}{T(z)}\oint_{\{|x|=\sqrt{F(z)}\}}\frac{R(|x|^2)}{|\nabla H(x)|}\ud l_z=R(F(z)),\\
\frac{\ud }{\ud z}\varphi^\wedge(z)
=\frac{R^\prime(F(z))}{1+\zeta^\prime(F(z))},\qquad 
\frac{\ud^2 }{\ud z^2}\varphi^\wedge(z)
=\frac{R^{\prime\prime}(F(z))}{(1+\zeta^\prime(F(z)))^2}-\frac{R^{\prime}(F(z))\zeta^{\prime\prime}(F(z))}{(1+\zeta^\prime(F(z)))^3}.
\end{gather*} 
Collecting the above formulas, we arrive at
\begin{align*}
\bar{\mathcal L}\varphi^\wedge(z)=\frac{A(z)}{T(z)}\frac{\ud^2}{\ud z^2}\varphi^\wedge(z)+\frac{A^\prime(z)}{T(z)}\frac{\ud}{\ud z}\varphi^\wedge(z)=2F(z)R^{\prime\prime}(F(z))+2R^{\prime}(F(z)).
\end{align*}
By the fact that $H(x)=z$ is equivalent to $|x|^2=F(z)$, we have that $(\mathcal L_\epsilon \varphi)^\wedge=\bar{\mathcal L} \varphi^\wedge$ for radial $\varphi$. The proof is completed.
 \end{proof}

To illustrate  that
the standard Euler--Maruyama approximation cannot preserve the asymptotical behavior \eqref{eq:u-u}, throughout this part, we assume  for simplicity that $b\equiv0$, $\mu\neq \delta_0$, Assumptions \ref{Asp:gamma}-\ref{Asp:zeta} hold and the initial value $\psi=(\psi^\wedge)^{\vee}$. Here $\delta_0$ denotes the Dirac measure concentrated at $0\in\R^2$.
Then the Euler--Maruyama approximation for \eqref{eq:SPDE} reads:
\begin{equation*}
U_{\epsilon,EM}^{n}=U_{\epsilon,EM}^{n-1}+\tau\mathcal L_\epsilon U_{\epsilon,EM}^{n-1}+G(U_{\epsilon,EM}^{n-1})\delta \mathcal{W}_{n-1},\quad n=1,\ldots,N.
\end{equation*}
Similarly, the Euler--Maruyama method applied to the limiting equation \eqref{eq:utzk-p} gives
\begin{equation*}
\bar{U}_{EM}^{n}=\bar{U}_{EM}^{n-1}+\tau\bar{\mathcal L} \bar{U}_{EM}^{n-1}+G(\bar{U}_{EM}^{n-1})\delta \bar{\mathcal{W}}_{n-1},\quad n=1,\ldots,N.
\end{equation*}
For the first step (i.e., $n=1$), by Lemma \ref{lemL-L} and $(\psi^{\wedge})^{\vee}=\psi$,
\begin{align*}
\E\left[\|U_{\epsilon,EM}^{1}-(\bar U_{EM}^{1})^{\vee}\|_{\mathbb{H}_\gamma}^2\right]&=\E\left[\|G(\psi)\delta \mathcal{W}_{0}-G(\psi)(\delta \bar{\mathcal{W}}_{0})\circ H\|_{\mathbb{H}_\gamma}^2\right]\\
&=\tau\sum_{j=1}^\infty\|G(\psi)\widehat{\mathfrak{u}_j\mu}-G(\psi)((\widehat{\mathfrak{u}_j\mu})^\wedge)^{\vee}\|_{\mathbb{H}_\gamma}^2.
\end{align*}
Note that in this case, the Euler--Maruyama approximation is not divergent since $n=1.$
Claim that for $\mu\neq\delta_0$, 
\begin{equation}\label{eq:uj0mu}
\widehat{\mathfrak{u}_{j_0}\mu}\quad\text{ is not radial for at least one  }j_0\in\mathbb{N}_+.
\end{equation} 
For such non-radial $\widehat{\mathfrak{u}_{j_0}\mu}$, by the positivity of $\gamma$ in Assumption \ref{Asp:gamma}, we have
$$\|G(\psi)\widehat{\mathfrak{u}_{j_0}\mu}-G(\psi)((\widehat{\mathfrak{u}_{j_0}\mu})^\wedge)^{\vee}\|_{\mathbb{H}_\gamma}^2>0,$$
provided that  there exists a measurable set $E_1\subset\R^2$ with positive Lebesgue measure such that
$G(\psi(x))\neq0$ for $x\in E_1$.
In particular, for the additive noise (i.e., $g$ is a non-zero constant), it holds that for any fixed $\tau>0$,
\begin{align*}
\lim_{\epsilon\to 0}\E\left[\|U_{\epsilon,EM}^{1}-(\bar U_{EM}^{1})^{\vee}\|_{\mathbb{H}_\gamma}^2\right]>0.
\end{align*}
This proves that the standard Euler--Maruyama approximation for \eqref{eq:SPDE} cannot preserve the fast advection asymptotics \eqref{eq:u-u}  in general.

 To prove \eqref{eq:uj0mu},
 assume by contradiction that $x\mapsto\widehat{\mathfrak{u}_j\mu}(x)$ is radial for all $j\in\mathbb{N}_+$. 
 Since the set $\{\mathfrak{u}_j\}_{j=1}^\infty$ forms a complete orthonormal basis of $L^2_{(s)}(\R^2,\ud\mu)$, it holds that for any $x\in \R^2$,
$$
e^{-\textup{i}\langle x,\cdot\rangle}=\sum_{j=1}^\infty
    \widehat{\mathfrak{u}_j\mu}(x) \mathfrak{u}_j(\cdot)\quad \text{in }L^2(\R^2,\ud\mu).
$$
By the assumption that $x\mapsto\widehat{\mathfrak{u}_j\mu}(x)$ is radial for all $j\in\mathbb{N}_+$, we obtain that $$e^{-\textup{i}\langle (r\cos\theta_1,r\sin\theta_1),\cdot\rangle}=e^{-\textup{i}\langle (r\cos\theta_2,r\sin\theta_2),\cdot\rangle}\quad \text{in }L^2(\R^2,\ud\mu)$$
for any $r>0$ and $\theta_1,\theta_2\in[0,2\pi)$. This is equivalent to 
$$
\int_{\R^2}|e^{-\textup{i}r(\cos(\theta_1)\xi_1+\sin(\theta_1)\xi_2)}-e^{-\textup{i}r(\cos(\theta_2)\xi_1+\sin(\theta_2)\xi_2)}|^2\mu(\ud\xi)=0
$$
 for any $r>0$ and $\theta_1,\theta_2\in[0,2\pi)$.
For $r=1$, the above identity implies that
 \begin{align}\label{eq:cossin}
     &\cos(\theta_1)\xi_1+\sin(\theta_1)\xi_2-\cos(\theta_2)\xi_1-\sin(\theta_2)\xi_2\\\notag
     &=2\sin\left(\frac{\theta_1-\theta_2}{2}\right)\left[-\sin\left(\frac{\theta_1+\theta_2}{2}\right)\xi_1+\cos\left(\frac{\theta_1+\theta_2}{2}\right)\xi_2\right]\in2\pi\mathbb{Z}
 \end{align}
  for any $\xi=(\xi_1,\xi_2)\in\text{supp}(\mu)$ and 
all $\theta_1,\theta_2\in[0,2\pi)$. By the arbitrariness of $\theta_1,\theta_2\in[0,2\pi)$,
 \eqref{eq:cossin} holds if and only if $\xi_1=\xi_2=0$, which leads to a contradiction with $\text{supp}(\mu)\neq\{0\}$ and thus proves the  claim \eqref{eq:uj0mu}.

\section{Asymptotic-preserving property}\label{S5}
In this section, we study the AP property of the exponential Euler approximation  \eqref{eq:EEM} for \eqref{eq:SPDE}.  Since
$\lim_{\epsilon\to 0}U_\epsilon^N=\bar{U}^N\circ\Pi$ (see Theorem \ref{theo:asy}), where $\bar{U}^N$ is generated via  \eqref{eq:Un}, verifying the AP property of the exponential Euler approximation \eqref{eq:EEM} now boils down to analyzing the convergence of the exponential Euler approximation  \eqref{eq:Un} to the limiting equation \eqref{eq:utzk}. To reduce the complexity of the singularity near the vertice $\{O_k\}_{k=1}^{m_1}$, we restrict ourselves to
a special case that the graph $\Gamma=[0,\infty)$ with only two vertices, or 
equivalently, $H$ admits a unique critical point $0$. 
In particular, we will work under Assumption \ref{Asp:zeta} in this section.

\subsection{Regularity of semigroup on graph}
We begin with some useful estimates of  the functions $A$ and $T$,
which will be used in the regularity estimate of the semigroup $\{\bar{S}(t)\}_{t\ge0}$ generated by $\bar{\mathcal L}$.

\begin{lem}\label{lem:TFA}
Under Assumption \ref{Asp:zeta}, 
for any $z\in[0,\infty)$,
\begin{gather}\label{eq:T}
\frac{\pi}{1+\tilde{r}_0}\le T(z)\le\frac{\pi}{1-r_0},\\\label{eq:A}
 2\pi \frac{1-r_0}{1+\tilde{r}_0}z\le A(z)\le2\pi \frac{1+\tilde{r}_0}{1-r_0}z.
\end{gather}
\end{lem}
\begin{proof}
The first assertion \eqref{eq:T} comes from \eqref{eq:zeta1} and \eqref{eq:TA}.
By \eqref{eq:zeta1} and $\zeta(0)=0$,
\begin{gather*}
(1-r_0)z\le (Id+\zeta)(z)=z+\zeta(0)+\int_0^z\zeta^\prime(y)\ud y\le (1+\tilde{r}_0)z\quad \forall~z\in[0,\infty).
\end{gather*}
Therefore, $(Id+\zeta)((1-r_0)^{-1}z)\ge z$ and $(Id+\zeta)((1+\tilde{r}_0)^{-1}z)\le z$.
Since
\begin{align}\label{eq:F'}
\frac{1}{1-r_0}\ge F^\prime(z)=\frac{1}{1+\zeta^\prime(F(z))}\ge \frac{1}{1+\tilde{r}_0}>0,\quad z\in[0,\infty),
\end{align} 
$F$ is monotone increasing on $[0,\infty)$. Hence, 
$F(z)\le F((Id+\zeta)((1-r_0)^{-1}z))\le (1-r_0)^{-1}z$ and $F(z)\ge F((Id+\zeta)((1+\tilde{r}_0)^{-1}z))\ge (1+\tilde{r}_0)^{-1}z$ for all $z\in[0,\infty)$, namely, 
\begin{equation}\label{eq:F}
 \frac{1}{1+\tilde{r}_0}z\le F(z)\le \frac{1}{1-r_0}z.
\end{equation}
 Finally, the second inequality \eqref{eq:A} follows from \eqref{eq:F}, \eqref{eq:TA} and \eqref{eq:zeta1}.
\end{proof}

\begin{lem}\label{lem:AT1}
Under Assumption \ref{Asp:zeta}, for any $z\in[0,\infty)$,
\begin{gather*}
|T^\prime(z)|+|A^\prime(z)|\le C,\\\notag
|A^{\prime\prime}(z)|+|A(z)T^\prime(z)|\le C.
\end{gather*}
\end{lem}
\begin{proof}
By \eqref{eq:T'}, \eqref{eq:zeta1} and the first inequality of \eqref{eq:zeta3}, we have $|T^\prime(z)|\le C$. Since
$A^\prime(z)
=
2\pi+2\pi F(z)\zeta^{\prime\prime}(F(z))F^{\prime}(z)$,
from \eqref{eq:zeta2} and \eqref{eq:F'}
we can obtain that $|A^\prime(z)|\le C$.
The relations $T(z)=\pi F^\prime(z)$ and $|T^\prime(z)|\le C$ imply that $|F^{\prime\prime}(z)|$ is uniformly bounded for all $z\in[0,\infty)$.
Furthermore, by \eqref{eq:zeta2}, \eqref{eq:zeta3}, \eqref{eq:F'}, and
\begin{align*}
A^{\prime\prime}(z)
&=2\pi \left(\zeta^{\prime\prime}(F(z))|F^{\prime}(z)|^2+F(z)\zeta^{\prime\prime\prime}(F(z))|F^{\prime}(z)|^2+F(z)\zeta^{\prime\prime}(F(z))F^{\prime\prime}(z)\right),
\end{align*}
we obtain that $|A^{\prime\prime}|$ is uniformly bounded on $[0,\infty)$. Using
\eqref{eq:zeta2}, \eqref{eq:T} and $$A(z)T^\prime(z)=-2\pi^2 \frac{F(z)\zeta^{\prime\prime}(F(z)) }{(1+\zeta^\prime(F(z)))^2}=-2 F(z)\zeta^{\prime\prime}(F(z)) |T(z)|^2,$$
one can see that $|A(z)T^\prime(z)|$ is also uniformly bounded. 
\end{proof}

Given an integrable function $\tilde{\varrho}:[0,\infty)\to [0,\infty)$, we denote
$\bar{L}^2_{\tilde{\varrho}}:=L^2(0,\infty;\tilde{\varrho}(z)\ud z)$ endowed with the norm 
$\|f\|_{\bar{L}^2_{\tilde{\varrho}}}:=(\int_0^\infty |f(z)|^2\tilde{\varrho}(z)\ud z)^{\frac12}$. Under Assumptions \ref{Asp:gamma}-\ref{Asp:zeta}, the space $\bar{\mathbb{H}}_\gamma$ defined in \eqref{eq:H1} coincides with $\bar{L}^2_{\gamma T}$, where $\gamma=\vartheta$. 
 The spaces
$\bar{L}^2_{\tilde{\varrho} T}$ and $\bar{L}^2_{\tilde{\varrho}}$ are equivalent
due to \eqref{eq:T}. 
Next, we present the 
 a priori estimate of the solution $\bar u$ to \eqref{eq:utzk-p} in graph weighted $L^2$-spaces.

\begin{lem}\label{lem:Sf1}
Under Assumptions \ref{Asp:gamma}-\ref{Asp:zeta}, we have the following results.
 \begin{enumerate}
 \item[(i)] For any $f\in\bar{L}^2_{\gamma T}$,
\begin{gather*}
 \|\bar{S}(t)f\|^2_{\bar{L}^2_{\gamma T}}+2\int_0^t\|\partial_{z}\bar{S}(s)f\|_{\bar{L}^2_{A\gamma}}^2\ud s\lesssim\|f\|_{\bar{L}^2_{\gamma T}}^2\quad\forall~ t\in[0,\mathsf T].
\end{gather*}
\item[(ii)]If
 $\psi^{\wedge}\in \bar{L}^2_{\gamma T}$, then 
 \begin{align*}
\E\left[\|\bar{u}(t)\|^2_{\bar{L}^2_{\gamma T}}\right]+2\E\int_0^t\|\partial_z\bar{u}(r)\|^2_{\bar{L}_{A\gamma}^2}\ud r\lesssim \|\psi^{\wedge}\|^2_{\bar{L}^2_{\gamma T}}+1\quad\forall~ t\in[0,\mathsf T].
\end{align*} 
\end{enumerate}
\end{lem}
\begin{proof}

Note that $\bar U_{f}(t):=\bar{S}(t)f$ satisfies that $\bar U_{f}(0,z)=f(z)$, and
\begin{equation}\label{eq:partbUf}
\partial_t\bar U_{f}(t,z)=\bar{\mathcal L}\bar U_{f}(t,z)=\frac{A(z)}{T(z)}\partial_{zz}\bar U_{f}(t,z)+\frac{A^\prime(z)}{T(z)}\partial_z\bar U_{f}(t,z),\quad t\in[0,\mathsf T].
\end{equation}
Due to Assumption \ref{Asp:gamma}, Lemma \ref{lem:AT1}, and Lemma \ref{lem:TFA}, we have that
$$|A(z)\gamma^{\prime\prime}(z)|+|A^\prime(z)\gamma^\prime(z)|\lesssim \gamma(z)\lesssim \gamma(z)T(z),\quad z\in[0,\infty).$$
Multiplying $2\bar U_{f}(t,z)$ and integrating with respect to $\gamma (z)T(z)\ud z$ on both sides of \eqref{eq:partbUf}, and then using the integration by parts formula, we obtain
\begin{align*}
&\partial_t\int_0^\infty|\bar U_{f}(t,z)|^2\gamma(z)T(z)\ud z=2\int_0^\infty \bar U_{f}(t,z)\bar{\mathcal L}\bar U_{f}(t,z)\gamma(z)T(z)\ud z\\
&=\int_0^\infty A(z)\left[\partial_{zz}|\bar U_{f}(t,z)|^2-2|\partial_z\bar U_{f}(t,z)|^2\right]\gamma(z)\ud z+\int_0^\infty A^\prime(z)\partial_z|\bar U_{f}(t,z)|^2\gamma(z)\ud z\\
&=-2\int_0^\infty A(z)|\partial_z\bar U_{f}(t,z)|^2\gamma(z)\ud z-\int_0^\infty A(z)\partial_z|\bar U_{f}(t,z)|^2\gamma^\prime(z)\ud z\\
&=-2\int_0^\infty A(z)|\partial_z\bar U_{f}(t,z)|^2\gamma(z)\ud z+\int_0^\infty |\bar U_{f}(t,z)|^2(A(z)\gamma^{\prime\prime}(z)+A^\prime(z)\gamma^\prime(z))\ud z\\ 
&\le-2\int_0^\infty A(z)|\partial_z\bar U_{f}(t,z)|^2\gamma(z)\ud z+C\int_0^\infty |\bar U_{f}(t,z)|^2\gamma(z)T(z)\ud z.
\end{align*}
Then using the Gronwall inequality yields the first assertion $(i)$.

By the first assertion $(i)$, $\|\bar{S}(t) \|_{\mathscr{L}(\bar{L}^2_{\gamma T})}\le C$ 
 for any $t\in[0,\mathsf T]$. Hence,
it follows from \eqref{eq:ut}, the H\"older inequality and an argument similar to \eqref{eq:barGlin-growth} 
that
\begin{align*}
\E\left[\|\bar{u}(t)\|^2_{\bar{L}^2_{\gamma T}}\right]&\lesssim \|\bar{S}(t) \psi^{\wedge}\|^2_{\bar{L}^2_{\gamma T}}+\int_0^t \E\left[\|\bar{S}(t-s) B(\bar{u}(s))\|^2_{\bar{L}^2_{\gamma T}} \right]\ud s\\
&\quad+\int_0^t\sum_{j=1}^\infty\E\left[\| \bar{S}(t-s) G(\bar{u}(s))(\widehat{\mathfrak{u}_j\mu})^{\wedge} \|^2_{\bar{L}^2_{\gamma T}}\right]\ud s\\
&\lesssim\|\psi^{\wedge}\|^2_{\bar{L}^2_{\gamma T}}+\int_0^t \E\left[\|B(\bar{u}(s))\|^2_{\bar{L}^2_{\gamma T}} \right]\ud s
+\int_0^t\sum_{j=1}^\infty\E\left[\| G(\bar{u}(s))(\widehat{\mathfrak{u}_j\mu})^{\wedge} \|^2_{\bar{L}^2_{\gamma T}}\right]\ud s\\
&\lesssim \|\psi^{\wedge}\|^2_{\bar{L}^2_{\gamma T}} +1+\int_0^t \E\left[\|\bar{u}(s)\|^2_{\bar{L}^2_{\gamma T}} \right]\ud s,
\end{align*}
in virtue of the linear growth of $b$ and $g$. Finally
using the Gronwall inequality leads to the second assertion $(ii)$.
\end{proof}

Under Assumptions \ref{Asp:zeta}-\ref{Asp:gamma-S},
Lemma \ref{lem:Sf1} also holds for $\gamma$ replaced by $\sqrt\gamma$.  In particular, for any $t\in[0,\mathsf T]$,
 \begin{align}\label{eq:utHgamma}
\E\left[\|\bar{u}(t)\|^2_{\bar{\mathbb{H}}_{\gamma}}\right]\le C\left( \|\psi^{\wedge}\|^2_{\bar{\mathbb{H}}_{\gamma}}+1\right),\qquad
\E\left[\|\bar{u}(t)\|^2_{\bar{\mathbb{H}}_{\sqrt\gamma}}\right]\le C\left( \|\psi^{\wedge}\|^2_{\bar{\mathbb{H}}_{\sqrt\gamma}}+1\right).
\end{align}

\begin{rem}\label{rem:St-I} Given $f\in \bar{\mathbb{H}}_\gamma$, the triangle inequality indicates \begin{align*}
\|(\bar{S}(t)-I)f\|_{\bar{\mathbb{H}}_\gamma}\le \|(\bar{S}(t)-S_\epsilon(t)^\wedge)f\|_{\bar{\mathbb{H}}_\gamma}+\|(S_\epsilon(t)^\wedge-I)f\|_{\bar{\mathbb{H}}_\gamma}.
\end{align*}
By \eqref{eq:Seps-S} and $(f^\vee)^\wedge=f$, for any $t>0$, 
$$\|(\bar{S}(t)-S_\epsilon(t)^\wedge)f\|_{\bar{\mathbb{H}}_\gamma}=\|\bar{S}(t)f-(S_\epsilon(t) f^\vee)^\wedge\|_{\bar{\mathbb{H}}_\gamma}\to 0\quad \text{as }\epsilon\to 0.$$
On the other hand, in view of \eqref{fvarphi} and Proposition \ref{prop:S-I-1},
\begin{align*}
\|(S_\epsilon(t)^\wedge-I)f\|_{\bar{\mathbb{H}}_\gamma}
&\le \|(S_\epsilon(t)-I)f^\vee\|_{\mathbb{H}_\gamma}\\
& \lesssim t^{\frac12}\left(\|\nabla (f^\vee)\|_{\mathbb{H}_\gamma}+\epsilon^{-\frac12}\|f^\vee\|_{\mathbb{H}_\gamma}+\epsilon^{-1} \|f^\vee\|_{\mathbb{H}_{\sqrt\gamma}}\right),
\end{align*}
{where the  right-hand side tends to $\infty$ as $\epsilon$ vanishes}. 
\end{rem}

With the observation in Remark \ref{rem:St-I}, instead of considering the space $\bar{\mathbb{H}}_\gamma=\bar{L}^2_{\gamma T}$, we choose to deal with the difference between $\bar{S}(t)$ and the identity operator $I$ in another suitable graph weighted $L^2$-space, in which the space $\bar{L}_{A\gamma}^2$ is a candidate.
By \eqref{eq:partbUf},
the Minkowski and H\"older inequalities, 
it holds that for any $t\in(0,\mathsf T]$,
\begin{align}\label{eq:S-I-0}\notag
&\|(\bar{S}(t)-I)f\|_{\bar{L}_{A\gamma}^2}=\Big\|\int_0^t\partial_s\bar{S}(s)f\ud s\Big\|_{\bar{L}_{A\gamma}^2}\le\int_0^t\left\|\mathcal{\bar{\mathcal L}}\bar{S}(s)f\right\|_{\bar{L}_{A\gamma}^2}\ud s\\\notag
&\quad\le t^{\frac12}\left(\int_0^t\int_0^\infty\Big|\frac{A(z)}{T(z)}\partial_{zz} \bar{S}(s)f(z)+\frac{A^\prime(z)}{T(z)}\partial_z\bar{S}(s)f(z)\Big|^2\gamma(z)A(z)\ud z\ud s\right)^{\frac12}\\\notag
&\quad\lesssim t^{\frac12}\left(\int_0^t\int_0^\infty\Big|\frac{A}{T}\partial_{zz}\bar{S}(s)f\Big|^2\gamma A\ud z\ud s\right)^{\frac12}+t^{\frac12}\left(\int_0^t\int_0^\infty\frac{|A^\prime|^2}{T^2}|\partial_z\bar{S}(s)f|^2\gamma A
\ud z\ud s\right)^{\frac12}\\
&\quad\lesssim t^{\frac12}\left(\int_0^t\int_0^\infty\frac{A^3}{T}|\partial_{zz}\bar{S}(s)f|^2\gamma\ud z\ud s\right)^{\frac12}+t^{\frac12}\|f\|_{\bar{\mathbb{H}}_\gamma},
\end{align}
where the last line is due to Lemma \ref{lem:Sf1}$(i)$, the boundedness of $|A^\prime(z)|$ in Lemma \ref{lem:AT1} and the lower bound of $T(z)$ in \eqref{eq:T}.
To proceed, we need the following gradient estimate of the semigroup $\bar{S}(\cdot)$ under the $L^\infty(0,\mathsf T;\bar{L}_{A^2\gamma}^2)$-norm, in which the linear part $\frac{A(z)}{T(z)}\frac{\ud ^2}{\ud z^2}$ in $\bar{\mathcal L}$ can lead to an upper bound for the $L^2(0,\mathsf T;\bar{L}_{A^3\gamma/T}^2)$-norm of $\partial_{zz}\bar{S}(\cdot)$.

\begin{lem}\label{lem:partzbUf}
Let Assumptions \ref{Asp:gamma}-\ref{Asp:zeta}  hold, 
$f\in \bar{\mathbb{H}}_\gamma$ and $f^\prime\in \bar{L}_{A^2\gamma}^2$. Then for $t\in[0,\mathsf T]$,
\begin{gather*}
 \|\partial_z\bar{S}(t)f\|_{\bar{L}_{A^2\gamma}^2}^2+\int_0^t\int_0^\infty\frac{A^3}{T}|\partial_{zz}\bar{S}(s)f|^2\gamma\ud z\ud s
 \lesssim\|f^\prime\|_{\bar{L}_{A^2\gamma}^2}^2+\|f\|^2_{\bar{\mathbb{H}}_\gamma}.
\end{gather*}
\end{lem}
\begin{proof}
By the chain rule, $\bar V_{f}(t):=\partial_z\bar U_{f}(t)=\partial_z\bar S(t)f$ satisfies that $\partial_t\bar V_{f}(t,z)=\widetilde{\mathcal L}\bar V_{f}(t,z)$ with $\bar V_{f}(0)=f^\prime$ and
\begin{align*}
\widetilde{\mathcal L}&:=\frac{A(z)}{T(z)}\frac{\ud^2}{\ud z^2}+\left(2\frac{A^\prime(z)}{T(z)}-\frac{A(z)T^\prime(z)}{T(z)^2}\right)\frac{\ud}{\ud z}+\left(\frac{A^{\prime\prime}(z)}{T(z)}-\frac{A^\prime(z)T^\prime(z)}{T(z)^2}\right).
\end{align*}
 Further, by the chain rule,
\begin{align}\label{eq:Vf2}
&\partial_t |\bar V_{f}(t,z)|^2=2\bar V_{f}(t,z)\partial_t\bar V_{f}(t,z)=2\bar V_{f}(t,z)\widetilde{\mathcal L}\bar V_{f}(t,z)\\\notag
&=\frac{A(z)}{T(z)}\left[\partial_{zz}|\bar V_{f}(t,z)|^2-2|\partial_{z}\bar V_{f}(t,z)|^2\right]+\left(2\frac{A^\prime(z)}{T(z)}-\frac{A(z)T^\prime(z)}{T(z)^2}\right)\partial_z|\bar V_{f}(t,z)|^2\\\notag&\quad+2\left(\frac{A^{\prime\prime}(z)}{T(z)}-\frac{A^\prime(z)T^\prime(z)}{T(z)^2}\right)|\bar V_{f}(t,z)|^2.
\end{align}
 Integrating with respect to $\gamma(z)A(z)^2\ud z$ on both sides of \eqref{eq:Vf2} leads to 
\begin{align}\label{eq:Ufb2}
&\frac{\ud}{\ud t} \int_0^\infty A(z)^2|\bar V_{f}(t,z)|^2\gamma(z)\ud z= \int_0^\infty 2A(z)^2\bar V_{f}(t,z)\widetilde{\mathcal L}\bar V_{f}(t,z)\gamma(z)\ud z\\\notag&=\int_0^\infty \frac{A(z)^3}{T(z)}\partial_{zz}|\bar V_{f}(t,z)|^2\gamma(z)\ud z\\\notag
&\quad-2\int_0^\infty\frac{A(z)^3}{T(z)}|\partial_{z}\bar V_{f}(t,z)|^2\gamma(z)\ud z\\\notag
&\quad +\int_0^\infty \left(2\frac{A(z)^2A^\prime(z)}{T(z)}-\frac{A(z)^3T^\prime(z)}{T(z)^2}\right)\partial_z|\bar V_{f}(t,z)|^2 \gamma(z)\ud z\\\notag
&\quad+2\int_0^\infty A(z)^2\left(\frac{A^{\prime\prime}(z)}{T(z)}-\frac{A^\prime(z)T^\prime(z)}{T(z)^2}\right)|\bar V_{f}(t,z)|^2\gamma(z)\ud z=:\sum_{l=1}^4J_l.
\end{align}
In virtue of the integration by parts formula and $A(0)=0$,
\begin{align*}
J_1=-\int_0^\infty\partial_{z}|\bar V_{f}|^2 \left(\frac{3A^2A^\prime\gamma+A^3\gamma^\prime}{T}-\frac{A^3\gamma T^\prime}{T^2}\right)\ud z.
\end{align*}
Taking advantage of \eqref{eq:T}, Lemma \ref{lem:AT1}, and 
$|z\gamma^{\prime\prime}(z)|+|\gamma^{\prime}(z)|\lesssim \gamma(z)$, it holds that
 \begin{align}\label{eq:AT}
\frac{\left|A^\prime(z)\gamma^\prime(z)\right|+\left|A(z)\gamma^{\prime\prime}(z)\right|}{T(z)}+\left|\frac{A(z)\gamma^\prime(z)T^\prime(z)}{T(z)^2}\right|\lesssim \gamma(z)\quad\forall~z\in[0,\infty).
\end{align}
Utilizing the Young inequality and integrating by parts again,
\begin{align*}
J_1+J_3&=-2\int_0^\infty\bar V_{f}\partial_{z}\bar V_{f}\frac{A^2A^\prime}{T}\gamma\ud z
-\int_0^\infty \partial_z|\bar V_{f}|^2\frac{A^3\gamma^\prime}{T}\ud z\\
&\le\int_0^\infty|\partial_{z}\bar V_{f}|^2 \frac{A^3}{T}\gamma\ud z+\int_0^\infty|\bar V_{f}|^2\frac{A|A^\prime|^2}{T}\gamma\ud z\\
&\quad+ \int_0^\infty A^2\left(\frac{3A^\prime\gamma^\prime+A\gamma^{\prime\prime}}{T}-\frac{A\gamma^\prime T^\prime}{T^2}\right)|\bar V_{f}|^2\ud z\\
&\le -\frac{1}{2}J_2+C\int_0^\infty|\bar V_{f}|^2A\gamma\ud z+ C\int_0^\infty A^2\gamma|\bar V_{f}|^2\ud z,
\end{align*}
thanks to \eqref{eq:T}, \eqref{eq:AT} and the boundedness of $|A^\prime|$ in Lemma \ref{lem:AT1}.

Besides, it follows from Lemma \ref{lem:AT1} and \eqref{eq:T} that
$$\left|\frac{A^{\prime\prime}(z)}{T(z)}-\frac{A^\prime(z)T^\prime(z)}{T(z)^2}\right|\le\left|\frac{A^{\prime\prime}(z)}{T(z)}\right|+\left|\frac{A^\prime(z)T^\prime(z)}{T(z)^2}\right|\lesssim 1,$$
which yields that $J_4\lesssim\int_0^\infty|\bar V_{f}|^2A^2\gamma\ud z$. Collecting above estimates yields that
\begin{align*}
\sum_{l=1}^4J_l\le C\int_0^\infty A^2|\bar V_{f}|^2\gamma\ud z-\int_0^\infty\frac{A^3}{T}|\partial_{z}\bar V_{f}|^2\gamma\ud z+C\int_0^\infty|\bar V_{f}|^2A\gamma\ud z.
\end{align*}
Plugging it into \eqref{eq:Ufb2} and using Lemma \ref{lem:Sf1}$(i)$, we conclude that
\begin{align*}
&\quad\ \int_0^\infty A^2|\bar V_{f}(t)|^2\gamma\ud z+\int_0^t\int_0^\infty\frac{A^3}{T}|\partial_{z}\bar V_{f}(s)|^2\gamma\ud z\ud s\\\notag&\le \int_0^\infty A^2|f^\prime|^2\gamma\ud z+C\int_0^t\int_0^\infty A|\bar V_{f}(s)|^2\gamma\ud z\ud s+C\int_0^t\int_0^\infty A^2|\bar V_{f}(s)|^2\gamma\ud z\ud s\\\notag
&\le \int_0^\infty A^2|f^\prime|^2\gamma\ud z+C\int_0^\infty |f|^2\gamma T\ud z+C\int_0^t\int_0^\infty A^2|\bar V_{f}(s)|^2\gamma\ud z\ud s.
\end{align*}
Finally applying the Gronwall inequality produces the desired  result.
\end{proof}

Recall that in Proposition \ref{prop:S-I-1}, we have obtained the temporal H\"older continuity of the semigroup $\{S_\epsilon(t)\}_{t\in[0,T]}$ in $\mathbb{H}_\gamma$. 
Combining Lemma \ref{lem:partzbUf} and \eqref{eq:S-I-0}, we obtain the temporal H\"older continuity of the semigroup $\{\bar{S}(t)\}_{t\in[0,T]}$ in $\bar{L}_{A\gamma}^2$.
\begin{lem}\label{lem:S-I}
Let Assumptions \ref{Asp:gamma}-\ref{Asp:zeta}  hold, 
$f\in \bar{\mathbb{H}}_\gamma$ and $f^\prime\in \bar{L}_{A^2\gamma}^2$. Then for $t\in[0,\mathsf T]$,
\begin{equation*}
\|(\bar{S}(t)-I)f\|_{\bar{L}_{A\gamma}^2}\lesssim t^{\frac12}\big(\|f^{\prime}\|_{\bar{L}_{A^2\gamma}^2}+\|f\|_{\bar{\mathbb{H}}_\gamma}\big).
\end{equation*}
\end{lem}

\subsection{Convergence  of exponential Euler approximation on graph}
In this part, we study the convergence analysis of the exponential Euler approximation \eqref{eq:Un} for \eqref{eq:utzk-p}. Let us begin with the gradient estimate of \eqref{eq:utzk-p}.
\begin{prop}\label{prop:uD1}
Let $b,g$ be continuously differentiable with bounded derivatives, and  Assumptions \ref{Asp:zeta}-\ref{Asp:gamma-S} hold.
If \eqref{eq:noise} holds,
 $\psi^\wedge\in \bar{\mathbb{H}}_{\sqrt\gamma}$, and
 $(\psi^\wedge)^{\prime}\in \bar{L}_{A^2\gamma}^2$,
 then
\begin{gather*}
\E\left[\|\partial_z\bar{u}(t)\|^2_{\bar{L}^2_{A^2\gamma}}\right]
 \lesssim1+\|(\psi^\wedge)^\prime\|_{\bar{L}^2_{A^2\gamma}}^2+\|\psi^\wedge\|^2_{\bar{\mathbb{H}}_{\sqrt\gamma}}\quad\forall~t\in[0,\mathsf T].
\end{gather*}
\end{prop}
\begin{proof}
Under Assumption \ref{Asp:gamma-S}, $\gamma$ is uniformly bounded in $[0,\infty)$. This implies that $\|\cdot\|^2_{\bar{\mathbb{H}}_{\gamma}}\lesssim\|\cdot\|^2_{\bar{\mathbb{H}}_{\gamma}}$ since for any
$f\in \bar{\mathbb{H}}_{\sqrt\gamma}$,
\begin{equation}\label{eq:psiH}
\|f\|^2_{\bar{\mathbb{H}}_{\gamma}}
\le\sup_{z\in[0,\infty)}\sqrt{\gamma(z)}\|f\|^2_{\bar{\mathbb{H}}_{\sqrt\gamma}}\lesssim\|f\|^2_{\bar{\mathbb{H}}_{\sqrt\gamma}}.
\end{equation}
By \eqref{eq:ut}, 
the Burkholder and H\"older inequalities, and Lemma \ref{lem:partzbUf},
\begin{align}\label{eq:utA2}
\E\Big[\|\partial_z\bar{u}(t)\|_{\bar{L}^2_{A^2\gamma}}^2\Big]&\lesssim\|\partial_z\bar{S}(t)\psi^\wedge\|_{\bar{L}^2_{A^2\gamma}}^2+\int_0^t\E\Big[\|\partial_z[\bar{S}(t-r)B(\bar{u}(r))]\|_{\bar{L}^2_{A^2\gamma}}^2\Big]\ud r\\\notag
&\quad+\E\bigg[\Big\|\sum_{j=1}^\infty\int_0^t\partial_z\left[\bar{S}(t-r)\left(G(\bar{u}(r))(\widehat{\mathfrak{u}_j\mu})^{\wedge}\right)\right]\ud \beta_j(r)\Big\|_{\bar{L}^2_{A^2\gamma}}^2\bigg]\\\notag
&\lesssim
\|(\psi^\wedge)^\prime\|_{\bar{L}^2_{A^2\gamma}}^2+\|\psi^\wedge\|^2_{\bar{\mathbb{H}}_\gamma}\\\notag
&\quad+\int_0^t\E\Big[\|\partial_zB(\bar{u}(r))\|_{\bar{L}^2_{A^2\gamma}}^2\Big]\ud r+\int_0^t\E\left[\|B(\bar{u}(r))\|_{\bar{\mathbb{H}}_\gamma}^2\right]\ud r\\\notag
&\quad+\int_0^t\E\Big[\sum_{j=1}^\infty\left\|\partial_z\left[\bar{S}(t-r)\left(G(\bar{u}(r))(\widehat{\mathfrak{u}_j\mu})^{\wedge}\right)\right]\right\|_{\bar{L}^2_{A^2\gamma}}^2\Big]\ud r.
\end{align}
Making use of Lemma \ref{lem:partzbUf} again, it holds that for any $j\in\mathbb{N}_+$,
\begin{align}\label{eq:DSGu}
\big\|\partial_z[&\bar{S}(t-r)(G(\bar{u}(r))(\widehat{\mathfrak{u}_j\mu})^{\wedge})]\big\|_{\bar{L}^2_{A^2\gamma}}^2\lesssim\left\|\partial_zG(\bar{u}(r))\cdot(\widehat{\mathfrak{u}_j\mu})^{\wedge}\right\|_{\bar{L}^2_{A^2\gamma}}^2\\\notag
&+\Big\|G(\bar{u}(r))\frac{\ud}{\ud z}(\widehat{\mathfrak{u}_j\mu})^{\wedge}\Big\|_{\bar{L}^2_{A^2\gamma}}^2+\left\|G(\bar{u}(r))(\widehat{\mathfrak{u}_j\mu})^{\wedge}\right\|^2_{\bar{\mathbb{H}}_\gamma}.
\end{align}

For any  differentiable function $\varphi:\R^2\to\R$ with bounded derivative, it holds that
\begin{align*}
\frac{\ud}{\ud z}\varphi^{\wedge}(z)&=\frac{1}{T(z)} \frac{\ud}{\ud z}\oint_{C(z)} \frac{\varphi(x)}{|\nabla H(x)|} \ud l_z-\frac{T^\prime(z)}{T(z)^2}\oint_{C(z)} \frac{\varphi(x)}{|\nabla H(x)|} \ud l_z,
\quad z\in(0,\infty).
\end{align*}
According to \cite[Lemma 1.1]{FW12}, for any $z\in(0,\infty)$,
\begin{align*}
\frac{\ud}{\ud z}\oint_{C(z)} \frac{\varphi(x)}{|\nabla H(x)|} \ud l_z&=\oint_{C(z)}\frac{\nabla (\frac{\varphi}{|\nabla H|^2})(x)\cdot\nabla H(x)}{|\nabla H(x)|}+\frac{\varphi(x)}{|\nabla H(x)|^2}\frac{\Delta H(x)}{|\nabla H(x)|}\ud l_z\\
&=\oint_{C(z)}\frac{\nabla\varphi\cdot\nabla H}{|\nabla H|^3}-2\varphi\frac{(\nabla H)^\top\nabla^2H\nabla H}{|\nabla H|^5}+\varphi\frac{\Delta H}{|\nabla H|^3}\ud l_z.
\end{align*}
For the Hamiltonian $H(x)=|x|^2+\zeta(|x|^2)$ satisfying Assumption \ref{Asp:zeta},
by \eqref{eq:zeta1}, on every level set $C(z)=\{x\in\R^2:|x|=\sqrt{F(z)}\}$, one has that $\ud l_z=\sqrt{F(z)}\ud \theta$ and 
\begin{equation}\label{eq:FzH}
2\sqrt{F(z)}(1-r_0)\le|\nabla H(x)|=2|x|\big(1+\zeta^\prime(|x|^2)\big)\le2\sqrt{F(z)}(1+\tilde{r}_0).
\end{equation}
A direct calculation gives 
\begin{gather*}
\frac{\Delta H}{|\nabla H|^3}-2\frac{(\nabla H)^\top\nabla^2H\nabla H}{|\nabla H|^5}=-4\frac{|x|^2\zeta^{\prime\prime}(|x|^2)}{|\nabla H(x)|^3},\\
\frac{\nabla\varphi\cdot\nabla H}{|\nabla H|^3}-\varphi\frac{(\nabla H)^\top\nabla^2H\nabla H}{|\nabla H|^5}+\varphi\frac{\Delta H}{|\nabla H|^3}=\frac{2(1+\zeta^\prime(|x|^2))\langle x,\nabla\varphi(x)\rangle-4|x|^2\zeta^{\prime\prime}(|x|^2)\varphi(x)}{|\nabla H(x)|^3}.
\end{gather*}
Taking $\varphi=\widehat{\mathfrak{u}_j\mu}$, $j\in\mathbb N_+$, we obtain that for any $z\in(0,\infty)$, $\frac{\ud}{\ud z}(\widehat{\mathfrak{u}_j\mu})^{\wedge}(z)$ is equal to
\begin{align*}
&\frac{1}{T(z)}\oint_{C(z)}\frac{2(1+\zeta^\prime(|x|^2))\langle x,\nabla\widehat{\mathfrak{u}_j\mu}(x)\rangle-4|x|^2\zeta^{\prime\prime}(|x|^2)\widehat{\mathfrak{u}_j\mu}(x)}{|\nabla H(x)|^3}\ud l_z-\frac{T^\prime(z)}{T(z)^2}\oint_{C(z)} \frac{\widehat{\mathfrak{u}_j\mu}(x)}{|\nabla H(x)|} \ud l_z\\
&=\oint_{C(z)}\frac{2(1+\zeta^\prime(|x|^2))\langle x,\nabla\widehat{\mathfrak{u}_j\mu}(x)\rangle}{|\nabla H(x)|^2}\ud \mu_z+\oint_{C(z)}\left(\frac{-4|x|^2\zeta^{\prime\prime}(|x|^2)}{|\nabla H(x)|^2}- \frac{T^\prime(z)}{T(z)}\right)\widehat{\mathfrak{u}_j\mu}(x) \ud \mu_z,
\end{align*}
where $\ud\mu_z=\frac{1}{T(z)}\frac{1}{|\nabla H(x)|}\ud l_z$ is a probability measure supported on $C(z)$.
Since $T(z)\ge\pi/(1+\tilde{r}_0)$ for all $z\in[0,\infty)$, it follows from the Cauchy--Schwarz inequality, \eqref{eq:zeta1}-\eqref{eq:zeta3}, \eqref{eq:T}, and Lemma \ref{lem:AT1} that for any $j\in\mathbb{N}_+$,
\begin{align*}
&\Big\|G(\bar{u}(r))\frac{\ud}{\ud z}(\widehat{\mathfrak{u}_j\mu})^{\wedge}\Big\|_{\bar{L}^2_{A^2\gamma}}^2\\
&\lesssim \int_0^\infty|g(\bar{u}(r,z))|^2\bigg|\oint_{C(z)}2(1+\zeta^\prime(|x|^2))\frac{\langle x,\nabla\widehat{\mathfrak{u}_j\mu}(x)\rangle}{|\nabla H(x)|^2}\ud \mu_z\bigg|^2
A(z)^2\gamma(z)\ud z\\
&\quad+ \int_0^\infty|g(\bar{u}(r,z))|^2\bigg|\oint_{C(z)}\left(\frac{4|x|^2\zeta^{\prime\prime}(|x|^2)}{|\nabla H(x)|^2}+\frac{T^\prime(z)}{T(z)}\right)\widehat{\mathfrak{u}_j\mu}(x) \ud \mu_z\bigg|^2
A(z)^2\gamma(z)\ud z\\
&\lesssim \int_0^\infty|g(\bar{u}(r,z))|^2\oint_{C(z)}\frac{|\langle x,\nabla\widehat{\mathfrak{u}_j\mu}(x)\rangle|^2}{|\nabla H(x)|^4}\ud \mu_z
A(z)^2\gamma(z)\ud z\\
&\quad+ \int_0^\infty|g(\bar{u}(r,z))|^2\oint_{C(z)}\left(1+|\nabla H(x)|^{-4}\right)|\widehat{\mathfrak{u}_j\mu}(x)|^2 \ud \mu_z
A(z)^2\gamma(z)\ud z.
\end{align*}
By the Parserval inequality and \eqref{eq:noise},
\begin{gather*}
\sum_{j=1}^\infty|\langle x,\nabla\widehat{\mathfrak{u}_j\mu}(x)\rangle|^2=\sum_{j=1}^\infty\Big|\int_{\R^2}\langle x,\textup{i}\xi\rangle e^{\textup{i}x\cdot\xi}\mathfrak{u}_j(\xi)\mu(\ud\xi)\Big|^2=\int_{\R^2}|\langle x,\textup{i}\xi\rangle e^{\textup{i}x\cdot\xi}|^2\mu(\ud \xi)\lesssim |x|^2.
\end{gather*}
In this way, by using \eqref{eq:ujmu0} and the linear growth of $g$,
\begin{align*}
\sum_{j=1}^\infty\Big\|G(\bar{u}(r))\frac{\ud}{\ud z}(\widehat{\mathfrak{u}_j\mu})^{\wedge}\Big\|_{\bar{L}^2_{A^2\gamma}}^2
\lesssim\int_0^\infty|g(\bar{u}(r,z))|^2\oint_{C(z)}\frac{|x|^2}{|\nabla H(x)|^4}\ud \mu_z
A(z)^2\gamma(z)\ud z\\
+\int_0^\infty|g(\bar{u}(r,z))|^2\oint_{C(z)}\left(1+|\nabla H(x)|^{-4}\right)\ud \mu_z
A(z)^2\gamma(z)\ud z.
\end{align*}
According to \eqref{eq:FzH}, \eqref{eq:A} and \eqref{eq:F}, one has that $|x|\lesssim |\nabla H(x)|$ and
 $|\nabla H(x)|^{-1}\lesssim F(z)^{-\frac12}\lesssim A(z)^{-\frac12}$ for any $x\in C(z)$ and $z>0$. Hence,
\begin{align*}
\sum_{j=1}^\infty\Big\|G(\bar{u}(r))\frac{\ud}{\ud z}(\widehat{\mathfrak{u}_j\mu})^{\wedge}\Big\|_{\bar{L}^2_{A^2\gamma}}^2
&\lesssim\int_0^\infty|g(\bar{u}(r,z))|^2(1+z+z^2)\gamma(z)\ud z\\
&\lesssim\int_0^\infty\left(1+|\bar{u}(r,z)|^2\right)\left(\gamma(z)+\sqrt{\gamma(z)}\right)\ud z\\
&\lesssim 1+\|\bar{u}(r)\|_{\bar{\mathbb{H}}_{\sqrt\gamma}}^2+\|\bar{u}(r)\|_{\bar{\mathbb{H}}_{\gamma}}^2\lesssim 1+\|\bar{u}(r)\|_{\bar{\mathbb{H}}_{\sqrt\gamma}}^2,
\end{align*}
where we have also used \eqref{eq:psiH}, $\gamma,\sqrt{\gamma}\in L^1(0,\infty)$ and the fact that $z\mapsto (z+z^2)\sqrt{\gamma(z)}$ is uniformly bounded on $[0,\infty)$ because $\gamma$ decays exponentially fast at infinity.
By summing over $j\in\mathbb N_+$ on both sides of \eqref{eq:DSGu}, one deduces that
\begin{align}\label{eq:DSGu-1}
&\sum_{j=1}^\infty\big\|\partial_z[\bar{S}(t-r)G(\bar{u}(r))(\widehat{\mathfrak{u}_j\mu})^{\wedge}]\big\|_{\bar{L}^2_{A^2\gamma}}^2\lesssim\left\|\partial_z\bar{u}(r)\right\|_{\bar{L}^2_{A^2\gamma}}^2+1+\|\bar{u}(r)\|_{\bar{\mathbb{H}}_{\sqrt\gamma}}^2,
\end{align}
thanks to \eqref{eq:barGlin-growth} and \eqref{eq:psiH}.
Gathering \eqref{eq:psiH}, \eqref{eq:utA2}, \eqref{eq:utHgamma} and  \eqref{eq:DSGu-1} together leads to
\begin{align*}
\E\left[\|\partial_z\bar{u}(t)\|_{\bar{L}^2_{A^2\gamma}}^2\right]
&\lesssim\|(\psi^\wedge)^\prime\|_{\bar{L}^2_{A^2\gamma}}^2+\|\psi^\wedge\|^2_{\bar{\mathbb{H}}_\gamma}+\int_0^t\E\Big[\|\partial_z\bar{u}(r)\|_{\bar{L}^2_{A^2\gamma}}^2\Big]\ud r\\
&\quad+1+\int_0^t\E\left[\|\bar{u}(r)\|_{\bar{\mathbb{H}}_{\sqrt\gamma}}^2\right]\ud r+\int_0^t\E\left[\|\bar{u}(r)\|_{\bar{\mathbb{H}}_{\gamma}}^2\right]\ud r\\
&\lesssim1+\|(\psi^\wedge)^\prime\|_{\bar{L}^2_{A^2\gamma}}^2+\|\psi^\wedge\|^2_{\bar{\mathbb{H}}_{\sqrt\gamma}} +\int_0^t\E\Big[\|\partial_z\bar{u}(r)\|_{\bar{L}^2_{A^2\gamma}}^2\Big]\ud r,
\end{align*}
which in combination with the Gronwall inequality completes the proof.
\end{proof}

Based on the above gradient estimate, we further present the temporal H\"older continuity of the solution $\bar{u}$ to the limiting equation \eqref{eq:utzk-p}.

\begin{prop}\label{eq:baru-Holder}
Under the same conditions of Proposition \ref{prop:uD1}, 
\begin{align*}
\E\left[\|\bar{u}(t)-\bar{u}(s)\|_{\bar{L}^2_{A\gamma}}^2\right]\le C(t-s)\quad \text{for any }0\le s<t\le \mathsf T.
\end{align*}
\end{prop}

\begin{proof}
Assume for simplicity that $b\equiv 0$.
Since $z\mapsto z\sqrt{\gamma(z)}$ is uniformly bounded on $[0,\infty)$, it follows from Lemma \ref{lem:TFA} that
$|A(z)\gamma(z)|\lesssim z\gamma(z)\lesssim\sqrt{\gamma(z)}T(z)$ for all $z\in[0,\infty)$,
which indicates that
$\|\cdot\|_{\bar{L}^2_{A\gamma}}\lesssim\|\cdot\|_{\bar{\mathbb{H}}_{\sqrt\gamma}}$.
Hence, by means of Lemma \ref{lem:S-I} and the Burkholder inequality,
\begin{align*}
&\quad\ \E\left[\|\bar{u}(t)-\bar{u}(s)\|_{\bar{L}_{A\gamma}^2}^2\right]\\
&\lesssim\E\left[\|(\bar{S}(t-s)-I)\bar{u}(s)\|_{\bar{L}_{A\gamma}^2}^2\right]+\E\left[\Big\|\int_s^t\bar{S}(t-r)G(\bar{u}(r))\ud \bar{\mathcal W}(r)\Big\|_{\bar{L}_{A\gamma}^2}^2\right]\\
&\lesssim (t-s)\Big(\E\big[\|\partial_z\bar{u}(s)\|_{\bar{L}_{A^2\gamma}^2}^2\big]+\E\big[\|\bar{u}(s)\|_{\bar{\mathbb{H}}_{\gamma}}^2\big]\Big)+\E\left[\Big\|\int_s^t\bar{S}(t-r)G(\bar{u}(r))\ud \bar{\mathcal W}(r)\Big\|_{\bar{\mathbb{H}}_{\sqrt\gamma}}^2\right]\\
&\lesssim (t-s)+\int_s^t\sum_{j=1}^\infty\E\left[\left\|\bar{S}(t-r)G(\bar{u}(r))(\widehat{\mathfrak{u}_j\mu})^{\wedge}\right\|_{\bar{\mathbb{H}}_{\sqrt\gamma}}^2\right]\ud r,
\end{align*}
where in the last step we also used Proposition \ref{prop:uD1} and the first inequality of \eqref{eq:utHgamma}. Using Lemma \ref{lem:Sf1}$(i)$ with $\gamma$ replaced by $\sqrt\gamma$, the linear growth of $g$, the second inequality of \eqref{eq:utHgamma}, and \eqref{eq:ujmu1}, it holds that for any $r\in[s,t]\subset[0,\mathsf T]$,
\begin{align}\label{eq:SGu}
&\sum_{j=1}^\infty\E\left[\left\|\bar{S}(t-r)G(\bar{u}(r))(\widehat{\mathfrak{u}_j\mu})^\wedge\right\|_{\bar{\mathbb{H}}_{\sqrt\gamma}}^2\right]\lesssim\sum_{j=1}^\infty\E\left[\left\|G(\bar{u}(r))(\widehat{\mathfrak{u}_j\mu})^\wedge\right\|_{\bar{\mathbb{H}}_{\sqrt\gamma}}^2\right]\\\notag
&\lesssim\mu(\R^2)\E\left[\left\|G(\bar{u}(r))\right\|_{\bar{\mathbb{H}}_{\sqrt\gamma}}^2\right]\lesssim1+\E\left[\left\|\bar{u}(r)\right\|_{\bar{\mathbb{H}}_{\sqrt\gamma}}^2\right]\lesssim1,
\end{align}
which completes the proof.
\end{proof}

Motivated by Lemma \ref{lem:S-I},
we measure
the difference between $\bar{U}^N$ and $\bar{u}(\mathsf T)$ in the graph weighted space $\bar{L}^2_{A\gamma}$.

\begin{tho}\label{theo:MS-1D}
Let $b,g$ be continuously differentiable with bounded derivatives, and  Assumptions \ref{Asp:zeta}-\ref{Asp:gamma-S}  hold.
If \eqref{eq:noise} holds,
 $\psi^\wedge\in \bar{\mathbb{H}}_{\sqrt\gamma}$, and
 $(\psi^\wedge)^{\prime}\in \bar{L}_{A^2\gamma}^2$,
 then
\begin{gather*}
\mathbb{E}\left[\|\bar{u}(t_n)-\bar{U}^{n}\|^2_{\bar{L}^2_{A\gamma}}\right]\lesssim \tau\quad\forall~n=1,\ldots,N.
\end{gather*}
\end{tho}

\begin{proof}
By \eqref{eq:ut} and \eqref{eq:Un}, similarly as in the proof of \eqref{eq:u-U}, we have
\begin{align}\label{eq:u-Ubar}
&\mathbb{E}\left[\|\bar{u}(t_n)-\bar{U}^{n}\|^2_{\bar{L}^2_{A\gamma}}\right]\\\notag
&\lesssim\sum_{j=0}^{n-1}\int_{t_{j}}^{t_{j+1}}\mathbb{E}\left[\left\|\left(\bar{S}(t_n-s)-\bar{S}(t_n-t_j)\right)B(\bar{u}(s))\right\|^2_{\bar{L}^2_{A\gamma}}\right]\ud s\\\notag
&\quad+ \sum_{j=0}^{n-1}\int_{t_{j}}^{t_{j+1}}\sum_{l=1}^\infty\mathbb{E}\left[\left\|\left(\bar{S}(t_n-s)-\bar{S}(t_n-t_j)\right)G(\bar{u}(s))(\widehat{\mathfrak{u}_l\mu})^\wedge\right\|^2_{\bar{L}^2_{A\gamma}}\right]\ud s\\\notag
&\quad+\sum_{j=0}^{n-1}\int_{t_{j}}^{t_{j+1}}\mathbb{E}\left[\|\bar{u}(s)-\bar{u}(t_{j})\|_{\bar{L}^2_{A\gamma}}^2\right]\ud s+\sum_{j=0}^{n-1}\int_{t_{j}}^{t_{j+1}}\mathbb{E}\left[\|\bar{u}(t_{j})-\bar{U}^{j}\|_{\bar{L}^2_{A\gamma}}^2\right]\ud s.
\end{align}
In virtue of Lemma \ref{lem:S-I}, for any $s\in(t_{j},t_{j+1}]$ and $l\in\mathbb N_+$,
\begin{align*}
& \mathbb{E}\left[\left\|\left(\bar{S}(t_n-s)-\bar{S}(t_n-t_j)\right)G(\bar{u}(s))(\widehat{\mathfrak{u}_l\mu})^\wedge\right\|^2_{\bar{L}^2_{A\gamma}}\right]\\
&=\mathbb{E}\left[\left\|\left(\bar{S}(s-t_j)-I\right)\bar{S}(t_n-s)G(\bar{u}(s))(\widehat{\mathfrak{u}_l\mu})^\wedge\right\|^2_{\bar{L}^2_{A\gamma}}\right]\\
&\lesssim\tau\mathbb{E}\left[\left\|\partial_z[\bar{S}(t_n-s)G(\bar{u}(s))(\widehat{\mathfrak{u}_l\mu})^\wedge]\right\|^2_{\bar{L}^2_{A^2\gamma}}\right]+\tau\mathbb{E}\left[\left\|\bar{S}(t_n-s)G(\bar{u}(s))(\widehat{\mathfrak{u}_l\mu})^\wedge\right\|^2_{\bar{\mathbb{H}}_{\gamma}}\right].
\end{align*}
Repeating the proof of \eqref{eq:SGu}, one has that 
$\sum_{l=1}^\infty\mathbb{E}[\|\bar{S}(t_n-s)G(\bar{u}(s))(\widehat{\mathfrak{u}_l\mu})^\wedge\|^2_{\bar{\mathbb{H}}_{\gamma}}]\lesssim 1.$
On the other hand, by \eqref{eq:DSGu-1} with $t=t_n$ and $r=s$, \eqref{eq:utHgamma} and Proposition \ref{prop:uD1},
\begin{align*}
&\sum_{l=1}^\infty\mathbb{E}\left[\left\|\partial_z[\bar{S}(t_n-s)G(\bar{u}(s))(\widehat{\mathfrak{u}_l\mu})^\wedge]\right\|^2_{\bar{L}^2_{A^2\gamma}}\right]
\lesssim 1.
\end{align*}
Combining the previous three estimates, we deduce that 
$$\sum_{l=1}^\infty\mathbb{E}\left[\left\|\left(\bar{S}(t_n-s)-\bar{S}(t_n-t_j)\right)G(\bar{u}(s))(\widehat{\mathfrak{u}_l\mu})^\wedge\right\|^2_{\bar{L}^2_{A\gamma}}\right]\lesssim \tau.$$
In the similar manner, it can be also verified that
$$\mathbb{E}\left[\left\|\left(\bar{S}(t_n-s)-\bar{S}(t_n-t_j)\right)B(\bar{u}(s))\right\|^2_{\bar{L}^2_{A\gamma}}\right]\lesssim \tau.$$
Hence, by Proposition \ref{eq:baru-Holder} and \eqref{eq:u-Ubar}, we derive that
\begin{align*}
\mathbb{E}\left[\|\bar{u}(t_n)-\bar{U}^{n}\|^2_{\bar{L}^2_{A\gamma}}\right]\le C\tau+C\tau\sum_{j=0}^{n-1}\mathbb{E}\left[\|\bar{u}(t_{j})-\bar{U}^{j}\|_{\bar{L}^2_{A\gamma}}^2\right],
\end{align*}
and then using the discrete Gronwall inequality finishes the proof.
\end{proof}

 \subsection{AP property of exponential Euler approximation}

In view of  Theorem \ref{theo:MS-1D}, it is natural to 
 consider the weighted $L^2$-space $\mathbb{H}_{A\gamma}:=L^2(\R^2,A(H(x))\gamma(H(x))\ud x)$. As a result, one can obtain that
$$
\mathbb{E}[\|\bar{u}(\mathsf{T})^\vee-(\bar{U}^{N})^\vee\|^2_{\mathbb{H}_{A\gamma}}]=\E\int_{\R^2}|\bar{u}(\mathsf{T})^\vee(x)-(\bar{U}^{N})^\vee(x)|^2A(H(x))\gamma(H(x))\ud x\lesssim\tau,
$$
since $\|f\|_{\bar{L}^2_{A\gamma}}\lesssim\|f^\vee\|_{\mathbb{H}_{A\gamma}}\lesssim\|f\|_{\bar{L}^2_{A\gamma}}$ for any $f\in\bar{L}^2_{A\gamma}$ due to \eqref{eq:T}. 
 This motivates us to verify the asymptotical behavior \eqref{eq:UNep-u} in  $L^2(\Omega,\mathbb{H}_{A\gamma})$, as the following theorem shows.

\begin{tho}\label{tho:AP}
Let $b,g$ be continuously differentiable with bounded derivatives, \eqref{eq:Tneq0},
 \eqref{eq:noise}, and  Assumptions \ref{Asp:zeta}-\ref{Asp:gamma-S}  hold.
Assume that $\psi^{\wedge}\in \bar{\mathbb{H}}_{\sqrt\gamma}$ and
 $(\psi^{\wedge})^{\prime}\in \bar{L}_{A^2\gamma}^2$. Then
 \begin{gather}\label{eq:AP}
\lim_{N\to\infty}\lim_{\epsilon\to 0}U_\epsilon^N
=\bar{u}(\mathsf T)^{\vee}\quad \text{in } L^2(\Omega,\mathbb{H}_{A\gamma}). 
 \end{gather}
 If in addition $\psi=(\psi^{\wedge})^{\vee}$ and $(\psi^{\wedge})^{\prime}\in \bar{L}_{A\gamma}^2$, then
  \begin{gather}\label{eq:AP1}
\lim_{\epsilon\to 0}\lim_{N\to\infty}U_\epsilon^N=\bar{u}(\mathsf T)^\vee\quad \text{in } L^2(\Omega,\mathbb{H}_{\gamma}).
 \end{gather}

\end{tho}
\begin{proof}
 Under Assumption \ref{Asp:gamma-S},
  Assumption \ref{Asp:gamma} holds with $\gamma$ replaced by $\sqrt{\gamma}$, and thus we can apply \cite[Theorem 5.3]{CF19} and
Theorem \ref{theo:asy} to obtain 
\begin{equation}\label{eq:sqrtgamma}
\lim_{\epsilon\to 0}\E\left[\|u_{\epsilon}(\mathsf T)-(\bar u(\mathsf T))^{\vee}\|_{\mathbb{H}_{\sqrt\gamma}}^2\right]=0,
\quad \text{and}\quad
\lim_{\epsilon\to 0}\E\left[\|U_{\epsilon}^{N}-(\bar U^{N})^{\vee}\|_{\mathbb{H}_{\sqrt\gamma}}^2\right]=0,
\end{equation}
respectively.
Since $A\sqrt{\gamma}$ is bounded on $[0,\infty)$, one has
$\|\cdot\|_{\mathbb{H}_{A\gamma}}\lesssim\|\cdot\|_{\mathbb{H}_{\sqrt{\gamma}}}$.
Note that
\begin{align*}
\lim_{\epsilon\to0}\E\left[\|U_{\epsilon}^{N}-\bar u(\mathsf T)^\vee\|^2_{\mathbb{H}_{A\gamma}}\right]
&\lesssim\lim_{\epsilon\to0}\E\left[\|U_{\epsilon}^{N}-(\bar U^{N})^{\vee}\|_{\mathbb{H}_{\sqrt\gamma}}^2\right]+\E\left[\|\bar{u}(\mathsf{T})^\vee-(\bar{U}^{N})^\vee\|^2_{\mathbb{H}_{A\gamma}}\right]\\
&\lesssim\E\left[\|\bar{u}(\mathsf{T})^\vee-(\bar{U}^{N})^\vee\|^2_{\mathbb{H}_{A\gamma}}\right]\lesssim\tau,
\end{align*}
for any $N\ge1$,
which yields  \eqref{eq:AP} by taking $N=\mathsf{T}/\tau\to\infty$.

For any radial $\varphi:\R^2\to\R$, 
$$
\|\varphi\|^2_{\mathbb{H}_{\sqrt\gamma}}
=\int_0^\infty\oint_{C(z)}\frac{|\varphi^\wedge(H(x))|^2\sqrt{\gamma(H(x))}}{|\nabla H(x)|}\ud l_z\ud z
=\|\varphi^\wedge\|^2_{\bar{\mathbb{H}}_{\sqrt\gamma}}.
$$
Similarly, for any radial $\varphi:\R^2\to\R$, by $\nabla\varphi(x)=\nabla\varphi^{\wedge}(H(x))=(\varphi^{\wedge})^\prime(H(x))\nabla H(x)$, 
\begin{align*}
\|\nabla\varphi\|^2_{\mathbb{H}_{\gamma}}
&=\int_0^\infty\oint_{C(z)}\frac{|(\varphi^{\wedge})^\prime(H(x))\nabla H(x)|^2{\gamma(H(x))}}{|\nabla H(x)|}\ud l_z\ud z=2\|(\varphi^\wedge)^\prime\|^2_{\bar{L}^2_{A\gamma}}.
\end{align*}
The assumptions $\psi^{\wedge}\in \bar{\mathbb{H}}_{\sqrt\gamma}$ and $(\psi^{\wedge})^{\prime}\in \bar{L}_{A\gamma}^2$ ensure that $\psi\in\mathbb{H}_{\sqrt\gamma}$ and $\nabla\psi\in\mathbb{H}_{\gamma}$.
Hence one can apply Theorem \ref{theo:MS-2D} to deduce that for any $\epsilon\in(0,1]$,
 \begin{align*}
\lim_{N\to \infty}\E\left[\|U_{\epsilon}^{N}-(\bar u(\mathsf T))^{\vee}\|_{\mathbb{H}_\gamma}^2\right]
&\lesssim\E\left[\|u_{\epsilon}(\mathsf T)-(\bar u(\mathsf T))^{\vee}\|_{\mathbb{H}_{\gamma}}^2\right]+\lim_{N\to \infty}\mathbb{E}\left[\|u_{\epsilon}(\mathsf{T})-U_{\epsilon}^{N}\|^2_{\mathbb{H}_{\gamma}}\right]\\
&\le\E\left[\|u_{\epsilon}(\mathsf T)-(\bar u(\mathsf T))^{\vee}\|_{\mathbb{H}_{\gamma}}^2\right],
\end{align*}
which along with  \eqref{eq:u-u}  
 proves \eqref{eq:AP1}.
\end{proof}

 Theorem \ref{tho:AP} reveals the AP property of the exponential Euler approximation \eqref{eq:EEM} for \eqref{eq:SPDE}, i.e., as the discretization parameter $N=\mathsf T/\tau$ tends to infinity,
 the fast advection asymptotical limit of the exponential Euler approximation \eqref{eq:EEM}  is consistent with that of the solution to \eqref{eq:SPDE}.

 \begin{coro}\label{coro:AP}
 Let $b,g$ be continuously differentiable with bounded derivatives, and $\gamma$  given by Example \ref{Ex:gamma-S1} with $\lambda>4$ or Example \ref{Ex:gamma-S}. Assume that \eqref{eq:Tneq0},
 \eqref{eq:noise} and  Assumption \ref{Asp:zeta}  hold. If 
 $\psi=(\psi^{\wedge})^{\vee}$, where $\int_0^\infty|\psi^{\wedge}(z)|^2\gamma^{\frac14}(z)\ud z<\infty$ and
 $(\psi^{\wedge})^{\prime}\in \bar{L}_{A^2\gamma}^2\cap \bar{L}_{A\sqrt{\gamma}}^2$, then the following limits exchange
  \begin{gather*}
\lim_{N\to\infty}\lim_{\epsilon\to 0}U_\epsilon^N=\bar{u}(\mathsf T)^\vee=\lim_{\epsilon\to 0}\lim_{N\to\infty}U_\epsilon^N\quad \text{in }  L^2(\Omega,\mathbb{H}_{A\gamma}).
 \end{gather*}
 \end{coro}
 \begin{proof}
 By Remark \ref{rem:A3}, if $\gamma$ is given by Example \ref{Ex:gamma-S1} with $\lambda>4$ or Example \ref{Ex:gamma-S}, then Assumption \ref{Asp:gamma-S} also holds with $\gamma$ replaced by $\sqrt\gamma$. As a result, one can applying Theorem \ref{theo:MS-2D}  with $\gamma$ replaced by $\sqrt\gamma$ to deduce that for any $\epsilon>0$, $$\lim_{N\to \infty}\mathbb{E}\left[\|u_{\epsilon}(\mathsf{T})-U_{\epsilon}^{N}\|^2_{\mathbb{H}_{\sqrt\gamma}}\right]=0,$$
 thanks to $\int_0^\infty|\psi^{\wedge}(z)|^2\gamma^{\frac14}(z)\ud z<\infty$ and $(\psi^{\wedge})^{\prime}\in  \bar{L}_{A\sqrt{\gamma}}^2$. The above equality together with the first equality in \eqref{eq:sqrtgamma} yields $\lim_{\epsilon\to 0}\lim_{N\to\infty}U_\epsilon^N=\bar{u}(\mathsf T)^\vee$ in $L^2(\Omega,\mathbb{H}_{A\gamma})$. Finally taking \eqref{eq:AP} into account finally completes the proof.
 \end{proof}

\section{Numerical experiments}\label{S6}

In this section, we implement several numerical experiments to verify our main theoretic analysis, including the mean square convergence analysis for a fixed $\epsilon>0$
and the asymptotic behavior of the exponential Euler approximations in the small $\epsilon$ limit. In the numerical experiments, we take $b\equiv0$ and focus on
 $H(x)=|x|^2+\zeta(|x|^2)$
 with $\zeta(z)=1-e^{-\frac{1}{2}z}$ or $\zeta(z)=\sqrt{1+z}-1$ for $z\in[0,\infty)$.
   \begin{enumerate}
\item[(1)] For $\zeta(z)=1-e^{-\frac{1}{2}z}$, we can explicitly compute that $\zeta^\prime(z)=\frac{1}{2}e^{-\frac{1}{2}z}$, $\zeta^{\prime\prime}(z)=-\frac14e^{-\frac{1}{2}z}$ and 
$F(z)=(Id+\zeta)^{-1}(z)=
 z+2\textup{W}_0(\frac12{\mathrm{e}}^{\frac{1}{2}-\frac{z}{2}})-1$, $z\in[0,\infty)$, where $\textup{W}_0$ is the principal branch of the Lambert W function, i.e., $w=\textup{W}_0(z)$ if and only if $we^w=z$. 
 \item[(2)] For $\zeta(z)=\sqrt{1+z}-1$, it holds that $\zeta^\prime(z)=\frac{1}{2}(1+z)^{-\frac12}$, $\zeta^{\prime\prime}(z)=-\frac{1}{4}(1+z)^{-\frac32}$, and $F(z)=z-\frac{1}{2}\sqrt{4z+9}+\frac32$.
 \end{enumerate}
 
\subsection{Convergence orders}\label{Sec:5.1}
 In order to numerically solve 
\eqref{eq:SPDE},
we first truncate the domain $\R^2$ to $[-L,L]\times[-L,L]$ and endow the truncated problem 
\begin{equation}\label{eq:SPDE-trun}
\partial_t \textbf{u}_{\epsilon}(t,x)=\frac{1}{2} \Delta \textbf{u}_{\epsilon}(t,x)+\frac{1}{\epsilon}\left\langle\nabla^{\perp} H(x), \nabla \textbf{u}_{\epsilon}(t,x)\right\rangle+g(\textbf{u}_{\epsilon}(t,x))\partial_t\mathcal{W}(t,x)
\end{equation}
in $[0,\mathsf T]\times[-L,L]^2$
 with zero Dirichlet boundary condition and the initial value $\textbf{u}_{\epsilon}(0,x)=\psi(x)$. To realize the numerical simulation in computers,
 we combine the exponential Euler approximation \eqref{eq:EEM} 
 with a central finite difference method in space to obtain a fully 
discrete numerical approximation $\{\mathbf{u}^{M,\tau}_{\epsilon,n}\}_{n=0}^N$ for \eqref{eq:SPDE-trun}
 (see Appendix \ref{App:B1} for more details), where $h=L/M$ and $\tau=\mathsf T/N$ are the space and time stepsizes, respectively.
 \begin{figure}[!htb]
\centering
\includegraphics[width=0.8\linewidth]{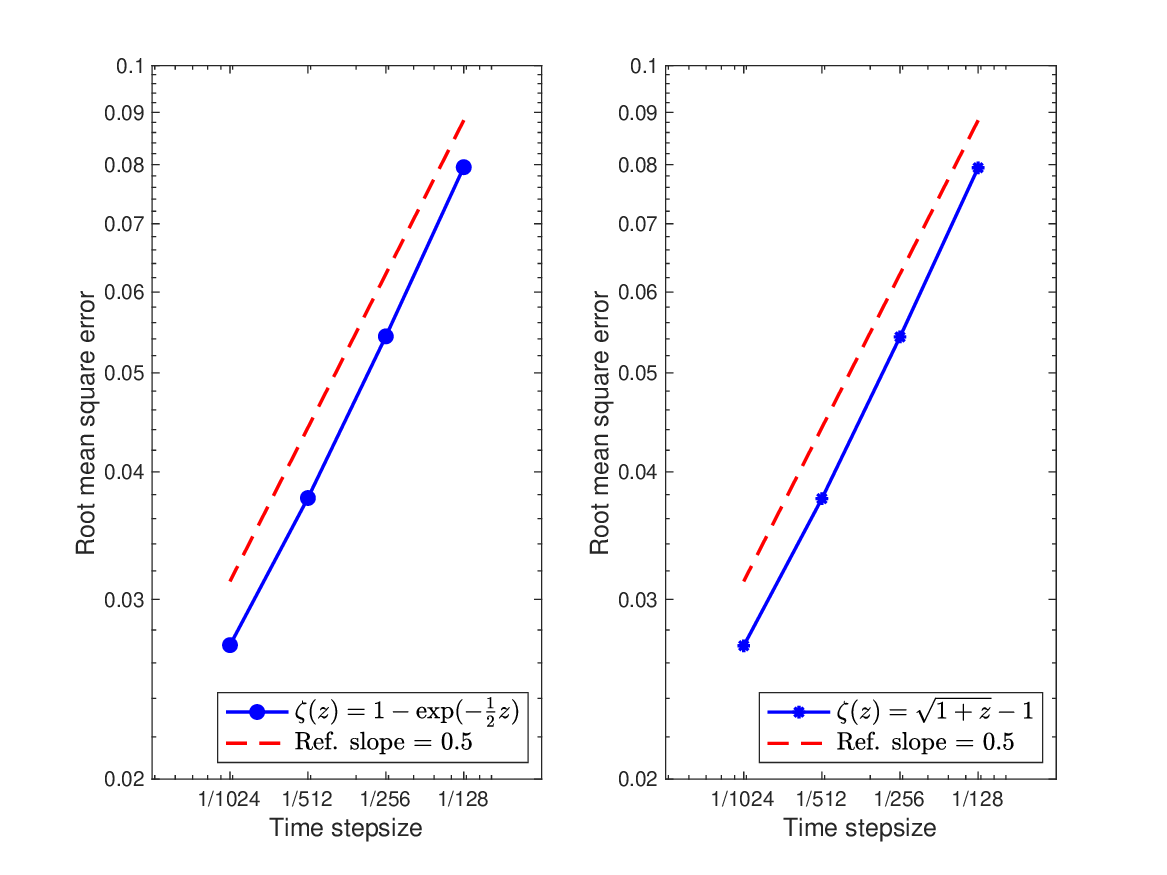}
\caption{Errors of the exponential Euler approximation for \eqref{eq:SPDE-trun} with $L=1$,  $\epsilon=1$,  $h=0.2$, $P=500$.}\label{fig:1}
\end{figure}

\begin{figure}[!htb]
\centering
\includegraphics[width=0.8\linewidth]{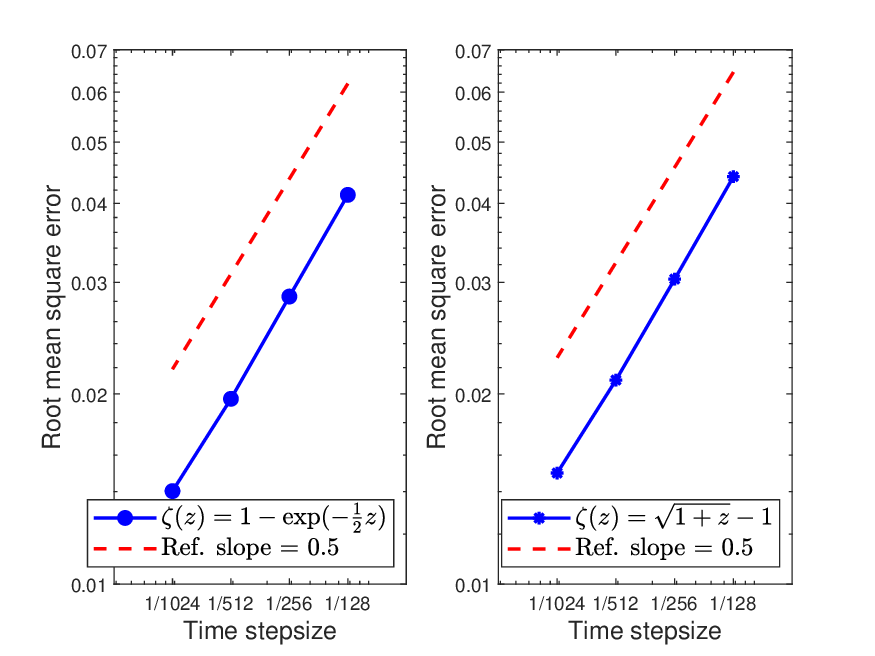}
\caption{ Errors of the exponential Euler approximation for \eqref{eq:utzk-p-trun} with $L=10$, $h=0.1$ and $\textup{P}=1000$.}\label{fig:2}
\end{figure}
As for the numerical simulation of limiting equation \eqref{eq:utzk-p},
we truncate the interval $[0,\infty)$ into $[0,L]$ and 
obtain the associated truncated problem
\begin{equation}\label{eq:utzk-p-trun}
\partial_t \bar{u}(t, z)=\frac{A(z)}{T(z)} \partial_{zz}\bar{u}(t, z)+\frac{A^\prime(z)}{T(z)}\partial_z\bar{u}(t, z)+g(\bar{u}(t, z)) \partial_t \bar{\mathcal{W}}(t, z)
\end{equation}
in $[0,\mathsf T]\times(0,L]$
with $\bar{u}(0, z)=\psi^{\wedge}(z)$ and the boundary conditions
$\frac{A(0)}{T(0)}\bar{u}^{\prime\prime}(t,0)=0$ (see e.g., \cite{AM13}) and $\bar{u}(t,L)=0$.  A fully discrete numerical approximation $\{\bar{u}_{n}^{M,\tau}\}_{n=0}^N$
 of \eqref{eq:utzk-p-trun} combining the exponential Euler approximation \eqref{eq:Un} and a finite difference method is presented in Appendix \ref{App:B2}. 
Since the present work mainly focuses on the time discretizations, we don't discuss the domain truncation errors of \eqref{eq:SPDE-trun} and \eqref{eq:utzk-p-trun}, as well as the errors of spatial discretizations here. We hope to investigate these issues in the future.

In order to verify Theorem \ref{theo:MS-2D}, we measure the error
\begin{align*}
\mathbb{E}\sum_{i=1-M}^{M-1}\sum_{j=1-M}^{M-1}\left|\mathbf{u}_\epsilon(\mathsf T,ih,jh)-[\mathbf{u}^{M,\tau}_{\epsilon,N}]_{k(i,j)}\right|^2\gamma(H(ih,jh))h^2,
  \end{align*}
where $[\mathbf{u}^{M,\tau}_{\epsilon,N}]_{k(i,j)}$ denotes the $k(i,j)$th component of $\mathbf{u}^{M,\tau}_{\epsilon,N}$ (see \eqref{eq:kij}).
Since the exact solution of \eqref{eq:SPDE-trun} is unavailable, we measure the mean square error of the numerical solutions via 
\begin{small}
\begin{align}\label{eq:Error2d}
\left(\frac{1}{\textup{P}}\sum_{p=1}^\textup{P}\sum_{i=1-M}^{M-1}\sum_{j=1-M}^{M-1}\left|[\mathbf{u}^{M,\tau_{l+1}}_{\epsilon,N}]_{k(i,j)}(\omega_p)-[\mathbf{u}^{M,\tau_l}_{\epsilon,N}]_{k(i,j)}(\omega_p)\right|^2\gamma(H(ih,jh))h^2\right)^{\frac12}.
  \end{align}
  \end{small}We set $g(u)=\sin(u)$, $\gamma(z)=\exp(-\sqrt{z})$, $\Lambda(x)=\frac{1}{\pi}e^{-|x|^2}$ (the Fourier transform of the spectral measure $\mu$ of $\mathcal W$), and $\psi(x)=\exp(-H(x))$ to measure the error at the endpoint $\mathsf T=1/8$  with different time stepsizes $\tau_l=2^{-(6+l)}$, $l=1,\ldots,5$, and the expectations are approximated by computing averages over $\textup{P}$ samples. In Fig.\ \ref{fig:1}, we depict the mean square convergence error \eqref{eq:Error2d} due to the temporal discretization on a log-log scale, together with the reference line with slope $\frac12$. 
Analogously, to verify Theorem \ref{theo:MS-1D}, we display the error 
\begin{align*}
\left(\frac{1}{\textup{P}}\sum_{p=1}^\textup{P}\sum_{i=0}^{M-1}\left|[\bar{u}^{M,\tau_{l+1}}_{N}]_{i}(\omega_p)-[\bar{u}^{M,\tau_{l}}_{N}]_{i}(\omega_p)\right|^2A(ih)\gamma(ih)h\right)^{\frac12}
  \end{align*}
 due to the temporal discretization on a log-log scale in Fig.\ \ref{fig:2}. Observe that the numerical results in Figs.\ \ref{fig:1}-\ref{fig:2} are consistent with theoretical mean square convergence order $\frac12$ of exponential Euler approximations.

\subsection{Fast advection asymptotics}
 In this part, we verify the asymptotical behavior
 \begin{equation}\label{eq:Un-Un-add}
 \lim_{\epsilon\to 0}\E\left[\|U_{\epsilon}^{N}-(\bar U^{N})^{\vee}\|_{\mathbb{H}_\gamma}^2\right]=0
 \end{equation}
 in the additive nose case ($g\equiv1$) with $\psi=(\psi^\wedge)^\vee.$ 
In this case, the numerical solutions $U^N_\epsilon$ and $\bar{U}^N$  satisfy
 \begin{equation}\label{eq:UNUN}
 U^N_\epsilon=S_\epsilon(t_N) \psi+\sum_{r=0}^{N-1}S_\epsilon(t_N-t_r)\delta\mathcal W_{r}
 \text{ and }
  \bar{U}^N=\bar{S}(t) \psi^\wedge+\sum_{r=0}^{N-1}\bar{S}(t_N-t_r)\delta\bar{\mathcal W}_{r},
 \end{equation}
 respectively.
Let us introduce an implementable numerical approximation for the expectation in \eqref{eq:Un-Un-add}. First,
we truncate the Karhunen--Loève  expansion \eqref{eq:KL} of $\mathcal W$ and approximate
\begin{align*}
\delta\mathcal W_{r}(\cdot)\approx\sum_{l=1}^{M_1}\widehat{\mathfrak{u}_l\mu}(\cdot)\left(\beta_l(t_{r+1})-\beta_l(t_{r})\right),\quad \delta\bar{\mathcal W}_{r}(\cdot)\approx\sum_{l=1}^{M_1}(\widehat{\mathfrak{u}_l\mu})^\wedge(\cdot)\left(\beta_l(t_{r+1})-\beta_l(t_{r})\right)
\end{align*}
in \eqref{eq:UNUN} for $r=0,1,\ldots,N-1$, where $M_1\in\mathbb{N}_+$. Under Assumption \ref{Asp:zeta}, the limiting process $\bar{Y}$ satisfies
\begin{equation}\label{eq:Y}\ud\bar Y(t)=\frac{A^\prime(\bar Y(t))}{T(\bar Y(t))}\ud t+\sqrt{\frac{2A(\bar Y(t))}{T(\bar Y(t))}}\ud \bar{\textup B}(t),\quad t>0,
\end{equation}
where $\bar{\textup B}=\{\bar{\textup B}(t)\}_{t\ge0}$ is some 1-dimensional Brownian motion independent of $\{\textup B(t)\}_{t\ge0}$ in \eqref{eq:Xt}.
\begin{figure}[!htb]
\centering
\includegraphics[width=0.8\linewidth]{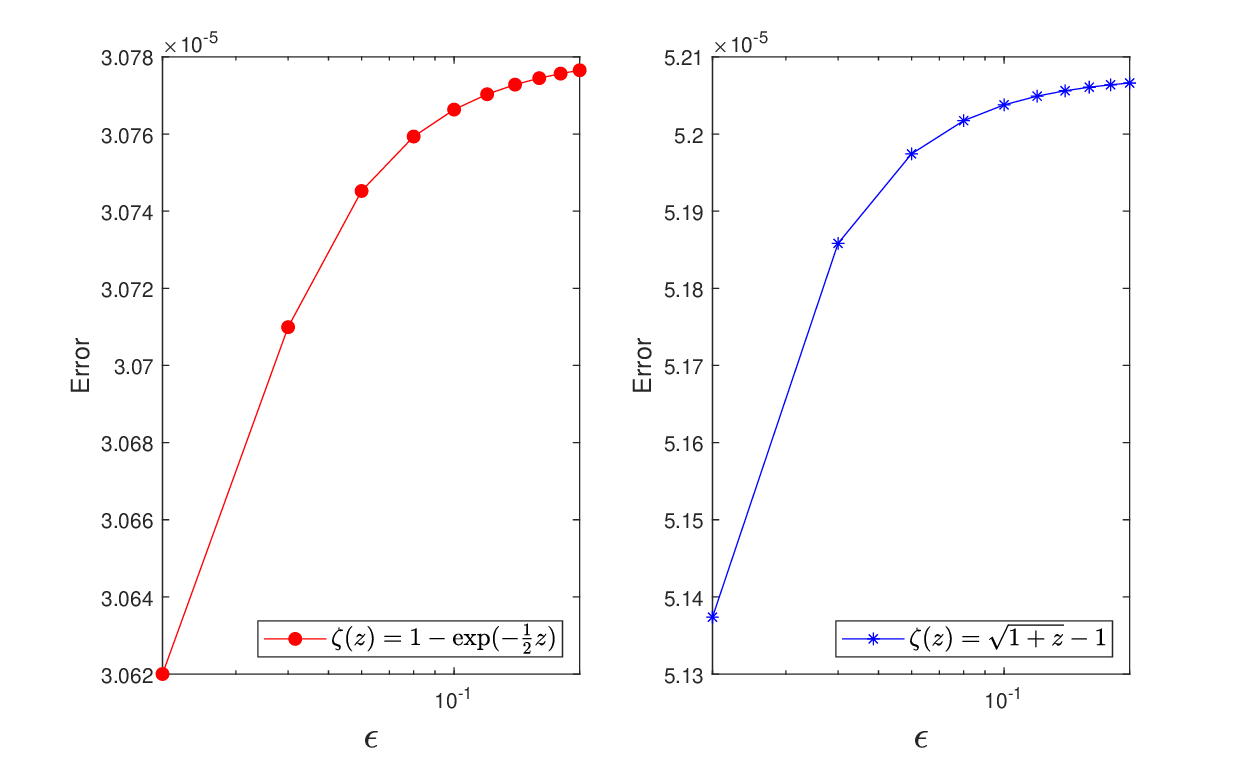}
\caption{The asymptotic error $\E[\|U_{\epsilon}^{n}-(\bar U^{n})^{\vee}\|_{\mathbb{H}_\gamma}^2]$ with $\textup{P}=5000$, $\textup{Q}=100$, $\mathsf T=2^{-13}$, $\tau=2^{-18}$, $h=2$, $M_1=1$, $M_2=5$, $\psi\equiv 0$ and $(\widehat{u_1\mu})^\wedge(z)=10\sin(z)+6z$ against with $\epsilon=0.02:0.02:0.2.$}\label{fig:4}
\end{figure}

For the realization of the semigroups, we note that
\begin{gather*}
S_\epsilon (t)\varphi(x)=\E[\varphi(X^x_\epsilon(t))]
\quad\forall~\varphi\in \mathbb{H}_\gamma\\
(\bar{S}(t)f)^{\vee}(x)=\bar{S}(t)f(H(x))=\E[f(\bar Y^{H(x)}(t))]\quad\forall~ f\in\bar{\mathbb{H}}_\gamma.
\end{gather*}
Here $\{X^x_\epsilon(t)\}_{t>0}$ denotes the solution of \eqref{eq:Xt} starting from $x\in\R^2$, and $\{\bar Y^{z}(t)\}_{t>0}$ denotes the solution of \eqref{eq:Y}
 starting from $z\in[0,\infty)$.  
In practice, 
we generate the numerical solution $\widetilde{Y}(\mathsf{T})$ of $\bar{Y}(\mathsf{T})$ via
\begin{equation*}
\widetilde{Y}(t_{n+1})=\widetilde{Y}(t_{n})+\frac{A^\prime(\widetilde{Y}(t_n))}{T(\widetilde{Y}(t_n))}\tau+\sqrt{\max\left\{\frac{2A(\widetilde{Y}(t_n))}{T(\widetilde{Y}(t_n))},0\right\}} \left(\bar{\textup B}(t_{n+1})-\bar{\textup B}(t_{n})\right),
\end{equation*}
for $n=0,1,\ldots,N$ with $\widetilde{Y}(0)=\bar{Y}(0)$.
We utilize the Euler--Maruyama numerical solution to $\widetilde{X}_\epsilon(t_{N})$ to simulate $X_\epsilon(\mathsf{T})$, i.e.,
\begin{align*}
\widetilde{X}_\epsilon(t_{n+1})=\widetilde{X}_\epsilon(t_{n})+\frac{\tau}{\epsilon} \nabla^{\perp} H(\widetilde{X}_\epsilon(t_{n}))+\textup{B}(t_{n+1})-\textup{B}(t_{n})
\end{align*}
for $n=0,1,\ldots,N-1$
with $\widetilde{X}_\epsilon(0)=X_\epsilon(0)$. 
Collecting the above arguments, we obtain
 \begin{align}\label{eq:UNij}
 U^N_\epsilon(ih,jh)&\approx\E\left[\psi(\widetilde{X}_\epsilon^{(ih,jh)}(t_{N}))\right]\\\notag
 &\quad+\sum_{r=0}^{N-1}\sum_{l=1}^{M_1}\E\left[\widehat{\mathfrak{u}_l\mu}(\widetilde{X}_\epsilon^{(ih,jh)}(t_{N}-t_r))\right]\left(\beta_l(t_{r+1})-\beta_l(t_{r})\right),
\\\label{eq:uNij}
 ( \bar{U}^N)^{\vee}(ih,jh)&\approx\E[ \psi^\wedge(\widetilde{Y}^{H(ih,jh)}(t_N))]\\\notag&\quad+\sum_{r=0}^{N-1}\sum_{l=1}^{M_1}\E\left[(\widehat{\mathfrak{u}_l\mu})^\wedge(\widetilde{Y}^{H(ih,jh)}(t_{N}-t_r))\right]\left(\beta_l(t_{r+1})-\beta_l(t_{r})\right).
 \end{align}
 In practice,
we simulate the expectations in \eqref{eq:UNij} and \eqref{eq:uNij} by the Monte--Carlo simulation with $\textup{P}$ sample paths. 
 In addition, we approximation the space integral in \eqref{eq:Un-Un-add} as follows
 \begin{equation}\label{eq:U-U}
\E\left[\|U_{\epsilon}^{N}-(\bar U^{N})^{\vee}\|_{\mathbb{H}_\gamma}^2\right]\approx\sum_{i,j=1-M_2}^{M_2-1}\E\left[\left|U_{\epsilon}^{N}(ih,jh)-\bar U^{N}(H(ih,jh))\right|^2\right]\gamma(H(ih,jh))h^2,
\end{equation}
where the expectation can be simulated by the Monte--Carlo simulation with $\textup{Q}$ sample paths, and $U_{\epsilon}^{N}(ih,jh)$ and $\bar U^{N}(H(ih,jh))$ are simulated via
 \eqref{eq:UNij} and \eqref{eq:uNij}, respectively. In this way, we obtain a formal implementable numerical approximation for the mean square error between $U^N_\epsilon$ and $\bar{U}^N$. In the present work, we do not pursue
the computational accuracy of the approximation in \eqref{eq:U-U} with respect to the parameters $M_1,M_2,\textup{P},\textup{Q}$, etc. 
The quantity on the right hand side of \eqref{eq:U-U} with respect to $\epsilon$ is displayed in Fig.\ \ref{fig:4}, where $\gamma$ is chosen to be $1$ on the truncated domain $[-M_2h,M_2h]^2$ for simplicity.
It can be observed that the error $\E[\|U_{\epsilon}^{N}-(\bar U^{N})^{\vee}\|_{\mathbb{H}_\gamma}^2]$ is small when $\epsilon$ is small,
which is consistent with Theorem \ref{theo:asy}.


\section{Conclusion}
In this work, we study an asymptotic-preserving exponential Euler approximation for the multiscale stochastic RDA equation \eqref{eq:SPDE}. There are three main ingredients in verifying the AP property of the proposed approximation. 
One is a strong convergence error bound  of the exponential Euler approximation \eqref{eq:EEM}  which depends on $\frac 1{\epsilon}$ at most polynomially. The second one is the consistency on fast advection asymptotics between exponential Euler approximations  \eqref{eq:EEM} and \eqref{eq:Un}.
The last is the regularity estimate of SPDE on graph \eqref{eq:utzk} and the strong error estimate of \eqref{eq:Un}. There are still many interesting problems to be solved, and we list some possible aspects for future work. 

For instance, in the study of AP property (see section \ref{S5}), we need a technique assumption that the Hamiltonian $H$ possesses a unique critical point. Delving into the broader context where $H$ encompasses multiple critical points is more challenging. 
Besides, to simulate \eqref{eq:utzk}  in practice, one may discretize the continuous graph $\Gamma$ into a discrete one. This may lead to an SPDE on a discrete graph with possible singularity on the vertexes.
Further study on the properties and convergence of the full discretizations of SPDEs on graph, like \eqref{eq:utzk},  are also needed.

\begin{appendices}
\section{Proof of Theorem \ref{theo:asy}}\label{Sec:A3}
The proof is inspired by that of \cite[Theorem 5.2]{CF19}. 
For the sake of simplicity, we only prove the case that $b\equiv 0$ since the proof for the case $b\neq 0$ is analogous and tedious. 
Similarly to \eqref{eq:U1-bU1}, by \eqref{eq:Un} and 
 \eqref{eq:EEM}, for any $2\le n\le N$,
\begin{align}\label{eq:UU}
U_{\epsilon}^{n}-(\bar U^{n})^{\vee}&= \underbrace{S_\epsilon(t_n-t_1)U_{\epsilon}^{1}-\bar{S}(t_n-t_1)^\vee (\bar{U}^{1})^\vee}_{=:J_{\epsilon,n,1}}\\\notag
&\quad+\underbrace{\int_{t_1}^{t_n}S_\epsilon(t_n- \lfloor s\rfloor) \big(G(U_\epsilon^{\frac{\lfloor s\rfloor}{\tau}})-G((\bar{U}^{\frac{\lfloor s\rfloor}{\tau}})^\vee)\big)\ud\mathcal{W}(s)}_{J_{\epsilon,n,2}}\\\notag
&\quad+\underbrace{\int_{t_1}^{t_n}\big(S_\epsilon(t_n- \lfloor s\rfloor) -\bar{S}(t_n- \lfloor s\rfloor)^\vee \big)G((\bar{U}^{\frac{\lfloor s\rfloor}{\tau}})^\vee)\ud\mathcal{W}(s)}_{J_{\epsilon,n,3}}.
\end{align}
For any $2\le N_1\le N$, by virtue of \eqref{eq:Seps}, for any $p\ge2$,
\begin{align*}
\E\left[\sup_{2\le n\le N_1}\|J_{\epsilon,n,1}\|_{\mathbb{H}_\gamma}^p\right]&\lesssim 
\E\left[\sup_{2\le n\le N_1}\|S_\epsilon(t_n-t_1)\left(U_{\epsilon}^{1}-(\bar{U}^{1})^\vee\right)\|_{\mathbb{H}_\gamma}^p\right]\\
&\quad +\E\left[\sup_{2\le n\le N_1}\|\left(S_\epsilon(t_n-t_1)-\bar{S}(t_n-t_1)^\vee\right) (\bar{U}^{1})^\vee\|_{\mathbb{H}_\gamma}^p\right]\\
&\lesssim 
\E\left[\|U_{\epsilon}^{1}-(\bar{U}^{1})^\vee\|_{\mathbb{H}_\gamma}^p\right] +\E\left[\sup_{s\in[\tau,\mathsf T]}\|(S_\epsilon(s)-\bar{S}(s)^\vee) (\bar{U}^{1})^\vee\|_{\mathbb{H}_\gamma}^p\right].
\end{align*}
Since by \eqref{eq:St}, \eqref{eq:Seps}, and \eqref{eq:UnHp},
$$\sup_{s\in[\tau,\mathsf T]}\|(S_\epsilon(s)-\bar{S}(s)^\vee) (\bar{U}^{1})^\vee\|_{\mathbb{H}_\gamma}^p\lesssim\|(\bar{U}^{1})^\vee\|_{\mathbb{H}_\gamma}^p \lesssim\|\bar{U}^{1}\|_{\bar{\mathbb{H}}_\gamma}^p\in L^1(\Omega),$$
one can apply \eqref{eq:Seps-S} and the dominated convergence theorem to obtain that
$$\lim_{\epsilon\to0}\E\bigg[\sup_{s\in[\tau,\mathsf T]}\|(S_\epsilon(s)-\bar{S}(s)^\vee) (\bar{U}^{1})^\vee\|_{\mathbb{H}_\gamma}^p\bigg]=0.$$
From the proof of Lemma \ref{prop:AC}, one can prove that $$\lim_{\epsilon\to0}\E\left[\|J_{\epsilon,1,1}\|_{\mathbb{H}_\gamma}^p\right]=\lim_{\epsilon\to0}\E[\|U_{\epsilon}^{1}-(\bar{U}^{1})^\vee\|_{\mathbb{H}_\gamma}^p]=0.$$ 
Consequently, it holds that
\begin{equation}\label{eq:Jn1}
\!\!\!\lim_{\epsilon\to 0} \E\left[\sup_{1\le n\le N_1}\|J_{\epsilon,n,1}\|_{\mathbb{H}_\gamma}^p\right]\!\le\lim_{\epsilon\to 0} \left\{\E\left[\|J_{\epsilon,1,1}\|_{\mathbb{H}_\gamma}^p\right]+\E\left[\sup_{2\le n\le N_1}\|J_{\epsilon,n,1}\|_{\mathbb{H}_\gamma}^p\right]\right\}=0.
\end{equation}

For any $\alpha\in(0,\frac12)$, the following elementary identity holds
$$\int_{\lfloor s\rfloor}^{t_n}(t_n-\sigma)^{\alpha-1}(\sigma-\lfloor s\rfloor)^{-\alpha}\ud \sigma=\frac{\pi}{\sin(\pi\alpha)}\quad\forall~ 0\le s<t_n\le \mathsf T.$$
By a factorization argument (cf. \cite{DP14}), for any $\alpha\in(0,\frac12)$,
\begin{align*}
J_{\epsilon,n,2}=\frac{\sin (\pi \alpha)}{\pi}\int_{t_1}^{t_n}(t_n-\sigma)^{\alpha-1}S_\epsilon(t_n-\sigma)Z^\epsilon_\alpha(\sigma)\ud \sigma,
\end{align*}
where $Z_\alpha^\epsilon(\sigma)=\int_{t_1}^{\lceil\sigma\rceil}(\sigma-\lfloor s\rfloor)^{-\alpha}S_\epsilon(\sigma-\lfloor s\rfloor)\big(G(U_\epsilon^{\frac{\lfloor s\rfloor}{\tau}})-G((\bar{U}^{\frac{\lfloor s\rfloor}{\tau}})^\vee)\big)\ud\mathcal{W}(s)$. Here $\lceil\sigma\rceil=t_{i+1}$ for $\sigma\in(t_{i},t_{i+1}]$ with $i=0,1,\ldots,N-1$.
By the H\"older inequality and \eqref{eq:St}, for any $p>1/\alpha $,
\begin{align}\label{eq:Zalpha}
\E\left[\sup_{1\le n\le N_1}\|J_{\epsilon,n.2}\|_{\mathbb{H}_\gamma}^p\right]\le C_{p,\alpha,T}\int_{t_1}^{t_{N_1}}\E\left[\|Z^\epsilon_\alpha(\sigma)\|_{\mathbb{H}_\gamma}^p\right]\ud \sigma,\quad 1\le N_1\le N.
\end{align}
The Burkholder inequality, \eqref{eq:St}, and the Lipschitz continuity of $G$ indicate
\begin{align*}
&\E\left[\|Z^\epsilon_\alpha(\sigma)\|_{\mathbb{H}_\gamma}^p\right]\\
&\lesssim\E\left[\bigg|\int_{t_1}^{\lceil\sigma\rceil}(\sigma-\lfloor s\rfloor)^{-2\alpha}\Big\|S_\epsilon(\sigma-\lfloor s\rfloor)\big(G(U_\epsilon^{\frac{\lfloor s\rfloor}{\tau}})-G((\bar{U}^{\frac{\lfloor s\rfloor}{\tau}})^\vee)\big)\Big\|^2_{\mathscr{L}_2(RK,\mathbb{H}_\gamma)}\ud s\bigg|^{\frac p2}\right]\\
&\lesssim\E\left[\bigg|\int_{t_1}^{\lceil\sigma\rceil}(\sigma-\lfloor s\rfloor)^{-2\alpha}\|U_\epsilon^{\frac{\lfloor s\rfloor}{\tau}}-(\bar{U}^{\frac{\lfloor s\rfloor}{\tau}})^\vee\|^2_{\mathbb{H}_\gamma}\ud s\bigg|^{\frac p2}\right]\\
&\lesssim\bigg|\int_{t_1}^{\lceil\sigma\rceil}(\sigma-\lfloor s\rfloor)^{-2\alpha}\|U_\epsilon^{\frac{\lfloor s\rfloor}{\tau}}-(\bar{U}^{\frac{\lfloor s\rfloor}{\tau}})^\vee\|^2_{L^p(\Omega,\mathbb{H}_\gamma)}\ud s\bigg|^{\frac p2}.
\end{align*}
Hence for any $1\le N_1\le N$,
\begin{align*}
&\int_{t_1}^{t_{N_1}}\E\left[\|Z_\alpha^\epsilon(\sigma)\|_{\mathbb{H}_\gamma}^p\right]\ud\sigma\\
&\le C\sum_{i=1}^{N_1-1}\int_{t_i}^{t_{i+1}}\bigg|\int_{t_1}^{t_{i+1}}(\sigma-\lfloor s\rfloor)^{-2\alpha}\|U_\epsilon^{\frac{\lfloor s\rfloor}{\tau}}-(\bar{U}^{\frac{\lfloor s\rfloor}{\tau}})^\vee\|^2_{L^p(\Omega,\mathbb{H}_\gamma)}\ud s\bigg|^{\frac p2}\ud\sigma\\
&\le C \sum_{i=1}^{N_1-1}\sup_{1\le l\le i}\E\Big[\|U_\epsilon^{l}-(\bar{U}^{l})^\vee\|^p_{\mathbb{H}_\gamma}\Big]\int_{t_i}^{t_{i+1}}\bigg|\int_{t_1}^{t_{i+1}}(\sigma-\lfloor s\rfloor)^{-2\alpha}\ud s\bigg|^{\frac p2}\ud\sigma\\
&\le C\tau \sum_{i=1}^{N_1-1}\sup_{1\le l\le i}\E\Big[\|U_\epsilon^{l}-(\bar{U}^{l})^\vee\|^p_{\mathbb{H}_\gamma}\Big].
\end{align*}
Collecting the above estimates gives
\begin{align}\label{eq:Jn2}
\E\left[\sup_{1\le n\le N_1}\|J_{\epsilon,n,2}\|_{\mathbb{H}_\gamma}^p\right]\le C\tau\sum_{i=1}^{N_1-1}\sup_{1\le l\le i}\E\Big[\|U_\epsilon^{l}-(\bar{U}^{l})^\vee\|^p_{\mathbb{H}_\gamma}\Big].
\end{align}
Using the factorization argument again, we rewrite
\begin{align*}
J_{\epsilon,n,3}&=\frac{\sin(\pi\alpha)}{\pi}\int_{t_1}^{t_n}(t_n-\sigma)^{\alpha-1}S_\epsilon(t_n-\sigma)Z^\epsilon_{\alpha,1}(\sigma)\ud \sigma\\
&\quad+\frac{\sin(\pi\alpha)}{\pi}\int_{t_1}^{t_n}(t_n-\sigma)^{\alpha-1}(S_\epsilon(t_n-\sigma)-\bar{S}(t_n-\sigma)^\vee)Z_{\alpha,2}(\sigma)\ud \sigma,
\end{align*}
where
\begin{align*}
Z_{\alpha,1}^\epsilon(\sigma)&:=\int^{\lceil\sigma\rceil}_{t_1}(\sigma-\lfloor s\rfloor)^{-\alpha}\underbrace{\left(S_\epsilon(\sigma- \lfloor s\rfloor)-\bar{S}(\sigma- \lfloor s\rfloor)^\vee \right)G((\bar{U}^{\frac{\lfloor s\rfloor}{\tau}})^\vee)}_{=:\delta SG(s;\sigma,\epsilon)}\ud\mathcal{W}(s),\\
Z_{\alpha,2}(\sigma)&:=\int^{\lceil\sigma\rceil}_{t_1}(\sigma-\lfloor s\rfloor)^{-\alpha}\bar{S}(\sigma- \lfloor s\rfloor)^\vee
G((\bar{U}^{\frac{\lfloor s\rfloor}{\tau}})^\vee)\ud\mathcal{W}(s).
\end{align*}
Similarly as in the proof of \eqref{eq:Zalpha}, one also has that for any $p>1/\alpha$,
\begin{align}\label{eq:Zalpha12}
\E\Big[\sup_{1\le n\le N_1}&\|J_{\epsilon,n,3}\|_{\mathbb{H}_\gamma}^p\Big]\le C_{p,\alpha,T}\int_{t_1}^{t_{N_1}}\E\left[\|Z^\epsilon_{\alpha,1}(\sigma)\|_{\mathbb{H}_\gamma}^p\right]\ud \sigma\\\notag
&+C_{p,\alpha,T}\E\left[\sup_{1\le n\le N_1}\int_{t_1}^{t_{n}}\|(S_\epsilon(t_n-\sigma)-\bar{S}(t_n-\sigma)^\vee)Z_{\alpha,2}(\sigma)\|_{\mathbb{H}_\gamma}^p\ud \sigma\right].
\end{align}
According to \eqref{eq:St}, \eqref{eq:Seps}, and \eqref{eq:ujmu0},
\begin{align*}
&\sum_{j=1}^{\infty}\int^{\lceil\sigma\rceil}_{t_1}(\sigma-\lfloor s\rfloor)^{-2\alpha}\|\delta SG(s;\sigma,\epsilon)\widehat{\mathfrak{u}_j\mu}\|_{\mathbb{H}_\gamma}^2\ud s
\lesssim\int^{\lceil\sigma\rceil}_{t_1}(\sigma-\lfloor s\rfloor)^{-2\alpha}\|G((\bar{U}^{\frac{\lfloor s\rfloor}{\tau}})^\vee)\|_{\mathbb{H}_\gamma}^2\ud s.
\end{align*}
By means of the Minkowski inequality, the linear growth of $G$ and \eqref{eq:UnHp}, 
\begin{align*}
&\int_{t_1}^{t_{N_1}}\E\bigg[\Big|\int^{\lceil\sigma\rceil}_{t_1}(\sigma-\lfloor s\rfloor)^{-2\alpha}\|G((\bar{U}^{\frac{\lfloor s\rfloor}{\tau}})^\vee)\|_{\mathbb{H}_\gamma}^2\ud s\Big|^{\frac{p}{2}}\bigg]\ud \sigma\\
&\le
\int_{t_1}^{t_{N_1}}\bigg|\int^{\lceil\sigma\rceil}_{t_1}(\sigma-\lfloor s\rfloor)^{-2\alpha}\|G((\bar{U}^{\frac{\lfloor s\rfloor}{\tau}})^\vee)\|_{L^p(\Omega,\mathbb{H}_\gamma)}^2\ud s\bigg|^{\frac{p}{2}}\ud \sigma<\infty.
\end{align*}
Therefore, we can apply the dominated convergence theorem to deduce
$$\lim_{\Xi\to\infty}\int_{t_1}^{t_{N_1}}\E\bigg[\Big|\sum_{j=\Xi+1}^{\infty}\int^{\lceil\sigma\rceil}_{t_1}(\sigma-\lfloor s\rfloor)^{-2\alpha}\|\delta SG(s;\sigma,\epsilon)\widehat{\mathfrak{u}_j\mu}\|_{\mathbb{H}_\gamma}^2\ud s\Big|^{\frac{p}{2}}\bigg]\ud \sigma=0.$$
Consequently, the Burkholder inequality yields
\begin{align*}
&\int_{t_1}^{t_{N_1}}\E\left[\|Z^\epsilon_{\alpha,1}(\sigma)\|_{\mathbb{H}_\gamma}^p\right]\ud \sigma\\
&\lesssim \int_{t_1}^{t_{N_1}}\E\bigg[\Big|\sum_{j=\Xi+1}^{\infty}\int^{\lceil\sigma\rceil}_{t_1}(\sigma-\lfloor s\rfloor)^{-2\alpha}\|\delta SG(s;\sigma,\epsilon)\widehat{\mathfrak{u}_j\mu}\|_{\mathbb{H}_\gamma}^2\ud s\Big|^{\frac{p}{2}}\bigg]\ud \sigma\\
&\quad+\int_{t_1}^{t_{N_1}}\E\bigg[\Big|\sum_{j=1}^{\Xi}\int^{\lceil\sigma\rceil}_{t_1}(\sigma-\lfloor s\rfloor)^{-2\alpha}\|\delta SG(s;\sigma,\epsilon)\widehat{\mathfrak{u}_j\mu}\|_{\mathbb{H}_\gamma}^2\ud s\Big|^{\frac{p}{2}}\bigg]\ud \sigma.
\end{align*}
Using the dominated convergence theorem again, for any fixed $\Xi\ge1$, 
\begin{align*}
\lim_{\epsilon\to 0}\int_{t_1}^{t_{N_1}}\E\bigg[\Big|\sum_{j=1}^{\Xi}\int^{\lceil\sigma\rceil}_{t_1}(\sigma-\lfloor s\rfloor)^{-2\alpha}\|\delta SG(s;\sigma,\epsilon)\widehat{\mathfrak{u}_j\mu}\|_{\mathbb{H}_\gamma}^2\ud s\Big|^{\frac{p}{2}}\bigg]\ud \sigma=0,
\end{align*}
thanks to \eqref{eq:Seps-S}. In this way, we have shown that
\begin{equation}\label{eq:Za1}\lim_{\epsilon\to 0}\int_{t_1}^{t_{N_1}}\E\left[\|Z^\epsilon_{\alpha,1}(\sigma)\|_{\mathbb{H}_\gamma}^p\right]\ud \sigma=0.
\end{equation}
Next, let $\tau_0\in(0,\tau)$ be arbitrary. For any $1\le n\le N_1$,
\begin{align*}
&\int_{t_1}^{t_{n}}\|(S_\epsilon(t_n-\sigma)-\bar{S}(t_n-\sigma)^\vee)Z_{\alpha,2}(\sigma)\|_{\mathbb{H}_\gamma}^p\ud \sigma\\
&\le\int_{t_1}^{t_{n}-\tau_0}\|(S_\epsilon(t_n-\sigma)-\bar{S}(t_n-\sigma)^\vee)Z_{\alpha,2}(\sigma)\|_{\mathbb{H}_\gamma}^p\ud \sigma\\
&\quad +\sqrt{\tau_0}\left(\int_{t_{n}-\tau_0}^{t_n}\|(S_\epsilon(t_n-\sigma)-\bar{S}(t_n-\sigma)^\vee)Z_{\alpha,2}(\sigma)\|_{\mathbb{H}_\gamma}^{2p}\ud \sigma\right)^{\frac12}\\
&\lesssim\int_{t_1}^{t_{n}-\tau_0}\sup_{s\in[\tau_0,\mathsf T]}\|(S_\epsilon(s)-\bar{S}(s)^\vee)Z_{\alpha,2}(\sigma)\|_{\mathbb{H}_\gamma}^p\ud \sigma +\sqrt{\tau_0}\left(\int_{t_{n}-\tau_0}^{t_n}\|Z_{\alpha,2}(\sigma)\|_{\mathbb{H}_\gamma}^{2p}\ud \sigma\right)^{\frac12},
\end{align*}
in virtue of \eqref{eq:St} and \eqref{eq:Seps}.
Hence, by the Cauchy--Schwarz inequality,
\begin{align*}
&\E\left[\sup_{1\le n\le N_1}\int_{t_1}^{t_{n}}\|(S_\epsilon(t_n-\sigma)-\bar{S}(t_n-\sigma)^\vee)Z_{\alpha,2}(\sigma)\|_{\mathbb{H}_\gamma}^p\ud \sigma\right]\\
&\lesssim\E\int_{t_1}^{t_{N_1}}\sup_{s\in[\tau_0,\mathsf T]}\|(S_\epsilon(s)-\bar{S}(s)^\vee)Z_{\alpha,2}(\sigma)\|_{\mathbb{H}_\gamma}^p\ud \sigma+\sqrt{\tau_0}\, \bigg(\E\int_{0}^{\mathsf T}\|Z_{\alpha,2}(\sigma)\|_{\mathbb{H}_\gamma}^{2p}\ud \sigma\bigg)^{\frac12}.
\end{align*}
Taking \eqref{eq:UnHp} and \eqref{eq:St} into account, for any $p\ge1$ and $0<\alpha<1/2$,
\begin{align*}
&\E\int_{0}^{\mathsf T}\|Z_{\alpha,2}(\sigma)\|_{\mathbb{H}_\gamma}^{2p}\ud \sigma\\
&\lesssim\int_0^{\mathsf T}\E\bigg[\Big|\int^{\lceil\sigma\rceil}_{t_1}(\sigma-\lfloor s\rfloor)^{-2\alpha}\|\bar{S}(\sigma- \lfloor s\rfloor)^\vee
G((\bar{U}^{\frac{\lfloor s\rfloor}{\tau}})^\vee)\|^2_{\mathscr{L}_2(RK,\mathbb{H}_\gamma)}\ud s\Big|^{p}\bigg]\ud \sigma\\
&\lesssim\int_0^{\mathsf T}\bigg|\int^{\lceil\sigma\rceil}_{t_1}(\sigma-\lfloor s\rfloor)^{-2\alpha}\|
G((\bar{U}^{\frac{\lfloor s\rfloor}{\tau}})^\vee)\|^2_{L^{2p}(\Omega,\mathbb{H}_\gamma)}\ud s\bigg|^{p}\ud \sigma<\infty.
\end{align*}
Utilizing the dominated convergence theorem and \eqref{eq:Seps-S}, for any $\tau_0\in(0,\tau)$,
$$
\lim_{\epsilon\to0}\E\int_{t_1}^{t_{N_1}}\sup_{s\in[\tau_0,\mathsf T]}\|(S_\epsilon(s)-\bar{S}(s)^\vee)Z_{\alpha,2}(\sigma)\|_{\mathbb{H}_\gamma}^p\ud \sigma=0.
$$
Collecting the previous three estimates, we arrive at
\begin{align*}
&\limsup_{\epsilon\to0}\E\left[\sup_{1\le n\le N_1}\int_{t_1}^{t_{n}}\|(S_\epsilon(t_n-\sigma)-\bar{S}(t_n-\sigma)^\vee)Z_{\alpha,2}(\sigma)\|_{\mathbb{H}_\gamma}^p\ud \sigma\right]\\
&\lesssim \lim_{\epsilon\to0}\E\int_{t_1}^{t_{N_1}}\sup_{s\in[\tau_0,\mathsf T]}\|(S_\epsilon(s)-\bar{S}(s)^\vee)Z_{\alpha,2}(\sigma)\|_{\mathbb{H}_\gamma}^p\ud \sigma+C\sqrt{\tau_0}\lesssim C\sqrt{\tau_0},
\end{align*}
which along with the arbitrariness of $\tau_0$ yields
\begin{align*}
&\limsup_{\epsilon\to0}\E\left[\sup_{1\le n\le N_1}\int_{t_1}^{t_{n}}\|(S_\epsilon(t_n-\sigma)-\bar{S}(t_n-\sigma)^\vee)Z_{\alpha,2}(\sigma)\|_{\mathbb{H}_\gamma}^p\ud \sigma\right]=0.
\end{align*}
Inserting the above estimate and \eqref{eq:Za1} into \eqref{eq:Zalpha12} yields
\begin{align}\label{eq:Jn3}
\lim_{\epsilon\to 0}\E\left[\sup_{1\le n\le N_1}\|J_{\epsilon,n,3}\|_{\mathbb{H}_\gamma}^p\right]=0.
\end{align}
Combining \eqref{eq:UU} and \eqref{eq:Jn2}, we deduce from the discrete Gronwall inequality that
\begin{align*}
& \E\left[\sup_{1\le n\le N_1}\|U_{\epsilon}^{n}-(\bar U^{n})^{\vee}\|_{\mathbb{H}_\gamma}^p\right]\\
&\le C e^{CT}\left( \E\left[\sup_{1\le n\le N_1}\|J_{\epsilon,n,1}\|_{\mathbb{H}_\gamma}^p\right]+\E\left[\sup_{1\le n\le N_1}\|J_{\epsilon,n,3}\|_{\mathbb{H}_\gamma}^p\right]\right),
\end{align*}
which along with \eqref{eq:Jn1} and \eqref{eq:Jn3} yields the desired result.
\qed
\section{Full discretizations for truncated problems}

\subsection{A full discretization for truncated RDA equation}\label{App:B1}
For a space stepsize $h=L/M$ with $M\in \mathbb N_+$, we denote by 
 $\{\mathbf{X}_{k(i,j)}:=(ih,jh),1-M\le i,j\le M-1\}$ the interior spatial grid points, where \begin{equation}\label{eq:kij}
 k(i,j):=(i+M)+(j+M-1)(2M-1).
 \end{equation}
 For $k=1,2,\ldots,(2M-1)^2$, we can find the unique $(i_k,j_k)\in\{1-M,\ldots,M-1\}^2$ such that $k=k(i_k,j_k)$. Indeed, by \eqref{eq:kij}, one has 
$$i_k=\mbox{mod}(k-1,2M-1)-M+1,\quad j_k=\frac{k-i_k-M}{2M-1}+1-M.$$

By using the central finite difference method, for any $k=1,\ldots,(2M-1)^2$,
\begin{align}\label{eq:Delta-u}
&\Delta \mathbf{u}_\epsilon(t,\mathbf{X}_k)=\Delta \mathbf{u}_\epsilon(t,i_kh,j_kh)\\\notag
&\approx \frac{1}{h^2}\Big(\mathbf{u}_\epsilon(t,((i_k-1)h,j_kh)-2\mathbf{u}_\epsilon(t,i_kh,j_kh)+\mathbf{u}_\epsilon(t,(i_k+1)h,j_kh)\Big)\\\notag
&\quad+\frac{1}{h^2}\Big(\mathbf{u}_\epsilon(t,i_kh,(j_k-1)h)-2\mathbf{u}_\epsilon(t,i_kh,j_kh)+\mathbf{u}_\epsilon(t,i_kh,(j_k+1)h)\Big)\\\notag
&=\frac{1}{h^2}\Big(\mathbf{u}_\epsilon(t,\mathbf{X}_{k-1})-2\mathbf{u}_\epsilon(t,\mathbf{X}_{k})+\mathbf{u}_\epsilon(t,\mathbf{X}_{k+1})\Big)\\\notag&\quad+\frac{1}{h^2}\Big(\mathbf{u}_\epsilon(t,\mathbf{X}_{k-(2M-1)})-2\mathbf{u}_\epsilon(t,\mathbf{X}_{k})+\mathbf{u}_\epsilon(t,\mathbf{X}_{k+(2M-1)})\Big),
\end{align}
where values of $\mathbf{u}_\epsilon(t,\cdot)$ at boundary points vanish.
Introduce the matrix form of the 2-dimensional Dirichlet Laplacian 
\begin{align*}
\mathbf{A}&=I\otimes\triangle_M+\triangle_M\otimes I
\in\R^{(2M-1)^2\times (2M-1)^2},
\end{align*}
where $\triangle_M\in\R^{(2M-1)\times (2M-1)}$ is the matrix form of the 1-dimensional Dirichlet Laplacian, i.e., $(\triangle_M)_{ij}=-2$ if $i=j$, $(\triangle_M)_{ij}=1$ if $|i-j|=1$, and $(\triangle_M)_{ij}=0$ if $|i-j|\ge2$. Then \eqref{eq:Delta-u} admits the following compact form
\begin{align*}
\Delta \mathbf{u}_\epsilon(t,\mathbf{X}_k) \approx\frac{1}{h^2}\sum_{l=1}^{(2M-1)^2}\mathbf{A}_{k,l}\mathbf{u}_\epsilon(t,\mathbf{X}_{l}).
\end{align*}
Similar to \eqref{eq:Delta-u}, by the central finite difference method, for any $k=1,2,\ldots,(2M-1)^2$,
\begin{align}\label{eq:Hnabla-u}\notag
\langle\nabla^{\perp} H(\mathbf{X}_k), \nabla \mathbf{u}_\epsilon(t,\mathbf{X}_k)\rangle
&=-\partial_2H(\mathbf{X}_k)\partial_1 \mathbf{u}_\epsilon(t,i_kh,j_kh)+
\partial_1H(\mathbf{X}_k)\partial_2 \mathbf{u}_\epsilon(t,i_kh,j_kh)\\\notag
&\approx-\partial_2H(\mathbf{X}_k)\frac{\mathbf{u}_\epsilon(t,(i_k+1)h,j_kh)-\mathbf{u}_\epsilon(t,(i_k-1)h,j_kh)}{2h}\\\notag
&\quad+
\partial_1H(\mathbf{X}_k)\frac{\mathbf{u}_\epsilon(t,i_kh,(j_k+1)h)-\mathbf{u}_\epsilon(t,i_kh,(j_k-1)h)}{2h}\\\notag
&=-\partial_2H(\mathbf{X}_k)\frac{\mathbf{u}_\epsilon(t,\mathbf{X}_{k+1})-\mathbf{u}_\epsilon(t,\mathbf{X}_{k-1})}{2h}\\
&\quad+
\partial_1H(\mathbf{X}_k)\frac{\mathbf{u}_\epsilon(t,\mathbf{X}_{k+(2M-1)})-\mathbf{u}_\epsilon(t,\mathbf{X}_{k-(2M-1)})}{2h},
\end{align}
where $\partial_1H(x)=2x_1+2x_1\zeta^\prime(|x|^2)$ and $\partial_2H(x)=2x_2+2x_2\zeta^\prime(|x|^2)$ for any $x=(x_1,x_2)\in\R^2$.
Introduce $\mathbf{R}_1=I\otimes \bigtriangledown_M\in\R^{(2M-1)^2\times (2M-1)^2}$
and $\mathbf{R}_2=\bigtriangledown_M\otimes I\in\R^{(2M-1)^2\times (2M-1)^2}$, where $\bigtriangledown_M\in\R^{(2M-1)\times (2M-1)}$ is a matrix form of the gradient, i.e., $(\bigtriangledown_M)_{ij}=-1$ if $i=j+1$, $(\bigtriangledown_M)_{ij}=1$ if $i=j-1$, and $(\bigtriangledown_M)_{ij}=0$ if $|i-j|\neq 1$. Moreover,
we define two matrices 
\begin{align*}
\mathbf{H}_1(\mathbf{X})&=\textup{diag}(\partial_1H(\mathbf{X}_1),\partial_1H(\mathbf{X}_2),\ldots,\partial_1H(\mathbf{X}_{(2M-1)^2}))\in\R^{(2M-1)^2\times (2M-1)^2}, \\
\mathbf{H}_2(\mathbf{X})&=\textup{diag}(\partial_2H(\mathbf{X}_1),\partial_2H(\mathbf{X}_2),\ldots,\partial_2H(\mathbf{X}_{(2M-1)^2}))\in\R^{(2M-1)^2\times (2M-1)^2}
\end{align*}
for $\mathbf{X}=(\mathbf{X}_1,\ldots,\mathbf{X}_{(2M-1)^2})\in\R^{(2M-1)^2}$. Now we can recast 
\eqref{eq:Hnabla-u} into 
$$\left\langle\nabla^{\perp} H(\mathbf{X}_k), \nabla \mathbf{u}_\epsilon(t,\mathbf{X}_k)\right\rangle\approx\sum_{l=1}^{(2M-1)^2}\left(-\frac{1}{2h}\mathbf{H}_2(\mathbf{X})\mathbf{R}_1+\frac{1}{2h}\mathbf{H}_1(\mathbf X)\mathbf{R}_2\right)_{kl}\mathbf{u}_\epsilon(t,\mathbf{X}_l).$$
In this way, we obtain that the matrix
$\mathcal L_\epsilon^M:=\frac{1}{h^2}\mathbf{A}-\frac{1}{2h\epsilon}\mathbf{H}_2(\mathbf{X})\mathbf{R}_1+\frac{1}{2h\epsilon}\mathbf{H}_1(\mathbf X)\mathbf{R}_2\in \R^{(2M-1)^2\times (2M-1)^2}$
is a formal approximation of the operator $\mathcal L_\epsilon$. 

By \eqref{eq:Lam}, for any $k,l=1,2,\ldots,(2M-1)^2$,
\begin{align*}
\E\left[\mathcal{W}(t,\mathbf{X}_k)\mathcal{W}(s,\mathbf{X}_l)\right]&=(t\wedge s)\Lambda(\mathbf{X}_k-\mathbf{X}_l)
=(t\wedge s)\mathbf{F}_{k,l},
\end{align*}
where 
$\mathbf{F}_{k,l}:=\Lambda(\mathbf{X}_k-\mathbf{X}_l)=\Lambda((i_k-i_l)h,(j_k-j_l)h).$
By denoting \begin{align*}
\mathbf{F}&=(\mathbf{F}_{k,l})_{1\le k,l\le (2M-1)^2}\in\R^{(2M-1)^2\times (2M-1)^2},\\
\mathbf{W}^{(2M-1)^2}(t)&=(\mathcal{W}(t,\mathbf{X}_1),\mathcal{W}(t,\mathbf{X}_2),\ldots,\mathcal{W}(t,\mathbf{X}_{(2M-1)^2}))^\top,
\end{align*}
we have that $\{\mathbf{B}^{(2M-1)^2}(t):=\mathbf{F}^{-\frac12}\mathbf{W}^{(2M-1)^2}(t)\}_{t\ge0}$ is a $(2M-1)^2$-dimensional standard Brownian motion.
Define a function $\mathbf{G}:\R^{(2M-1)^2}\to \R^{(2M-1)^2\times (2M-1)^2}$ by
$$\mathbf{G}(v)=\textup{diag}(g(v_1),g(v_2),\ldots,g(v_{(2M-1)^2}))\quad \text{for }v=(v_1,\ldots,v_{(2M-1)^2})\in\R^{(2M-1)^2}.$$
 Then we obtain a spatial semi-discretization for \eqref{eq:SPDE-trun} as follows
\begin{align*}
\ud \mathbf{u}_\epsilon^{M}(t)&=\mathcal L_\epsilon^M \mathbf{u}^{M}_\epsilon(t)\ud t
+\mathbf{G}(\mathbf{u}^{M}_\epsilon(t))\mathbf{F}^{\frac12}\ud \mathbf{B}^{M}(t),\quad t\in[0,\mathsf T]\\\notag
\mathbf{u}^{M}_\epsilon(0)&=(\psi(\mathbf{X}_1),\psi(\mathbf{X}_2),\ldots,\psi(\mathbf{X}_{(2M-1)^2}))^\top,
\end{align*}
where the vector
$\mathbf{u}_{\epsilon}^{M}(t):=(\mathbf{u}^{M}_\epsilon(t,\mathbf{X}_1), \mathbf{u}_\epsilon^M(t,\mathbf{X}_2),\ldots, \mathbf{u}_\epsilon^M(t,\mathbf{X}_{(2M-1)^2}))^\top$ is a formal numerical approximation of $(\mathbf{u}_\epsilon(t,\mathbf{X}_1), \mathbf{u}_\epsilon(t,\mathbf{X}_2),\ldots, \mathbf{u}_\epsilon(t,\mathbf{X}_{(2M-1)^2}))^\top$. Combining the exponential Euler approximation \eqref{eq:EEM}, we further obtain a fully discrete numerical approximation for the truncated RDA equation \eqref{eq:SPDE-trun} as follows
\begin{align*}
\mathbf{u}^{M,\tau}_{\epsilon,n+1}&=\exp(\mathcal L_\epsilon^M\tau)\mathbf{u}^{M,\tau}_{\epsilon,n}
+\exp(\mathcal L_\epsilon^M\tau)\mathbf{G}(\mathbf{u}^{M,\tau}_{\epsilon,n})\mathbf{F}^{\frac12}(\mathbf{B}^{2M-1}(t_{n+1})-\mathbf{B}^{2M-1}(t_{n})) 
\end{align*}
for every $n=0,1,\ldots, N-1$ with
$\mathbf{u}^{M,\tau}_{\epsilon,0}=\mathbf{u}^{M}_\epsilon(0)$. 

\subsection{A full discretization for truncated limiting equation
}\label{App:B2}
Let $0=z_0<z_1<z_2<\cdots<z_M=L$ be a uniform partition of the interval $[0,L]$ with the space stepsize $h=L/M$.
 Based on the spatial finite difference method, 
we discrete the truncated limiting equation \eqref{eq:utzk-p-trun} in the space direction as follows
\begin{align}\label{eq:uM1D}\notag
\partial_t\bar{u}^{M}(t,z_0)&=\beta(z_0)\frac{\bar{u}^{M}(t,z_{1})-\bar{u}^{M}(t,z_0)}{h}+g(\bar{u}^{M}(t,z_0))\partial_t\bar{\mathcal{W}}(t,z_0),\\\notag
\partial_t\bar{u}^{M}(t,z_i)&=\alpha(z_i)\frac{\bar{u}^{M}(t,z_{i+1})-2\bar{u}^{M}(t,z_{i})+\bar{u}^{M}(t,z_{i-1})}{h^2}\\\notag
&\quad+\beta(z_i)\frac{\bar{u}^{M}(t,z_{i+1})-\bar{u}^{M}(t,z_{i-1})}{2h}+g(\bar{u}^{M}(t,z_i))\partial_t\bar{\mathcal{W}}(t,z_i),\\\notag
\partial_t\bar{u}^{M}(t,z_{M-1})&=\alpha(z_{M-1})\frac{-2\bar{u}^{M}(t,z_{M-1})+\bar{u}^{M}(t,z_{M-2})}{h^2}-\beta(z_{M-1})\frac{\bar{u}^{M}(t,z_{M-2})}{2h}\\
&\quad+g(\bar{u}^{M}(t,z_{M-1}))\partial_t\bar{\mathcal{W}}(t,z_{M-1}),
\end{align}
for $i=1,2,\ldots,M-2$, where the functions $\alpha$ and $\beta$ are given by \eqref{eq:ATA}.
For $Z=(z_0,z_1,\ldots,z_{M-1})\in\R^M$, we 
define a tridiagonal matrix $\mathcal L^M(Z)\in\R^{M\times M}$ with $\mathcal L^M_{1,1}(Z)=-\frac{\beta(z_0)}{h}$, $\mathcal L^M_{1,2}(Z)=\frac{\beta(z_0)}{h}$, $\mathcal L^M_{M,M-1}(Z)=\frac{\alpha(z_{M-1})}{h^2}-\frac{\beta(z_{M-1})}{2h}$, $\mathcal L^M_{M,M}(Z)=-2\frac{\alpha(z_{M-1})}{h^2}$, and
\begin{equation*}
\mathcal L^M_{i,j}(Z)=\begin{cases}\frac{\alpha(z_{i-1})}{h^2}-\frac{\beta(z_{i-1})}{2h},&\quad \text{if } j=i-1,\\
-2\frac{\alpha(z_{i-1})}{h^2},&\quad \text{if } j=i,\\
\frac{\alpha(z_{i-1})}{h^2}+\frac{\beta(z_{i-1})}{2h},&\quad \text{if } j=i+1,
\end{cases}
\end{equation*}
 for $i=2,\ldots, M-1$.
Under Assumption \ref{Asp:zeta}, for a bounded function $\varphi:\R^2\to \R$,
 \begin{align*}
\varphi^\wedge(z)=\frac{1}{2\pi}\int_0^{2\pi}\varphi(\sqrt{F(z)}\cos \theta,\sqrt{F(z)}\sin \theta)\ud \theta,\quad z\in[0,\infty).
\end{align*}
Therefore, by \eqref{eq:Lam} and $\bar{\mathcal{W}}(t,\cdot)=\mathcal{W}(t,\cdot)^\wedge$, one has
\begin{align}\label{eq:WW}\notag
&\E\left[\bar{\mathcal{W}}(t,z)\bar{\mathcal{W}}(s,y)\right]
=\sum_{j=1}^\infty(\widehat{\mathfrak{u}_j\mu})^{\wedge}(z)(\widehat{\mathfrak{u}_j\mu})^{\wedge}(y)(t\wedge s)
=(t\wedge s)\bar{\Lambda}(z,y)
\end{align}
 for any $t, s>0$ and $z,y\in[0,\infty)$,
where the spatial covariance function $\bar{\Lambda}$ is given by 
$$\bar{\Lambda}(z,y)=\!\frac{1}{4\pi^2}\int_0^{2\pi}\int_0^{2\pi}\Lambda(\sqrt{F(z)}\cos\theta-\sqrt{F(y)}\cos\eta,\sqrt{F(z)}\sin\theta-\sqrt{F(y)}\sin\eta)\ud \theta\ud \eta
$$
for any $z,y\in(0,\infty)$. 
By introducing 
$\bar{ W}^{M}(t):=(\bar{\mathcal W}(t,z_0),\bar{\mathcal W}(t,z_1),\ldots,\bar{\mathcal W}(t,z_{M-1}))^\top$ and $Q:=(\bar{\Lambda}(z_k,z_l))_{1\le k,l\le M}\in\R^{M\times M}$,
we have that $\bar{{B}}^{M}(t)=Q^{-\frac12}{\bar{W}}^{M}(t)$
is an $M$-dimensional standard Brownian motion. As a result, by denoting $\bar{u}^M(t)=(\bar{u}^M(t,z_0),\bar{u}^M(t,z_1),\ldots, \bar{u}^M(t,z_{M-1}))^\top$, we can rewrite \eqref{eq:uM1D} into an $M$-dimensional stochastic ordinary differential equation 
\begin{align*}
\ud \bar{u}^M(t)&=\mathcal L^M(Z)\bar{u}^M(t) \ud t+\bar{G}(\bar{u}^M(t))Q^{\frac12}\ud\bar{B}^{M}(t),\quad t\in[0,\mathsf{T}],\\
\bar{u}^M(0)&=(\psi^\wedge(z_0),\psi^\wedge(z_1),\ldots, \psi^\wedge(z_{M-1}))^\top\in \R^{M}.
\end{align*}
Here $\bar{G}:\R^{M}\to \R^{M\times M}$ is defined by
$$\bar{G}(v)=\textup{diag}(g(v_1),g(v_2),\ldots,g(v_{M})), \quad\text{for } v=(v_1,\ldots,v_{M})\in\R^{M}.$$
Applying further the exponential Euler approximation yields a full discretization of the truncated problem \eqref{eq:utzk-p-trun}, which reads
$$\bar{u}^{M,\tau}_{n+1}=\exp(\mathcal L^M(Z)\tau)\bar{u}^{M,\tau}_n+\exp(\mathcal L^M(Z)\tau)\bar{G}(\bar{u}^{M,\tau}_n)Q^{\frac12}\big(\bar{B}^{M}(t_{n+1})-\bar{B}^{M}(t_{n})\big)$$
for $n=0,1,\ldots,N-1$ and $\bar{u}^M_{0}=\bar{u}^M(0)$.





\end{appendices}

\backmatter

\section*{Declarations}
\bmhead{Acknowledgments}
 The research  is partially supported by the Hong Kong Research Grant Council ECS grant 25302822, GRF grant 15302823, NSFC grant 12301526, the internal grants (P0039016, P0045336, P0046811) from Hong Kong Polytechnic University and the CAS AMSS-PolyU Joint Laboratory of Applied Mathematics.
\vskip 1em

\noindent \textbf{Data Availability} Data sharing not applicable to this article as no datasets were generated or analyzed during the current study.

\vskip 1em

\noindent \textbf{Conflict of interest} The authors declare that they have no conflict of interest.

\bibliographystyle{abbrv}
\bibliography{mybibfile}

\begin{thebibliography}{10}

\bibitem{AM13}
A.~A. Albanese and E.~M. Mangino.
\newblock One-dimensional degenerate diffusion operators.
\newblock {\em Mediterr. J. Math.}, 10(2):707--729, 2013.

\bibitem{ACQ20}
R.~Anton, D.~Cohen, and L.~Quer-Sardanyons.
\newblock A fully discrete approximation of the one-dimensional stochastic heat
  equation.
\newblock {\em IMA J. Numer. Anal.}, 40(1):247--284, 2020.

\bibitem{BT10}
R.~M. Balan and C.~A. Tudor.
\newblock Stochastic heat equation with multiplicative fractional-colored
  noise.
\newblock {\em J. Theoret. Probab.}, 23(3):834--870, 2010.

\bibitem{BMZ08}
S.~Bonaccorsi, C.~Marinelli, and G.~Ziglio.
\newblock Stochastic {F}itz{H}ugh-{N}agumo equations on networks with impulsive
  noise.
\newblock {\em Electron. J. Probab.}, 13:no. 49, 1362--1379, 2008.

\bibitem{BR22}
C.-E. Br\'{e}hier and S.~Rakotonirina-Ricquebourg.
\newblock On asymptotic preserving schemes for a class of stochastic
  differential equations in averaging and diffusion approximation regimes.
\newblock {\em Multiscale Model. Simul.}, 20(1):118--163, 2022.

\bibitem{CF17}
S.~Cerrai and M.~Freidlin.
\newblock S{PDE}s on narrow domains and on graphs: an asymptotic approach.
\newblock {\em Ann. Inst. Henri Poincar\'{e} Probab. Stat.}, 53(2):865--899,
  2017.

\bibitem{CF19}
S.~Cerrai and M.~Freidlin.
\newblock Fast flow asymptotics for stochastic incompressible viscous fluids in
  {$\Bbb{R}^2$} and {SPDE}s on graphs.
\newblock {\em Probab. Theory Related Fields}, 173(1-2):491--535, 2019.

\bibitem{CX21}
S.~Cerrai and G.~Xi.
\newblock Incompressible viscous fluids in {$\Bbb R^2$} and {SPDE}s on graphs,
  in presence of fast advection and non smooth noise.
\newblock {\em Ann. Inst. Henri Poincar\'{e} Probab. Stat.}, 57(3):1636--1664,
  2021.

\bibitem{CHLZ12}
S.-N. Chow, W.~Huang, Y.~Li, and H.~Zhou.
\newblock Fokker-{P}lanck equations for a free energy functional or {M}arkov
  process on a graph.
\newblock {\em Arch. Ration. Mech. Anal.}, 203(3):969--1008, 2012.

\bibitem{MR3926122}
S.-N. Chow, W.~Li, and H.~Zhou.
\newblock A discrete {S}chr\"{o}dinger equation via optimal transport on
  graphs.
\newblock {\em J. Funct. Anal.}, 276(8):2440--2469, 2019.

\bibitem{MR4612606}
J.~Cui, S.~Liu, and H.~Zhou.
\newblock Optimal control for stochastic nonlinear {S}chr\"{o}dinger equation
  on graph.
\newblock {\em SIAM J. Control Optim.}, 61(4):2021--2042, 2023.

\bibitem{CLZ23}
J.~Cui, S.~Liu, and H.~Zhou.
\newblock Wasserstein {H}amiltonian flow with common noise on graph.
\newblock {\em SIAM J. Appl. Math.}, 83(2):484--509, 2023.

\bibitem{DP14}
G.~Da~Prato and J.~Zabczyk.
\newblock {\em Stochastic Equations in Infinite Dimensions}.
\newblock Cambridge University Press, Cambridge, second edition, 2014.

\bibitem{ELV05}
W.~E, D.~Liu, and E.~Vanden-Eijnden.
\newblock Analysis of multiscale methods for stochastic differential equations.
\newblock {\em Comm. Pure Appl. Math.}, 58(11):1544--1585, 2005.

\bibitem{FW21}
W.-T.~L. Fan.
\newblock Stochastic {PDE}s on graphs as scaling limits of discrete interacting
  systems.
\newblock {\em Bernoulli}, 27(3):1899--1941, 2021.

\bibitem{FW12}
M.~I. Freidlin and A.~D. Wentzell.
\newblock {\em Random Perturbations of Dynamical Systems}, volume 260 of {\em
  Grundlehren der mathematischen Wissenschaften [Fundamental Principles of
  Mathematical Sciences]}.
\newblock Springer, Heidelberg, third edition, 2012.
\newblock Translated from the 1979 Russian original by Joseph Sz\"{u}cs.

\bibitem{HO10}
M.~Hochbruck and A.~Ostermann.
\newblock Exponential integrators.
\newblock {\em Acta Numer.}, 19:209--286, 2010.

\bibitem{JS99}
S.~Jin.
\newblock Efficient asymptotic-preserving ({AP}) schemes for some multiscale
  kinetic equations.
\newblock {\em SIAM J. Sci. Comput.}, 21(2):441--454, 1999.

\bibitem{JS12}
S.~Jin.
\newblock Asymptotic preserving ({AP}) schemes for multiscale kinetic and
  hyperbolic equations: a review.
\newblock {\em Riv. Math. Univ. Parma (N.S.)}, 3(2):177--216, 2012.

\bibitem{LT13}
G.~J. Lord and A.~Tambue.
\newblock Stochastic exponential integrators for the finite element
  discretization of {SPDE}s for multiplicative and additive noise.
\newblock {\em IMA J. Numer. Anal.}, 33(2):515--543, 2013.

\bibitem{MJ11}
J.~Maas.
\newblock Gradient flows of the entropy for finite {M}arkov chains.
\newblock {\em J. Funct. Anal.}, 261(8):2250--2292, 2011.

\bibitem{MA11}
A.~Mielke.
\newblock A gradient structure for reaction-diffusion systems and for
  energy-drift-diffusion systems.
\newblock {\em Nonlinearity}, 24(4):1329--1346, 2011.

\bibitem{PZ97}
S.~Peszat and J.~Zabczyk.
\newblock Stochastic evolution equations with a spatially homogeneous {W}iener
  process.
\newblock {\em Stochastic Process. Appl.}, 72(2):187--204, 1997.

\end{thebibliography}

\end{document}